\documentclass[12pt]{amsart}
\usepackage{amsmath, amsthm, amssymb, geometry, url,graphicx,enumitem}
\usepackage{color}
\usepackage{overpic}
\usepackage[dvipsnames]{xcolor}

\evensidemargin \oddsidemargin
\author{Christian Millichap}
\address{Department of Mathematics\\ 
Furman University\\ 
Greenville, SC 29613}
\email{christian.millichap@furman.edu}

\author{Rolland Trapp}
\address{Department of Mathematics\\ 
California State University\\ 
San Bernardino, CA 92407}
\email{rtrapp@csusb.edu}

\title{Flat fully augmented links are determined by their complements}

\usepackage{fancyhdr}

\fancypagestyle{plain}

\DeclareMathAlphabet{\curly}{U}{rsfs}{m}{n}

\newtheorem{thm}{Theorem}[section]
\newtheorem{cor}[thm]{Corollary}

\newtheorem{prop}[thm]{Proposition}
\newtheorem{lem}[thm]{Lemma}
\theoremstyle{definition}
\newtheorem{df}[thm]{Definition}
\newtheorem*{rmk}{Remark}

\theoremstyle{remark}

\newtheorem*{namedtheorem}{\theoremname}
\newcommand{\theoremname}{testing}
\newenvironment{named}[1]{\renewcommand{\theoremname}{#1}\begin{namedtheorem}}{\end{namedtheorem}}

\def\1{\mathbf{1}}

\theoremstyle{plain}

\theoremstyle{remark}

\newtheorem*{ques}{Question}

\theoremstyle{plain}

\def\1{\mathbf{1}}

\makeatletter
\def\moverlay{\mathpalette\mov@rlay}
\def\mov@rlay#1#2{\leavevmode\vtop{%
   \baselineskip\z@skip \lineskiplimit-\maxdimen
   \ialign{\hfil$\m@th#1##$\hfil\cr#2\crcr}}}
\newcommand{\charfusion}[3][\mathord]{
    #1{\ifx#1\mathop\vphantom{#2}\fi
        \mathpalette\mov@rlay{#2\cr#3}
      }
    \ifx#1\mathop\expandafter\displaylimits\fi}
\makeatother

\makeatletter
\let\@@pmod\pmod
\DeclareRobustCommand{\pmod}{\@ifstar\@pmods\@@pmod}
\def\@pmods#1{\mkern4mu({\operator@font mod}\mkern 6mu#1)}
\makeatother

\begin{document}

\begin{abstract}
In this paper, we show that two flat fully augmented links with homeomorphic complements must be equivalent as links in $\mathbb{S}^{3}$. This requires a careful analysis of how  totally geodesic surfaces and cusps intersect in these link complements and behave under homeomorphism. One consequence of this analysis is a complete classification of flat fully augmented link complements that admit multiple reflection surfaces. In addition, our work classifies those symmetries of flat fully augmented link complements which are not induced by symmetries of the corresponding link.
\end{abstract}

\maketitle 
\let\thefootnote\relax\footnote{\textit{Date: \today} \hfill }


\section{Introduction}
\label{sec:intro}

Two links, $L_1$ and $L_2$, in $\mathbb{S}^{3}$ are  \textbf{equivalent} if there exists an orientation-preserving homeomorphism of pairs from $(\mathbb{S}^{3}, L_1)$ to $(\mathbb{S}^{3}, L_2)$. An equivalence of links induces a homeomorphism between the link complements $\mathbb{S}^{3} \setminus L_1$ and $\mathbb{S}^{3} \setminus L_2$, which shows that links determine their complements. However, the converse of this statement is generally not true. For instance, if a link contains an unknotted component, then a Dehn twist along this component is a homeomorphism of the complement that will frequently produce a non-equivalent link. Whitehead used this technique in \cite{Wh1937} to show that that there are infinitely many distinct links with the same complement as the Whitehead link. Some other known constructions for producing distinct links with the same complement can be found in Berge \cite{Berge} and Gordon \cite[Section 6]{Go2002}. In contrast to links,  the Gordon--Luecke Theorem \cite{GL1989} shows that knots are determined by their complements in $\mathbb{S}^{3}$. Knots in certain closed, oriented $3$-manifolds other than $\mathbb{S}^{3}$ are also determined by their complements; see Rong \cite{Ro1993}, Matignon \cite{Ma2010}, Gainullin \cite{Ga2018}, and Ichihara--Saito \cite{IS2021} for some examples.  Moving forward, we will always assume  knots and links are embedded in $\mathbb{S}^{3}$. This contrast between links with multiple components and knots motivated the following question raised by Mangum and Stanford \cite{MaSt2001}:

\begin{ques} Is there a set of links $\mathcal{S}$ such that if $L_1, L_2 \in \mathcal{S}$ and $\mathbb{S}^{3} \setminus L_1$ is homeomorphic to $\mathbb{S}^{3} \setminus L_2$, then $L_1$ is equivalent to $L_2$? \end{ques}

In the same paper, Mangum--Stanford show that the set of homologically trivial and Brunnian links (called HTB links in their paper) provide an affirmative answer to this question \cite[Theorem 3]{MaSt2001}.  As far as the authors know this is the only example in the literature of an infinite set of links, other than the set of knots, with this property.  This motivates the main goal of our paper, which is to show that the family of \textbf{flat fully augmented links} (flat FALs) also have this property.

Flat FALs are a family of hyperbolic links which can be obtained from link diagrams, meeting certain diagrammatic conditions, in the following way.  Given a link diagram add a trivial component (called a crossing circle) enclosing each twist region, then remove all twists so the strands of the diagram run parallel through the crossing circles (see Figure \ref{fig:intropic}). Knot circles of the resulting flat FAL are components that lie in the projection plane of the resulting diagram. Flat FALs, and more generally, FALs, are a rich family of hyperbolic links, which have received much attention in the last 20 years due to the explicit combinatorial descriptions of their geometric structures and connections to highly twisted knots via Dehn surgery; see \cite{BFT2015}, \cite{CDW2012}, \cite{F2017}, \cite{FP2007}, \cite{HMW2020}, \cite{KT2020}, \cite{MMR2020}, \cite{mrstz}, \cite{P2011}, \cite{P2008}, \cite{P2007} for some examples from the literature. We direct the reader to the beginning of Section \ref{sec:FALs} for more details on essential properties of flat FALs used in this paper. We can now formally state our main theorem, recalling that an isotopy of links in $\mathbb{S}^{3}$ induces an equivalence of links.
 
 \begin{thm}\label{MAINTHEOREM}
 Let $\mathcal{A}$ and $\mathcal{A}'$ be flat FALs.  Then $(\mathbb{S}^{3}, \mathcal{A})$ is isotopic to $(\mathbb{S}^{3}, \mathcal{A}')$ if and only if $\mathbb{S}^{3} \setminus \mathcal{A}$ is homeomorphic to $\mathbb{S}^{3} \setminus \mathcal{A}'$. 
 \end{thm}

Theorem \ref{MAINTHEOREM} provides a new infinite set of links  that are determined by their complements,  distinct from knots and  HTB links. By definition, every flat FAL contains at least three components, and so, no flat FALs are knots. At the same time, there are infinitely many flat FALs that are not HTB links. Specifically, any flat FAL with at least two knot circles is not an HTB link since it will contain a Hopf sublink.  Such a sublink violates the homologically trivial property that the linking number is $0$ for any two components. 

Another way our work differs from both Gordon--Luecke and Mangum--Stanford is in the techniques used.  The proof of Theorem \ref{MAINTHEOREM} greatly leverages the hyperbolic structure of flat FAL complements and relies on analyzing the behavior of totally geodesic surfaces and cusps under isometries induced by homeomorphism. This geometric approach differs from the purely topological approaches used by Gordon--Luecke and Magnum--Stanford, where these authors determine which Dehn surgeries on a knot or HTB link produce $\mathbb{S}^{3}$. 

We note that the \emph{flat} hypothesis of Theorem \ref{MAINTHEOREM} is necessary, making the result as general as possible within the class of FALs.  To see this, consider the twisted FALs of Figure \ref{fig:FlatReq}.  They differ by a Dehn twist on the top circle, so their complements are homeomorphic.  The links are distinct, however, since all components are unknots in Figure \ref{fig:FlatReq}$(a)$ while one component is a trefoil in Figure \ref{fig:FlatReq}$(b)$.

\begin{center}
\begin{figure}[h]
\[
\begin{array}{ccc}
\includegraphics{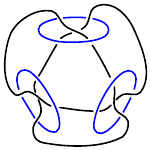}&\hspace{.4in}&\includegraphics{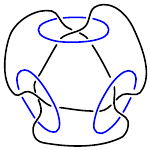}\\
(a)\textrm{ All components unknotted} & &(b)\textrm{ One trefoil component}
\end{array}
\]
\caption{Distinct twisted FALs with homeomorphic complements}
\label{fig:FlatReq}
\end{figure}
\end{center}

The proof of Theorem \ref{MAINTHEOREM} breaks into two cases, depending on the number of reflection surfaces contained in a flat FAL complement. Intuitively, a reflection surface for a flat FAL complement $\mathbb{S}^{3} \setminus \mathcal{A}$ is a totally geodesic surface that corresponds with the projection plane for an FAL diagram of $\mathcal{A}$; see Section \ref{sec:FALs} for more explicit details and properties of reflection surfaces. As an intermediary step in the proof of Theorem \ref{MAINTHEOREM}, we classify flat FAL complements with multiple reflection surfaces and show that homeomorphisms between flat FAL complements preserve reflection surfaces. For emphasis, we now state the classification of flat FALs with multiple reflection surfaces (see Figure \ref{fig:BP4O4} and the beginning of Section \ref{sec:2RS} for descriptions of the links $P_n$ and $O_n$).

\begin{thm}\label{thm:MainTheorem2}
Suppose $M = \mathbb{S}^{3} \setminus \mathcal{A}$ is a flat FAL complement with multiple distinct reflection surfaces. Then either 
\begin{itemize}
\item $\mathcal{A}$ is equivalent to the Borromean rings and $M$ contains exactly three distinct reflection surfaces, or
\item $\mathcal{A}$ is equivalent to $P_n$ with $n \geq 3$, or $O_n$ with $n \geq 2$, and $M$ contains exactly two distinct reflection surfaces. 
\end{itemize}
\end{thm}
 
A slightly more general version of this classification result is given in Theorem \ref{thm:MultipleRS}, along with several useful corollaries on how reflection surfaces behave under homeomorphism. 

The proof of Theorem \ref{thm:MainTheorem2}, relies on an analysis of how cusps and totally geodesic surfaces behave relative to different reflection surfaces in  a flat FAL complement. With Theorem \ref{thm:MainTheorem2} in hand, basic properties of the links $P_n$ and $O_n$ show that two flat FALs whose complements admit multiple reflection surfaces are equivalent as links if and only if their respective complements are homeomorphic. This is highlighted in Corollary \ref{cor:MRSdetermined}. 

Theorem \ref{thm:MainTheorem2} also implies that a flat FAL with multiple reflection surfaces cannot be homeomorphic to a flat FAL with a single reflection surface; see Corollary \ref{thm:uniquereflectionhomeo}. As a result,  we now only need to consider the case where there exists a homeomorphism $h:M \rightarrow M'$ between two flat FAL complements, each with unique reflection surfaces. This is a far more challenging task, and  relies  on the topology and geometry of thrice-punctured spheres in a flat FAL complement that are not contained in the reflection surface, which were classified by Morgan--Mork--Ransom--Spyropoulos--Trapp--Ziegler in \cite{mrstz}. In particular, we  make extensive use of sets of thrice-punctured spheres that separate $M$, whose homeomorphic images in $M'$ are greatly restricted by the  classification in \cite{mrstz}. These restrictions help us describe how the knot and crossing circles which intersect such thrice-punctured spheres must behave under homeomorphism. 

This, in turn, allows us to show that any flat FAL complement that admits a unique reflection surface and a non-trivial homeomorphism  must have a particular link structure, which we call a signature link; see Definition \ref{defn:SigLink}. Furthermore,  only one type of non-trivial homeomorphism is possible in this case, which we call a full-swap. We  carefully describe   full-swap homeomorphisms of signature links in terms of compositions of Dehn twists along sets of Hopf sub-links; see Definition \ref{defn:MLSwap}. In addition, the induced action of any such full-swap homeomorphism on the corresponding flat FAL (which must be a signature link)  can be made explicit on a diagrammatic level, which makes it easy to construct a specific isotopy between any such pair of flat FALs with homeomorphic complements. 

As partially noted in the previous paragraph, our work not only shows that flat FALs are determined by their complements, but classifies the types of homeomorphisms that can exist between flat FAL complements and what types of flat FAL complements can admit certain types of homeomorphisms. These severe geometric restrictions on self-homeomorphisms of flat FAL complements lead to a concise comparison between their symmetry groups and those of their complements. Recall that the symmetry group of a link $L \subset \mathbb{S}^{3}$ is the group of homeomorphisms of pairs $(\mathbb{S}^{3}, L)$ up to isotopy, which we denote by $Sym(\mathbb{S}^{3}, L)$. Similarly, the symmetry group of the corresponding link complement  is the group of self-homeomorphisms of $\mathbb{S}^{3} \setminus L$ up to isotopy, denoted by $Sym(\mathbb{S}^{3} \setminus L)$. Note that, any self-homeomorphism of $(\mathbb{S}^{3}, L)$ induces a self-homeomorphism of $\mathbb{S}^{3} \setminus L$, and so, $Sym(\mathbb{S}^{3}, L) \subseteq Sym(\mathbb{S}^{3} \setminus L)$. However, this can be a strict containment; see \cite{HeWe1992} for some examples. The following theorem  classifies symmetries of flat FAL complements that are not induced by symmetries of the corresponding link.

\begin{thm}\label{thm:SymmetryThm}
Let $\mathcal{A}$ be a flat FAL.  Then either
\begin{itemize}
\item $\mathcal{A}$ is not a signature link and both $\mathcal{A}$ and its complement $M = \mathbb{S}^3\setminus \mathcal{A}$ have the same symmetry group, or
\item $\mathcal{A}$ is a signature link and full-swaps on $\mathcal{A}$ generate symmetries of $M = \mathbb{S}^3\setminus \mathcal{A}$ which are not restrictions of symmetries of $\mathcal{A}$ to $M$.
\end{itemize}
\end{thm}

Part of Theorem \ref{thm:SymmetryThm} is proved in Theorem \ref{thm:SymmetryMRS} at the end of Section \ref{sec:MRS}. The rest of the proof of this theorem is completed at the end of Section \ref{sec:CDFF}.



We now describe the organization of the rest of this paper. In Section \ref{sec:FALs} we introduce flat FALs, the necessary terminology related to cusps and totally geodesic surfaces contained in flat FAL complements, and review some essential facts from the literature on the geometry of FAL complements. In Section \ref{sec:PropsRefSur}, we prove Theorem \ref{thm:MainTheorem2}. This allows us to focus the rest of the paper on homeomorphisms between flat FAL complements, each with a unique reflection surface and this  transition is discussed in Section \ref{sec:Transition}. Then Section \ref{Sec:SigLinks} discusses ``signature link'' complements, a special class of flat FAL complements, along with a ``full-swap'' homeomorphism that can be performed on any signature link complement. The image of a full-swap homeomorphism is another signature link complement, and it is straightforward to construct an explicit isotopy between their corresponding links.  In Section \ref{subsec:SepSets}, we prove some useful facts about sets of thrice-punctured spheres that separate a flat FAL complement (with a unique reflection surface) and how they behave under homeomorphisms. Finally, in Section \ref{sec:CDFF} we build off the tools from Section  \ref{subsec:SepSets} to show that any homeomorphism between flat FAL complements (each with a unique reflection surface) must essentially be a ``full-swap'' homeomorphism between ``signature link'' complements.  This allows us to construct an isotopy between the corresponding links and prove our main theorem. 

We would like to thank Jeffrey Meyer for comments and suggestions on preliminary work done for this paper.


\section{Totally Geodesic Surfaces and Cusps in Flat FALs}
\label{sec:FALs}

This section provides a brief introduction to flat FALs, compiles some known results about totally geodesic surfaces in their complements, and introduces two concepts used extensively in Section \ref{sec:PropsRefSur}: reflection-like surfaces and their induced structures on a flat FAL complement.  Aside from reflection-like surfaces and their induced structures, this section is a review of results in the literature.
 
 We first describe how to construct a flat FAL. For this process,  start with a link $L$ and a diagram $D(L)$. 
 We can build a diagram $D(\mathcal{F})$ for an FAL $\mathcal{F}$ corresponding to $L$ by augmenting each twist region in $D(L)$ with a circle and undoing all full-twists from each respective twist region. After this procedure, twist regions of $D(L)$ that had contained an odd number of crossings will still contain a single crossing in $D(\mathcal{F})$. If we remove all of these remaining single crossings, then we will have constructed a diagram $D(\mathcal{A})$ for a flat FAL $\mathcal{A}$, as illustrated in Figure \ref{fig:intropic}.

In our work, we will solely be interested in the case where a flat FAL $\mathcal{A}$ is hyperbolic in the sense that the complement $\mathbb{S}^{3} \setminus \mathcal{A}$ admits a complete metric of constant negative curvature. As noted in \cite[Theorem 2.5]{P2011}, a (flat) FAL $\mathcal{A}$ is hyperbolic if and only if it there exists a corresponding link diagram $D(L)$ that is non-splittable, prime, twist-reduced, and contains atleast two twist regions, where $D(\mathcal{A})$ is obtained from $D(L)$ by the augmentation process described above. We refer the reader to \cite[Section 1]{FP2007} for these diagrammatic definitions. Moving forward, we will assume any flat FAL is hyperbolic, and refer to $D(L)$ as the corresponding diagram that was augmented to construct $D(\mathcal{A})$.  
 
This augmentation process  partitions the components of $\mathcal{A}$ into \textbf{crossing circles}, the trivial components added via augmentation, and \textbf{knot circles}, components coming from the original link $L$. Observe that each crossing circle in a flat FAL bounds a \textbf{crossing disk}, a disk in $\mathbb{S}^3$ punctured twice by parallel knot circle arcs that replace an original twist region of $D(L)$.  We define a flat FAL diagram to be a link diagram together with this additional structure.  More precisely, we have

\begin{df}\label{defn:FALDiagram}
A link diagram $D(\mathcal{A})$ is a \textbf{flat FAL diagram} for a (hyperbolic) link $\mathcal{A}$ if $D(\mathcal{A})$ was constructed by fully augmenting a corresponding diagram $D(L)$ and removing all  crossings from each twist region.  In addition, we assume $D(\mathcal{A})$ carries with it the partition of components of $\mathcal{A}$ into crossing- and knot-circles, as well as the choice of crossing disks, relative to this augmentation of $D(L)$.  A link $\mathcal{A}$ is a \textbf{flat FAL} if it admits a flat FAL diagram $D(\mathcal{A})$.
\end{df}
 
 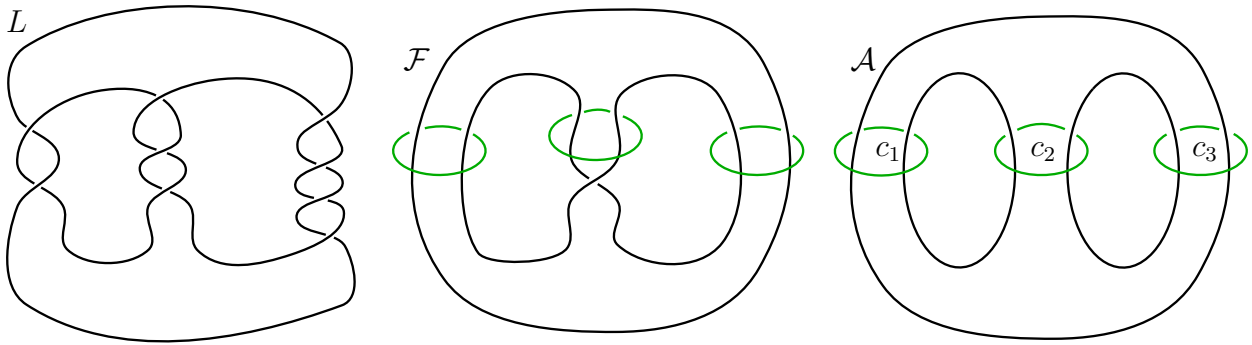
\begin{figure}[ht]
	\centering
	\begin{overpic}[width = \textwidth]{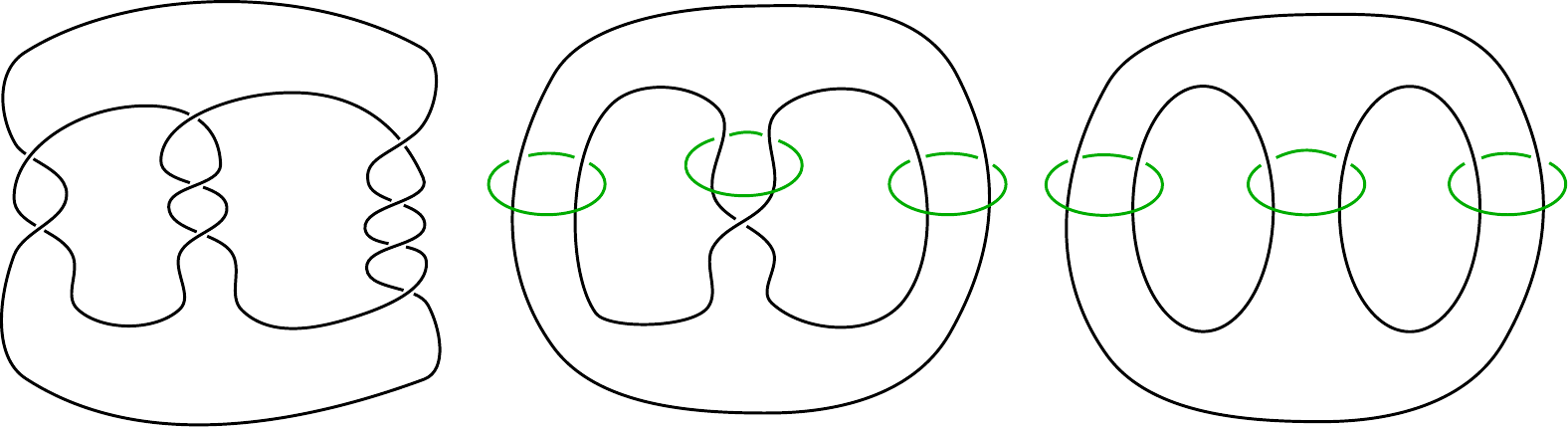}
		\put(0,25){$L$}
		\put(32,22){$\mathcal{F}$}
		\put(68,22){$\mathcal{A}$}
		\put(70,15){$c_{1}$}
		\put(82.5,15){$c_{2}$}
		\put(95.5,15){$c_{3}$}		
	\end{overpic}
	\caption{On the left is a diagram of a link $L$ with three twist regions. The middle diagram shows the corresponding FAL $\mathcal{F}$ obtained from fully augmenting $L$. The right diagram shows the corresponding flat FAL $\mathcal{A}$.   Crossing circles of $\mathcal{A}$ are labeled by $c_i$, for $i=1,2,3$.}
	\label{fig:intropic}
\end{figure}

The geometric structures of  augmented links were first studied in Adams \cite{ad2}. A particularly nice geometric decomposition of (flat) FAL complements into pairs of identical right-angled ideal polyhedra was  described by Agol--Thurston in the appendix of \cite{L2004}. This geometric decomposition has proved to be a fruitful tool for analyzing geometric and topological properties of FALs; see \cite{CDW2012}, \cite{F2017}, \cite{FP2007}, \cite{HMW2020}, \cite{KT2020}, and \cite{P2011} for a few examples.  In addition, infinite subclasses of FALs  have also been examined in the literature. For instance, Meyer--Millichap--Trapp \cite{MMR2020} studied the arithmeticity, invariant trace fields, symmetries, and hidden symmetries of FALs obtained by fully augmenting pretzel links, and Purcell examined FALs whose complements admit a decomposition into regular ideal hyperbolic octahedra in \cite{P2011}. In both cases, each subclass contains an infinite number of flat FALs. Thus the family of flat FALs is a large set of links  with many interesting properties.

 We now collect some important facts about totally geodesic surfaces and cusps contained inside a flat FAL complement. These properties will be essential for proving both Theorems \ref{MAINTHEOREM} and \ref{thm:MainTheorem2}. Most of the results stated here are known in the literature and we refer the reader to \cite{FP2007}, \cite{mrstz}, and \cite{P2011}  for more details on the geometric properties of FAL complements discussed here.

First, we describe  reflection surfaces, which are an important type of totally geodesic surface contained in every flat FAL complement. Given a flat FAL  $\mathcal{A}$ in $\mathbb{S}^{3} \cong \mathbb{R}^{3} \cup \{ \infty \}$, position the crossing circles and disks so that they are orthogonal to the projection plane and embed the knot circles in the projection plane. Such an isotopy is always possible based on the diagrammatic definition of a flat FAL. Then reflection in the projection plane maps every component of $\mathcal{A}$ to itself and fixes the projection plane point-wise. Mostow--Prasad rigidity implies that this projection plane corresponds with a totally geodesic surface $R \subset M = \mathbb{S}^{3} \setminus \mathcal{A}$ and there exists an orientation-reversing involution $\iota_{R}:M \rightarrow M$ that fixes $R$ point-wise corresponding with reflection in the projection plane. Since  $\mathcal{A}$ could admit FAL diagrams where different surfaces play the role of the projection plane, we provide the following definition. 

\begin{df}
Let $M = \mathbb{S}^{3} \setminus \mathcal{A}$ be a flat FAL complement and let $R \subset M$ be an embedded totally geodesic surface. If there exists an FAL diagram $D(\mathcal{A})$ in the projection plane $P$ for which $R = P \setminus \mathcal{A}$, then we say that $R$ is a \textbf{reflection surface} of $M$ (relative to the diagram $D(\mathcal{A})$).
\end{df}

 An FAL diagram  $D(\mathcal{A})$ partitions the cusps of $M = \mathbb{S}^{3} \setminus \mathcal{A}$ into \textbf{crossing circle cusps} and  \textbf{knot circle cusps}, corresponding with crossing circles and knot circles of $\mathcal{A}$, respectively.  We also say that this partition is  induced by a reflection surface $R$, though this ultimately depends on the FAL diagram since reflection surfaces are defined relative to the projection plane for a diagram. As we will see in Section \ref{sec:PropsRefSur}, it is possible for a flat FAL complement to admit distinct reflection surfaces that induce different partitions on the components of $\mathcal{A}$.  At the same time,  it is possible for a flat FAL complement to admit one reflection surface that induces different partitions on the components of $\mathcal{A}$; this phenomenon will be examined in Section \ref{Sec:SigLinks}. Here, we say a cusp of $M$ is an \textbf{$R$-knot circle cusp} (respectively, \textbf{$R$-crossing circle cusp}) if the corresponding link component of $\mathcal{A}$ is a knot circle (resp. crossing circle) relative to $R$. In this paper, we frequently use the same notation to refer to a component of $\mathcal{A}$ and the corresponding cusp on $M$. In addition, each $R$-crossing circle $C$ of $\mathcal{A}$ bounds an $R$-\textbf{crossing disk} $D$, which is a disk twice punctured by the two (not necessarily distinct) knot circles going through $C$. As noted in \cite[Lemma 2.1]{P2011}, each such $R$-crossing disk is an embedded totally geodesic thrice-punctured sphere in $M$. Furthermore, $R$ intersects $D$ orthogonally in the set of three simple, non-separating geodesics on $D$; see Figure \ref{fig:3PS} for a visual of simple geodesics on a thrice-punctured sphere. 

A flat FAL diagram $D(\mathcal{A})$ (or, alternatively, its reflection surface $R$) determines more structure than the partitioning of cusps of $M$ into knot and crossing circle cusps.  It also determines a meridian and longitude on each boundary torus  of a cusp neighborhood.  In this context, meridians and longitudes will be considered \textbf{slopes}, or unoriented isotopy classes of simple closed curves on this boundary torus. 

Each component $L$ of a flat FAL is topologically an unknot in $\mathbb{S}^3$, so the torus boundary $T_L$ of a tubular neighborhood $V_L$ of $L$ has a natural choice of meridional and longitudinal slopes.  The unknotted torus $T_L\subset \mathbb{S}^3$ bounds a solid torus on each side. A meridian $m$ is the slope of $T_L$ that bounds a disk in $V_L$, while a longitude $\ell$ is a slope of $T_L$ that bounds a disk in $\mathbb{S}^3\setminus V_L$.  The unoriented curves $m$ and $\ell$ are well-defined up to ambient isotopy on $T_L$ , so give well-defined slopes on $T_L$.

From a geometric perspective, the slopes $m$ and $\ell$ can be described as the intersection of $T_L$ with totally geodesic surfaces.  More precisely, let $\mathcal{A}$ be a flat FAL with reflection surface $R$.  If $K$ is a knot circle of $\mathcal{A}$ with torus boundary $T_K$ of a  cusp neighborhood of $K$, then $R\cap T_K$ is a pair of simple closed curves that represent longitudes on $T_K$.  Similarly, for any crossing circle $C$ of $\mathcal{A}$, the set $R\cap T_C$ is a pair of simple closed curves that represent meridians on $T_C$.  Crossing disks relative to the reflection surface $R$ are orthogonal to it, and so intersect $T_K$ in meridians and $T_C$ in longitudes.  See \cite[Lemma 2.3]{FP2007} for more details.

The longitudes and meridians just described will be referred to as \textbf{$R$-meridians} and \textbf{$R$-longitudes} when we want to emphasize the FAL diagram (and corresponding reflection surface) used to determine them. Since each torus boundary of a cusp $T$ of $M$ is rectangular (see \cite[Lemma 2.3]{FP2007}),  curves orthogonal to $R\cap T$ determine the second generator for the fundamental group of each  such torus.  We refer to the basis for each $\pi_{1}(T)$ described above as the \textbf{peripheral structure} relative to the FAL diagram $D(\mathcal{A})$.



The following proposition summarizes the essential features of a reflection surface in a flat FAL complement that we just discussed. 

\begin{prop}
\label{prop:RSfeatures}
Let $M = \mathbb{S}^{3} \setminus \mathcal{A}$ be a flat FAL complement with an FAL diagram $D(\mathcal{A})$ which partitions the components of $\mathcal{A}$ into crossing circles and knot circles. Then $M$ contains an embedded totally geodesic $R$ with the following features:
\begin{itemize}
\item $R$ corresponds with $P \setminus \mathcal{A}$, where $P$ is the projection plane for $D(\mathcal{A})$, 
\item $R$ is fixed point-wise by an orientation reversing involution $i_{R}: M \rightarrow M$,
\item $R$ intersects each $R$-crossing disk of $M$ orthogonally in its three simple non-separating geodesics, 
\item $R$ intersects every cusp of $M$ in two parallel curves,  
\item $R$ intersects the boundary torus $T_C$ of each crossing circle cusp $C$ in a pair of meridians and provides a peripheral structure on $\pi_{1}(T_C)$ relative to $D(\mathcal{A})$, and
\item $R$ intersects the boundary torus $T_K$ of each knot circle cusp $K$ in a pair of longitudes and provides a peripheral structure on $\pi_{1}(T_K)$ relative to $D(\mathcal{A})$. 
\end{itemize}
\end{prop}

Our main goal is to show that flat FALs are determined by their complements among the set of all flat FALs.  For this reason we will need to consider homeomorphic images of reflection surfaces and the structures associated with them that are highlighted in the previous proposition.  By Mostow--Prasad rigidty, any homeomorphism between flat FALs induces an isometry between them, and so, we can restrict our analysis to the isometric image of a reflection surface. 

Let $M, M'$ be flat FAL complements with reflection surfaces $R, R'$, respectively, and let $\rho :M'  \rightarrow M$ be an isometry between these hyperbolic $3$-manifolds. Since the definition of a reflection surface is diagram dependent, we can not immediately assume that $\rho(R')$ is a reflection surface for $M$.  This motivates the following definition.

\begin{df}\label{defn:Rlike}
We say that $S$ is a \textbf{reflection-like} surface for a flat FAL complement $M$ if there exists an isometry $\rho:M'\rightarrow M$ between flat FAL complements such that $S = \rho(R')$, where $R'$ is a reflection surface for $M'$. 
\end{df}

Note that, any reflection surface for $M$ is also a reflection-like surface for $M$ via the identity map. In addition,  any homeomorphic flat FAL complements have the same number of reflection-like surfaces.  Indeed, a homeomorphism $h:M' \rightarrow M$ between flat FAL complements $M'$ and $M$ induces a unique isometry $\rho_h:M'  \rightarrow M$.  By Definition \ref{defn:Rlike}, the image of each reflection-like surface in $M'$ is a reflection-like surface in $M$.  At the same time, $\rho_{h}^{-1}$ preserves the property of being reflection-like, so $M'$ and $M$ have the same number of reflection-like surfaces. Similarly, equivalent flat FALs have the same number of reflection-like surfaces since an equivalence of links induces a homeomorphism between their respective complements.

Many important topological and geometric properties of reflection surfaces will be preserved under isometry, which we highlight below. These directly follow from Proposition \ref{prop:RSfeatures}.

\begin{prop}
\label{prop:RLSfeatures}
Let $\rho:M'  \rightarrow M$ be an isometry between flat FAL complements, let $R' \subset M'$ be a reflection surface, and let $S=\rho(R')$ denote the corresponding reflection-like surface in $M$. Then 
\begin{itemize}
\item $S$ is fixed point-wise by an orientation reversing involution $i_{S}: M \rightarrow M$,
\item $S$ intersects the image of each $R'$-crossing disk of $M'$ orthogonally in its three simple non-separating geodesics, 
\item $S$ intersects every cusp of $M$ in two parallel curves,  
\item Let $T$ be the boundary torus of a cusp of $M$. Then $S$ determines a peripheral structure on $T$, namely, the image of the peripheral structure that $R'$ determines on $\rho^{-1}(T)$.  
\end{itemize}
\end{prop}

Our work in Section \ref{sec:PropsRefSur} will show that every reflection-like surface in an FAL complement is a reflection surface. This is highlighted in Corollary \ref{thm:uniquereflectionhomeo}. After proving this result, we will drop the term reflection-like surface, and instead, only use reflection surface. 

 A reflection-like surface $S \subset M$ determines a partition of the components of $\mathcal{A}$ into $S$-crossing circles and $S$-knot circles, coming from the partition of $\mathcal{A}'$ into $R'$-crossing circles and $R'$-knot circles. We use the term \textbf{$S$-structure} on $M$ to refer to this partition  of the components of $\mathcal{A}$, the peripheral structures on each cusp determined by $S$, and a fixed choice of images of $R'$-crossing disks in $M$.  Suppose $M$ contains distinct reflection-like surfaces $R$ and $S$. A component $L$ of $\mathcal{A}$ \textbf{changes type} if $L$ is an $R$-knot circle and an $S$-crossing circle (or vice-versa). This terminology will be useful when discussing flat FAL complements that admit multiple reflection surfaces.

In addition to reflection surfaces, (totally geodesic) thrice-punctured spheres contained in flat FAL complements will serve as powerful tools for analyzing homeomorphisms between these manifolds.  We now review properties of, and results regarding, embedded, totally geodesic, thrice-punctured spheres in flat FAL complements. 

Any thrice-punctured sphere has precisely six simple geodesics, three of which are separating, three of which are nonseparating, and none of which are closed, as depicted in Figure \ref{fig:3PS}. Any intersection of a totally geodesic thrice-punctured sphere $D$ with another totally geodesic surface in a flat FAL complement must be some  subset of these six simple geodesics on $D$ that are pairwise disjoint. This puts strong restrictions on the behavior of intersections between $D$ and other totally geodesic surfaces, which we shall exploit.

\begin{figure}[h]
\begin{center}
\includegraphics[width=2.5in]{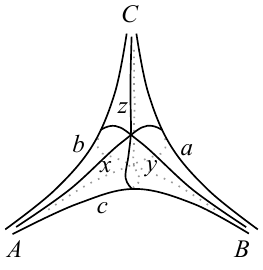}
\end{center}
\caption{A thrice-punctured sphere with separating geodesics labeled $x,y,z$ and nonseparating geodesics labeled $a,b,c$.}
\label{fig:3PS}
\end{figure}

Every FAL complement contains many totally geodesic thrice-punctured spheres that are not contained in a reflection surface, which we will call \textbf{non-reflection, thrice-punctured spheres} (relative to a designated reflection surface). Crossing disks in $M$, for example, constitute one category of such spheres.  Results in \cite{mrstz} imply there are three types of non-reflection thrice-punctured spheres in flat FAL complements, classified in terms of their intersections with $R$ and their punctures. Here, a \textbf{puncture} refers to the intersection of a totally geodesic surface in $M$ with a torus boundary of a cusp of $M$, which will produce a set of simple closed curves, each of which represents a slope on this torus. We sometimes use the term longitudinal puncture to refer to a puncture that is a representative for a longitude on this torus and we also use the term meridional puncture for a puncture that is a representative for a meridian on this torus (here meridian and longitude refer to the peripheral structure induced by a reflection surface $R\subset M$). We now define the three types of non-reflection, thrice-punctured spheres in flat FAL complements.

\begin{df} \label{df:NR3PS}
Let $D$ be a non-reflection thrice-punctured sphere in the flat FAL complement $M = \mathbb{S}^{3} \setminus \mathcal{A}$ with reflection surface $R$. 
\begin{itemize}
\item $D$ is a \textbf{crossing disk} if $D \cap R$ consists of the three non-separating geodesics of $D$ and the punctures of $D$ are one crossing circle longitude and two knot circle meridians,
\item $D$ is a \textbf{longitudinal disk} if $D \cap R$ consists of the three non-separating geodesics of $D$ and the punctures of $D$ are longitudes of three distinct crossing circles, and
\item $D$ is a \textbf{singly-separated disk} if $D \cap R$ consists of one separating geodesic of $D$ and the punctures of $D$ are a longitude of one crossing circle $C$ and two meridians of another crossing circle $C'$. 
\end{itemize}
Collectively, we refer to crossing disks and longitudinal disks as \textbf{$N$-disks} since their intersections with $R$ both consist of a non-separating geodesics of $D$. 
\end{df}

Note that this definition of crossing disk broadens the typical meaning of the term. The crossing circle $C$ illustrated in Figure \ref{fig:Types3ps}(a) bounds two distinct crossing disks, the one interior to $C$ as well as the shaded disk. The reason for extending the meaning of ``crossing disk'' is that any such disk $D$ for a crossing circle $C$ could be chosen as a crossing disk in $\mathcal{A}$ by replacing the current crossing disk with $D$. The two crossing disk structures come from fully augmenting diagrams of the same link obtained by flyping one twist region to a different part of the link diagram (see \cite{mrstz}). 

Parts (b) and (c) of Figure \ref{fig:Types3ps} illustrate a longitudinal and singly-separated disk, respectively. The following theorem from \cite{mrstz}  classifies non-reflection thrice-punctured spheres in flat FAL complements.

\begin{center}
\begin{figure}[h]
\[
\begin{array}{ccc}
\includegraphics[width=1.3in]{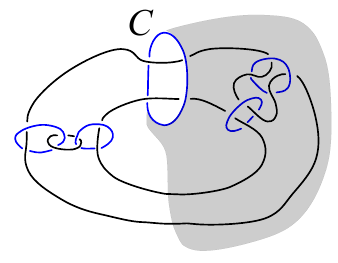} &\includegraphics[width=1in]{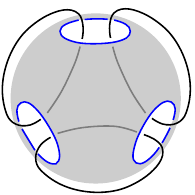}&\includegraphics[width=1.3in]{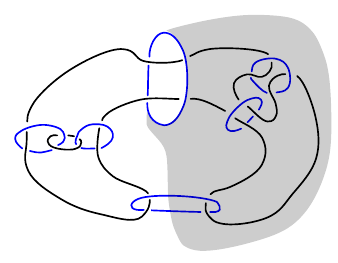}\\
(a)\textrm{ A crossing disk} & (b)\textrm{ A longitudinal disk} &(c)\textrm{ A singly-separated disk}
\end{array}
\]
\caption{Types of non-reflection, thrice-punctured spheres}
\label{fig:Types3ps}
\end{figure}
\end{center}

\begin{thm}\cite[Theorem 3.1]{mrstz}
\label{prop:BeltSumSummary}
Let $D$ be a non-reflection thrice-punctured sphere in the flat FAL complement $M = \mathbb{S}^{3} \setminus \mathcal{A}$ with reflection surface $R$. Then $D$ is orthogonal to $R$ in $M$ and $D$ is either a crossing, longitudinal, or singly-separated disk. 
\end{thm}

In our work, we will frequently make use of Theorem \ref{prop:BeltSumSummary} and the qualifications of non-reflection thrice-punctured spheres given in Definition \ref{df:NR3PS}.  The properties of separating and non-separating are topological, so if a reflection surface intersects a thrice-punctured sphere in its separating geodesic, so will any homeomorphic image of them.  Thus, if $D$ is singly-separated by the reflection surface $R$, then $h(D)$ is also separated by $h(R)$ along one separating geodesic. Similarly, the homeomorphic image of an $N$-disk relative to the reflection surface $R$ must intersect the reflection-like surface $h(R)$ in its non-separating geodesics.  In the case when $h(R)$ is again a reflection surface, we say that a homeomorphism preserves the type of a singly-separated disk, and preserves the property of being an $N$-disk.  The only type changes possible when $h(R)$ is a reflection surface, then, are between crossing and longitudinal disks.  These can be recognized by analyzing the images of punctures under this homeomorphism.




Finally,  we define a \textbf{separating pair} to be a pair of disjoint thrice-punctured spheres that separate $M$. The following theorem from \cite{mrstz} tells us exactly which pairs of thrice-punctured spheres are separating pairs in FAL complements. 

\begin{thm}\cite[Theorem 4.1]{mrstz}\label{thm:SepPair}
Let $\{S_1,S_2\}$ be a pair of disjoint essential thrice-punctured spheres in the complement $M = \mathbb{S}^3\setminus\mathcal{A}$ of the FAL $\mathcal{A}$.  The pair $\{S_1,S_2\}$ is a separating pair if and only if each is either a crossing disk or a singly-separated disk and their longitudinal slopes coincide.
\end{thm}

It is clear that any two crossing and/or singly-separated disks that share longitudinal crossing-circle punctures separate a flat FAL complement.  Theorem \ref{thm:SepPair} states that these are the only separating pairs in FAL complements.


\section{Flat FAL Complements with multiple reflection surfaces}
\label{sec:PropsRefSur}

Proposition \ref{prop:RSfeatures} shows that a reflection surface determines a lot of the geometric structure of a flat FAL complement. In this section, we will determine exactly how many reflection surfaces can exist in a flat FAL complement and exactly which flat FAL complements admit multiple reflection surfaces. This will be a useful first step towards our main goal of showing that flat FALs are determined by their complements.  We begin with some preliminary observations which specify how distinct reflection-like surfaces can intersect boundary tori of cusp neighborhoods.

As described in Section \ref{sec:FALs}, reflection-like surfaces induce a peripheral structure on boundary tori of cusp neighborhoods. Recall that the peripheral structure consists of two orthogonal, unoriented isotopy classes of simple closed curves (slopes), one labeled meridian and the other longitude. Let $M = \mathbb{S}^{3} \setminus \mathcal{A}$ denote a flat FAL complement with distinct reflection-like surfaces $R$ and $S$, and let $T$ denote the boundary torus of a cusp of $M$.  Our first lemma shows that the meridian-longitude pairs determined by $R$ and $S$ are the same set-wise.

\begin{lem}\label{lem:MeridianLongitude}
Let $M$ be a flat FAL complement with two distinct reflection-like surfaces $R$ and $S$ and let $T$ be the boundary torus of a cusp of $M$. Then an $R$-meridian of $T$ is either an $S$-meridian or $S$-longitude. Likewise, an $R$-longitude of $T$ is either an $S$-meridian or $S$-longitude. 
\end{lem}

\begin{proof}
The result will follow from the fact that these meridian-longitude pairs are orthogonal.

Let $T$ denote the boundary torus of a cusp neighborhood of $M$ for some cusp expansion of $M$. On this torus, let $m,\ell$ denote the $R$-meridian and $R$-longitude, with geodesic lengths $\mu, \lambda$, respectively. Likewise,   let $m', \ell'$ denote the $S$-meridian and $S$-longitude with lengths $\mu', \lambda'$,  respectively.  Since meridian-longitude pairs are orthogonal 
\[
\mu\lambda = Area(T)  = \mu'\lambda'. 
\]

Given an arbitrary orientation on the slopes $\{m,\ell\}$ and $\{m',\ell'\}$, there are integers $p,q,r,s$ with $m' = pm + q\ell$ and $\ell' = rm+s\ell$. By orthogonality we have
\[
\mu' = \sqrt{(p\mu)^2 + (q\lambda)^2}, \textrm{ and } \lambda' = \sqrt{(r\mu)^2 + (s\lambda)^2}.
\]
At least one of $p$ or $q$ is non-zero, which implies $\mu' \ge \textrm{min}\{\mu,\lambda\}$, with equality exactly when $m'$ is the shortest curve of the set $\{m,\ell\}$.  Moreover, when neither $p$ nor $q$ are zero, $\mu' > \textrm{max}\{ \mu,\lambda\}$.  Similar inequalities hold for $\lambda'$. If both $p$ and $q$ are non-zero, then 
\[
\mu'\lambda' > \textrm{max}\{ \mu,\lambda\}\textrm{min}\{\mu,\lambda\} = \mu\lambda.
\]

Since $\mu \lambda = \mu'\lambda'$, however, this implies the oriented and simple $m'$ is one of $\pm m$ or $\pm \ell$.  The slope $m'$, then, is one of the slopes $m$ or $\ell$.  Similarly, $\ell'$ equals the other slope and the sets $\{ m', \ell' \}$ and $\{ m,  \ell \}$ are equal. 
\end{proof}

Thus, if a flat FAL complement admits multiple reflection-like surfaces $\{R_i \}$, each such $R_i$-structure still induces the same basis on a boundary torus of a cusp, though their induced peripheral structures may differ since meridians and longitudes might switch roles (i.e. some components of $\mathcal{A}$ could change type).

\begin{rmk}
The proof of Lemma \ref{lem:MeridianLongitude} would not work if we assumed $M$ was just an FAL complement. If the corresponding FAL has some number of half-twists, then certain meridian-longitude pairs might no longer be orthogonal, and so, our arguments used above would no longer apply. We point this since multiple essential results for this paper rely on Lemma \ref{lem:MeridianLongitude}, and we want to make sure the reader knows why our work doesn't immediately apply more broadly to the class of FALs and not just flat FALs. 
\end{rmk}





The next two propositions clarify how a different reflection-like surface $S$ behaves relative to a given $R$-structure by considering $R$-knot circle and $R$-crossing circle cusps separately.

\begin{prop}\label{thm:KnotCircleMeridian}
Let $M$ be a flat FAL complement with two distinct reflection-like surfaces $R$ and $S$.  If $K$ is an $R$-knot circle and $T_K$ the boundary torus of a cusp corresponding to $K$, then $S\cap T_K$ is a pair of $R$-meridians of $K$.  
\end{prop}

\begin{proof}
Lemma \ref{lem:MeridianLongitude} implies that $S \cap T_K$ is either a pair of $R$-meridians or $R$-longitudes.  For the sake of contradiction, suppose that $S \cap T_K$ is a pair of $R$-longitudes.

Since $K$ is an $R$-knot circle, $R\cap T_K$ is a pair of $R$-longitudes.  First suppose that $R \cap T_K = S \cap T_K$. Then the  composition $\iota_{R}\circ\iota_{S}$ acts as the identity on $T_K$ and preserves the normal direction.  Therefore $\iota_{R}\circ\iota_{S}$ is the identity on $M$ by \cite[Proposition A.2.1]{BePe1992}, and $R= S$. This is a contradiction, which implies $S \cap T_K\ne R\cap T_K$. 

Now consider the case where $S\cap T_K$ are $R$-longitudes distinct from $R\cap T_K$, and let $D$ be an $R$-crossing disk punctured by $K$.  Note that $D$ intersects $T_K$ in $R$-meridians, which are orthogonal to the $R$-longitudes $S\cap T_K$.  Thus $S$ intersects $D$ orthogonally. Moreover, $R$ intersects $D$ in the non-separating geodesics of $D$, which implies that $S \cap D$ consists of a single separating geodesic with both ``ends" on $K$.  The other two punctures of $D$ must be on distinct cusps because one is an $R$-knot circle puncture, and the other an $R$-crossing circle puncture. Since $S$ intersects $D$ orthogonally along a separating geodesic, the reflection $\iota_{S}$ preserves $D$ and interchanges the two punctures on distinct cusps. This contradicts the fact that $\iota_{S}$ preserves all cusps, and $S\cap T_K$ cannot be $R$-longitudes distinct from $R\cap T_K$.

Since $S$ cannot meet $T_K$ in $R$-longitudes, the pair of curves $S \cap T_K$ are $R$-meridians. \end{proof}



\begin{prop}\label{prop:Rcrossingpossibilities}
Let $M$ be a flat FAL complement with two distinct reflection-like surfaces $R$ and $S$. Let $C$ be an $R$-crossing circle bounding the $R$-crossing disk $D$. Then either
\begin{enumerate}[label=\roman*.]
\item $D$ is a component of $S$, or
\item $S$ intersects $D$ orthogonally along the separating geodesic of $D$ with punctures on $C$.  In this case the $R$-crossing circle $C$ is also an $S$-crossing circle.
\end{enumerate}
\end{prop}

\begin{proof}
Let $C$ be the $R$-crossing circle bounding $D$.  By Lemma \ref{lem:MeridianLongitude} we know that $S \cap T_C$ is either two $R$-meridians or two $R$-longitudes of $C$.   

Consider the case where $S \cap T_C$ consists of two $R$-meridians of $C$.  Since $C$ is an $R$-crossing circle the intersection $R\cap T_C$ consists of $R$-meridians as well.  If the surfaces $R$ and $S$ intersect $T_C$ in the same curves then, as in the proof of Proposition \ref{thm:KnotCircleMeridian}, one arrives at the contradiction that $R=S$. Thus the $R$-meridians $S \cap T_C$ are distinct from those of $R\cap T_C$. Now $S$ and $D$ are an embedded totally geodesic surfaces in $M$, so their intersection is a union of disjoint, complete, simple geodesics.  This follows from the observation that the local picture of $S\cap D$, as seen in the universal cover, consists of two planes intersecting along a geodesic.  The only complete, simple geodesics of $D$ that intersect $T_C$ are two non-separating geodesics, say $\gamma_1, \gamma_2$, and one separating geodesic, label it $\gamma$. Since $R$ contains the geodesics $\gamma_1, \gamma_2$, their intersection with $T_C$ lies in $R\cap T_C$.  This implies that $S$ cannot contain $\gamma_1, \gamma_2$, since $S$ and $R$ intersect $T_C$ in distinct curves. Thus $S\cap D$ contains the separating geodesic $\gamma$ and neither $\gamma_1$ nor $\gamma_2$. All other complete, simple geodesics on $D$ intersect $\gamma$, which implies $\gamma = S\cap D$ and we are in case $(ii)$. Orthogonality follows from the fact that the $R$-meridians $S \cap T_C$ are orthogonal to the $R$-longitude $D \cap T_C$.

It remains to prove that $C$ is an $S$-crossing circle as well.  Since $S$ is reflection-like, there is a flat FAL $M'$ with reflection surface $R'$ and an isometry $h:M'\to M$ with $S=h(R')$.  Note that $D' = h^{-1}(D)$ is not contained in $R'$ since $D$ is not contained in $S$.  Moreover, since $S$ intersects $D$ in a separating geodesic, the geodesic $R'\cap D'$ separates $D'$ as well.  The classification of non-reflection thrice punctured spheres given in Theorem \ref{prop:BeltSumSummary} shows that $D'$ must be a singly-separated disk, and all punctures of singly-separated disks are crossing-circle punctures (see Figure \ref{fig:Types3ps}).  Thus all punctures of $D$ are $S$-crossing circles, and $C$ must be an $S$-crossing circle.

Now suppose $S \cap T_C$ consists of two $R$-longitudes of $T_C$. We will show $D\subset S$. Since $C$ is an $R$-crossing circle we know $R\cap T_C$ is an $R$-meridian, and since $D$ is an $R$-crossing disk for $C$ it intersects $T_C$ in a single $R$-longitude of $T_C$.  Thus $D\cap T_C$ is a single curve parallel to, or equal to one of, the curves in $S \cap T_C$.  

First consider the case where $D\cap T_C$ equals one of the curves of $S \cap T_C$, call it $\gamma$.  Then $S$ and $D$ are embedded totally geodesic surfaces whose intersection contains $\gamma$ which, despite being geodesic in the induced Euclidean metric on $T_C$, is not a geodesic in $M$.  If $S$ and $D$ met transversally, their intersection could only contain geodesics. As $D$ is connected and intersects $S$ in a non-geodesic curve, it must be a subset of $S$. 

Now consider the case where $D\cap T_C$ is parallel to, and distinct from, the curves of $S\cap T_C$. We will show this case leads to a contradiction.  As above, let $M'$ be a flat FAL with reflection surface $R'$ that admits an isometry $h:M'\to M$ for which $S = h(R')$.  Further, let $D' = h^{-1}(D)$ and $T'=h^{-1}\left(T_C\right)$. 

Since we are assuming $D\cap T_C$ is disjoint from $S\cap T_C$, so also $D'\cap T'$ and $R'\cap T'$ are disjoint and $D'$ is not a subset of $R'$.  Then $D'$ is a non-reflection, thrice-punctured sphere in $M'$ and must be one of the types described in Theorem \ref{prop:BeltSumSummary} (see Figure \ref{fig:Types3ps}).  The classification shows that punctures of non-reflection, thrice-punctured spheres are either knot meridians, crossing longitudes or crossing meridians.  Of these, only the crossing meridians of singly-separated disks are disjoint from the reflection surface (see Figure \ref{fig:Types3ps}$(c)$).  Thus $D'$ is a singly-separated disk in $M'$ and $D'\cap T'$ consists of two crossing circle meridians.  This, however, contradicts the fact that $D\cap T_C$ is a single component.  

Therefore, if $S \cap T_C$ consists of two $R$-longitudes of $T_C$, then $D$ is a component of $S$ and we are in case $(i)$ of the theorem.
 \end{proof}

\subsection{Three Reflection Surfaces}
\label{sec:MRS}

In this subsection we classify flat FAL complements that contain more than two distinct reflection-like surfaces. 


\begin{thm}
\label{thm:3rs}
Suppose a flat FAL complement $M = \mathbb{S}^{3} \setminus \mathcal{A}$ admits at least three distinct reflection-like surfaces. Then $\mathcal{A}$ is equivalent to the Borromean Rings and $M$ contains exactly three reflection-like surfaces, all of which are reflection surfaces. 
\end{thm}

\begin{proof}
First, we show that a flat FAL complement $M$ can admit at most 3 distinct reflection-like surfaces. This follows from the observation that at most one reflection-like surface can satisfy each case of Proposition \ref{prop:Rcrossingpossibilities}.  To see this, let $D$ be a crossing disk with respect to a reflection-like surface $R$, and suppose $S_1$ and $S_2$ are reflection-like surfaces distinct from $R$ in $M$. If both $S_1$ and $S_2$ contain $D$ (so satisfy Proposition \ref{prop:Rcrossingpossibilities}$(i)$), then $\iota_{S_2}\circ\iota_{S_1}$ is the identity on $D$ and preserves the normal direction. By \cite[Proposition A.2.1]{BePe1992} we have $S_1= S_2$.  On the other hand, if $S_1$ and $S_2$ both satisfy Proposition \ref{prop:Rcrossingpossibilities}$(ii)$, then both intersect $D$ orthogonally along the same separating geodesic.  This implies $\iota_{S_2}\circ\iota_{S_1}$ fixes any point on this geodesic as well as the tangent space there, so again $S_1 = S_2$.

Thus a flat FAL complement has at most 3 distinct reflection-like surfaces, and now suppose $M$ has $3$ distinct reflection-like surfaces: $R$, $S$, and $P$. Moreover, let $D$ be an $R$-crossing disk bounded by the $R$-crossing circle $C$.  Then the argument given above shows that (up to relabeling) $P$ contains $D$ while $S$ and $D$ satisfy Proposition \ref{prop:Rcrossingpossibilities}$(ii)$. In particular, $R$ and $S$ intersect $T_C$ in disjoint pairs of $R$-meridians. Our goal is to show that $P$ consists of exactly two (disjoint) thrice-punctured spheres, which allows us to quickly classify all such flat FAL complements with such a reflection surface. 

Since $S$ and $D$ satisfy Proposition \ref{prop:Rcrossingpossibilities}$(ii)$, $C$ is an $S$-crossing circle. We will let $D'$ be an $S$-crossing disk for $C$. Note that, $D' \neq D$  since $S$ intersects $D$ in a separating geodesic so $D$ cannot be an $S$-crossing disk. Now consider the surfaces $R$ and $P$ relative to $D'$.  Since $D'$ intersects $S$ orthogonally, and $R\cap T_C$ is parallel to $S\cap T_C$, we have $R$ intersects $D'$ orthogonally.  Proposition \ref{prop:Rcrossingpossibilities} applied to $D'$ and $R$ implies that they satisfy case $(ii)$.  Hence $P$ and $D'$ must satisfy case $(i)$. Thus $D'$ is a component of $P$, and $D \cup D' \subset P$.

Now consider $D\cup D'$ relative to the reflection-like surface $S$.  The disk $D$ is an $S$-singly-separated disk which shares a longitudinal crossing-circle puncture with the $S$-crossing disk $D'$. Thus $D'$ and $D$ form a separating pair by Theorem \ref{thm:SepPair}.  Since no proper subset of a reflection-like surface separates, we have $P = D\cup D'$. 

The reflection-like surface $P = D \cup D'$ has a total of six punctures, and must intersect the boundary of each cusp in two curves; so there are three cusps in $M$.  The only three-component flat FAL is the Borromean rings, so $\mathcal{A}$ must be the Borromean rings. 

\begin{figure}[h]
\begin{center}
\includegraphics[width=2.5in]{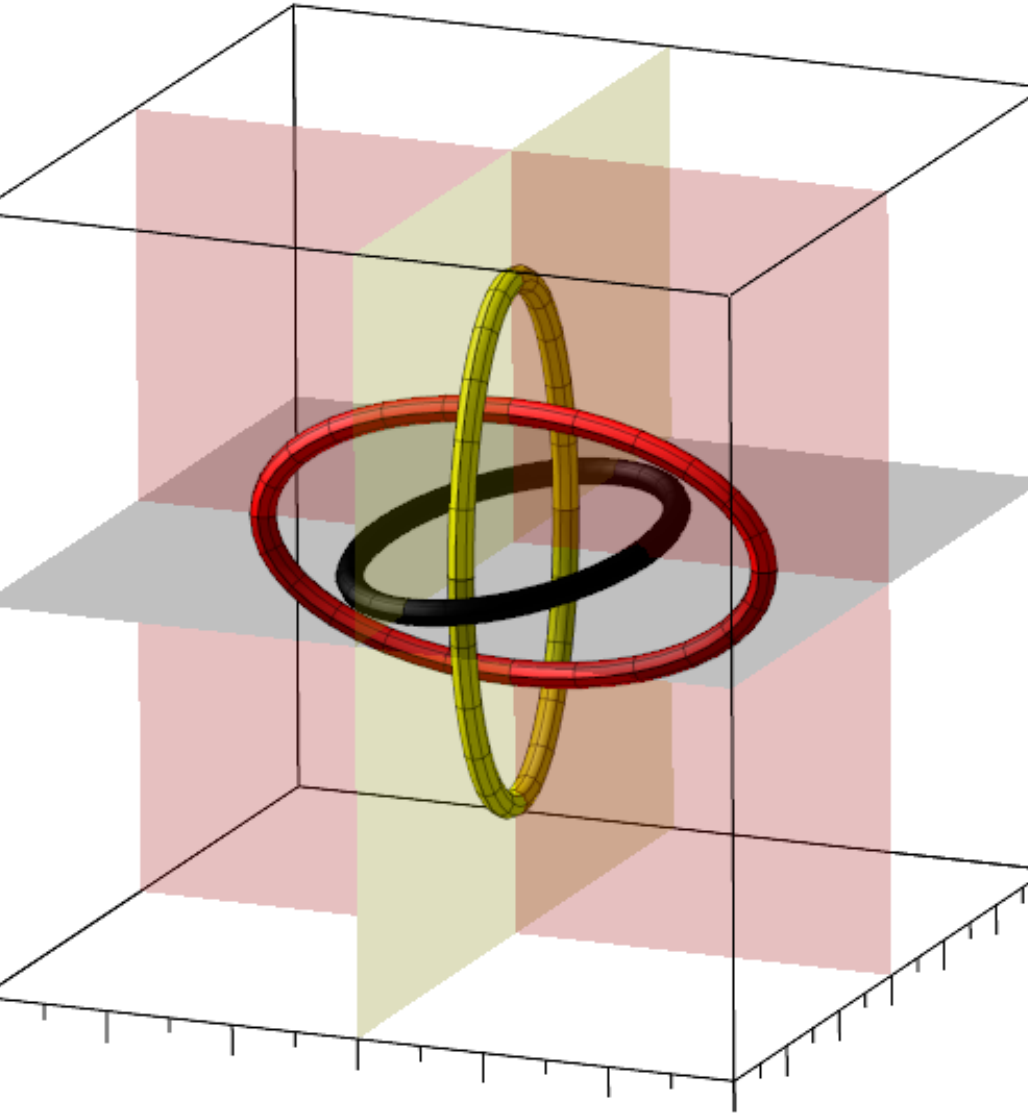}
\end{center}
\caption{Three reflection surfaces in the Borromean rings complement}
\label{fig:BringsMRS}
\end{figure}

Finally, note that the Borromean rings complement contains three distinct reflection surfaces, as depicted in Figure \ref{fig:BringsMRS}. Each shaded plane is a reflection surface with the link component it contains serving as the one knot circle.

\end{proof}


\subsection{Two Reflection Surfaces}
\label{sec:2RS}

Let $M = \mathbb{S}^{3} \setminus \mathcal{A}$ be a flat FAL complement with reflection-like surface $R$ and an additional (distinct)  reflection-like surface $R'$. Our goal in this subsection is to show any such flat FAL must be equivalent to either the Borromean Rings, $P_n$ with $n \geq 3$, or $O_n$ with $n \geq 2$. See Figure \ref{fig:BP4O4} for FAL diagrams of $P_4$ and $O_4$. 

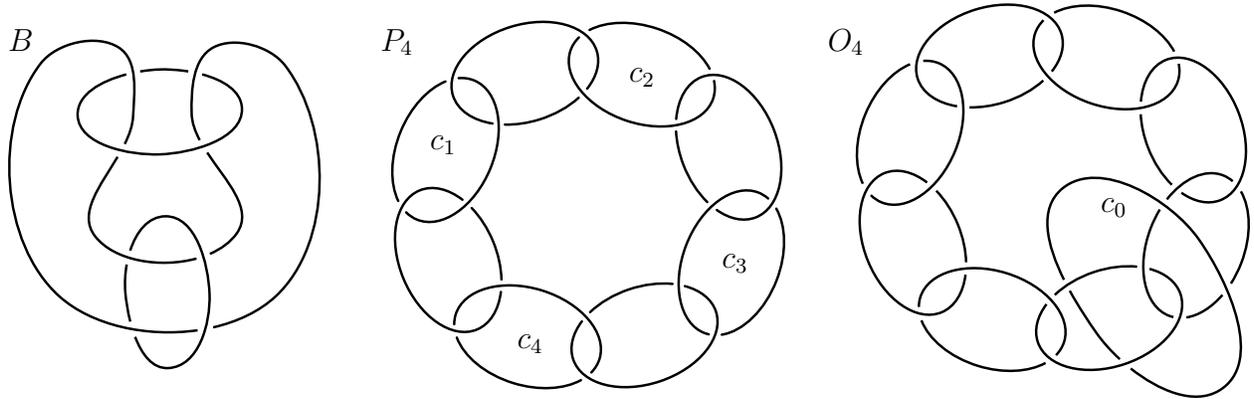
\begin{figure}[ht]
	\centering
	\begin{overpic}[width = \textwidth]{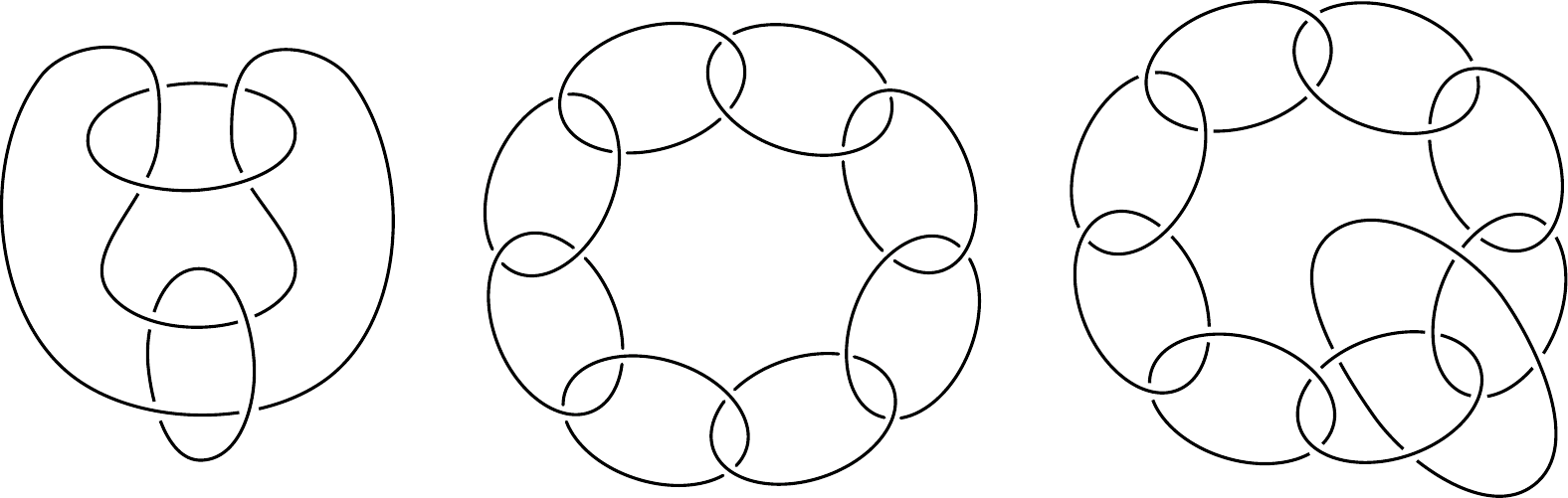}
		\put(0,28){$B$}
		\put(30,28){$P_4$}
		\put(66,28){$O_4$}
		\put(34,20){$c_{1}$}
		\put(50,25.5){$c_{2}$}
		\put(57.5,10.5){$c_{3}$}
		\put(41,4){$c_{4}$}
		\put(88,15){$c_{0}$}
		\end{overpic}
	\caption{On the left, the Borromean Rings, $B$, is depicted as a flat FAL. In the middle, $P_4$ is depicted with its four crossing circles labeled. On the right, $O_4$ is depicted, which can be constructed by adding the crossing circle $C_0$ to $P_4$.}
	\label{fig:BP4O4}
\end{figure}

More generally, $P_n$ is the minimally-twisted chain with $2n$ components.  Thus $P_n$ is a flat FAL that admits an  FAL diagram with $n$ crossing circles and $n$ knot circles, linked together in a chain alternating between crossing circles and knot circles. This link can also be described as fully augmenting a pretzel link with $n$ twist regions with an even number of crossings in each twist region, and performing a homeomorphism of the complement to undo all twists from each twist region.  The link $O_n$ is a flat FAL that admits an  FAL diagram with $n+1$ crossing circles and $n$ knot circles, where this FAL diagram can be obtained by adding a single crossing circle to the FAL diagram for $P_n$, just as depicted in Figure \ref{fig:BP4O4}. 

\begin{figure}[h]
\includegraphics{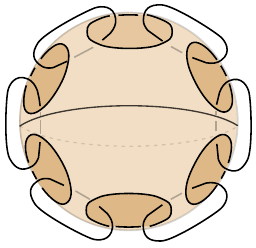}
\caption{The second reflection-like surface in $P_6$.}
\label{fig:SecondRS}
\end{figure}

Each of $P_n$ and $O_n$ contain (at least) two distinct reflection surfaces: one corresponding with the projection plane in the following figures and one corresponding with a $2$-sphere containing $C_1, \ldots, C_n$ and meeting $K_1, \ldots, K_n$ orthogonally in these figures; in the case of $O_n$, this $2$-sphere meets the link component $C_0$ orthogonally.  Note that a $90^{\circ}$ rotation of $\mathbb{S}^3$ about the axis of the chain interchanges the two reflection surfaces (see Figure \ref{fig:SecondRS} for the second reflection-like surface in $P_6$).

We now prove a useful fact about about thrice-punctured spheres contained in a reflection surface, which we will make use of throughout this subsection.

\begin{lem}
\label{lem:Reflection3ps}
Let $M$ be a flat FAL complement with reflection-like surface $R$, and let $D$ be a thrice-punctured sphere component of $R$.  Then $D$ has an odd number of $R$-knot circle punctures. 
\end{lem}

\begin{proof}
The properties involved are topological, so without loss of generality we assume that $R$ is the reflection surface of the flat FAL complement $M = \mathbb{S}^3\setminus \mathcal{A}$. Further let $\mathbb{S}^2$ be the projection two-sphere in $M$, so that $R=\mathbb{S}^2\setminus \mathcal{A}$ is the reflection surface. Finally, let $D$ be a thrice-punctured sphere component of $R$.  Since each component of $R$ has at least one knot circle puncture, we must show that $D$ does not have two knot circle punctures.  

Suppose, on the contrary, that exactly two knot circles, labeled $J$ and $K$, puncture $D$.  We will show that one of $J$ or $K$ has at most one crossing circle linking it, contradicting the fact that every knot circle is linked by at least two crossing circles. Some notation will be useful.

Note that $\mathbb{S}^2\setminus \left(J\cup K\right)$ consists of two disks and an annulus. Let $D_J$ and $D_K$ denote the disk components bounded by $J$ and $K$, respectively, and note that $D$ is the annular component.  Then $D$ is a once-punctured annulus and the additional puncture, label it the point $p$, comes from a crossing circle $C$.  Now $C$ punctures $\mathbb{S}^2$ twice, and only once in $D$, so (possibly after relabeling) $C$ also punctures $D_K$. Now $D$ is a thrice-punctured sphere, so there is a unique geodesic $\gamma_{pK}$ in $D$ joining the punctures $p$ and $K$. Thus, if $D_C$ is a crossing disk bounded by $C$ then $\gamma_{pK} = D_C\cap D$.

Now let $C'$ be a crossing circle other than $C$ with a crossing disk $D'$ punctured by $J$.  As $C$ is the only crossing circle puncturing $D$, we know $C'$ is disjoint from $D$.  However, since $J$ punctures $D'$ we know $D'\cap D$ is a nonempty subset of geodesics on $D$ with at least one endpoint on $J$.  The separating geodesic $\gamma_J$ of $D$ with both endpoints on $J$ intersects $\gamma_{pK}$ non-trivially, since $\gamma_{pK}$ is the non-separating geodesic of $D$ opposite $J$.  Since $\gamma_{pK} = D_C\cap D$ and crossing disks are disjoint, $D'$ cannot intersect $D$ in $\gamma_J$.  Further, $D'\cap D$ cannot be the non-separating geodesic of $D$ joining $J$ and $p$ since $p$ is on the crossing circle $C$ and $C'\ne C$.  Thus $D'\cap D$ must be the non-separating geodesic $\gamma_{JK}$ of $D$ joining $J$ and $K$.

We have shown that any crossing disk punctured by $J$ intersects $D$ in the geodesic $\gamma_{JK}$.  Since crossing disks are orthogonal to the reflection surface, at most one crossing disk intersects $D$ along a given geodesic.  Thus at most one crossing circle bounds a crossing disk punctured by $J$, contradicting the fact that each knot circle in an FAL is linked by at least two crossing circles.
\end{proof}

Each reflection-like surface partitions the cusps corresponding to components of $\mathcal{A}$ into a set of knot circle cusps and a set of crossing circle cusps. If some component of $\mathcal{A}$ changes type, then the geometry of  $M$ is greatly restricted. It is also possible that no ``swapping'' occurs, i.e., $R$ and $R'$ induce the same partition on the components of $\mathcal{A}$ into knot circle (cusps) and crossing circle (cusps). We now show that the latter case can not happen when $M$ has two distinct reflection surfaces. 

\begin{lem} 
\label{lem:samesurface}
Suppose $M = \mathbb{S}^{3} \setminus \mathcal{A}$ is a flat FAL complement with two  reflection-like surfaces $R$ and $R'$ that induce the same partitions on the cusps of $M$ into crossing circle cusps and knot circle cusps. Then $R = R'$.
\end{lem}
\begin{proof}
Proposition 4.6 of \cite{HMW2020} shows that there exists some $R$-crossing circle $C$ of $\mathcal{A}$  such that the corresponding $R$-crossing disk $D$ is the unique totally geodesic thrice-punctured sphere in $M$ that intersects the torus boundary of the cusp corresponding to $C$ in a longitude and intersects the boundary of two $R$-knot circle cusps (not necessarily distinct) in a  meridian on each of these cusps.  By assumption, $C$ is also an $R'$-crossing circle and the two $R$-knot circle cusps intersecting $D$ are also $R'$-knot circle cusps. By Lemma \ref{lem:Reflection3ps}, $D$ can not be a subset of $R'$, and so, $D$ is a non-reflection thrice-punctured sphere for $R'$.  Since $D$ has two $R'$-knot circle and one $R'$-crossing circle punctures,  Theorem \ref{prop:BeltSumSummary} implies $D$ must also be the corresponding $R'$-crossing disk for $C$. Let $T_K$ be the boundary torus of a cusp neighborhood for one of the knot circles that intersect $D$. Then $D \cap T_K$ is orthogonal to both $R \cap T_K$ and $R' \cap T_K$. However, if $R \neq R'$, then Proposition \ref{thm:KnotCircleMeridian} implies that $R \cap T_K$ and $R' \cap T_K$ are orthogonal, which provides a contradiction. Thus, $R = R'$ under these conditions. 
\end{proof}

In the rest of this subsection, we consider the case where a component of $\mathcal{A}$ changes type. First, we give a useful lemma. 

\begin{lem} 
\label{lem:factoid} Let $M = \mathbb{S}^{3} \setminus \mathcal{A}$ be a flat FAL complement with two distinct reflection-like surfaces $R$ and $R'$. Suppose $L$ is an $R$-knot circle and an $R'$-crossing circle.  If $D'$ is an $R'$-crossing disk for $L$, then $D'\subset R$; moreover, $D'$ is unique. \end{lem}

\begin{proof} To see this, note that since $L$ is an $R$-knot circle, the reflection surfaces $R'$ and $R$ intersect $T_L$ in orthogonal pairs of curves by Proposition \ref{thm:KnotCircleMeridian}.  As an $R$-knot circle there are at least two $R$-crossing disks that link $L$.   Then $\partial D'$ is a single curve on $T_L$ orthogonal to $R'\cap T_L$, so it must be an $R$-longitude parallel to $R\cap T_L$.  As such $D'$ intersects $R$-crossing disks linking $L$, and that intersection must be in a non-separating geodesic with one puncture on $L$.  To see this let $D$ be an $R$-crossing disk punctured by $L$. If $D'$ intersected $D$ in a separating geodesic of $D$ with endpoints on $L$ then $\partial D'$ would consist of two $R$-longitudes of $L$, but $\partial D' \cap T_L$ is a single $R$-longitude since $D'$ is an $R'$-crossing disk.  

Then $R$ and $D'$ have a common intersection with $R$-crossing disks that link $L$, and both are orthogonal to such crossing disks, so $D' \subset R$.

To see that $D'$ is unique, assume there is a second $R'$-crossing disk $D^{\ast}$ for $L$.  By the above argument, $D^{\ast}\subset R$ as well.  Then $D'\cup D^{\ast}$ form a separating pair contained in $R$.  Since a proper subset of $R$ cannot separate, $R$ must equal $D'\cup D^{\ast}$.  As in the proof of Theorem \ref{thm:3rs}, a cusp count shows manifold has three cusps and must be Borromean rings.  The Borromean rings, however, do not have a crossing circle that bounds two crossing disks in same $R'$-structure; therefore, $D'$ is unique. \end{proof}

Suppose $K$ is a component of $\mathcal{A}$ that corresponds with an $R$-knot circle cusp and an $R'$-crossing circle cusp. Let $D'$ be the $R'$-crossing disk bound by $K$. Then Lemma \ref{lem:factoid} guarantees that $D'$ is a component of $R$. Since $K$ is an $R$-knot circle, there exists some $R$-crossing circle $C$ that links $K$, along with an $R$-crossing disk $D$ which $K$ punctures. We now consider the intersection patterns for $D \cap D'$. The work of Yoshida \cite[Proposition 3.1]{Y2018} shows that there are three possible cases to consider: (i) $D \cap D'$ is a single non-separating geodesic on both $D$ and $D'$, (ii) $D \cap D'$ is a pair of non-separating geodesics on both $D$ and $D'$, and  (iii) $D \cap D'$ is a separating geodesic on one and non-separating on the other.  Note that since $D'\subset R$ and $D\cap R$ consists of non-separating geodesics on $D$, the curve $D\cap D'$ must be non-separating on $D$.  Hence, in case (iii), the separating geodesic must be on $D'$. See Figure \ref{fig:3cases} for corresponding diagrams. These cases are considered separately in the next three propositions.

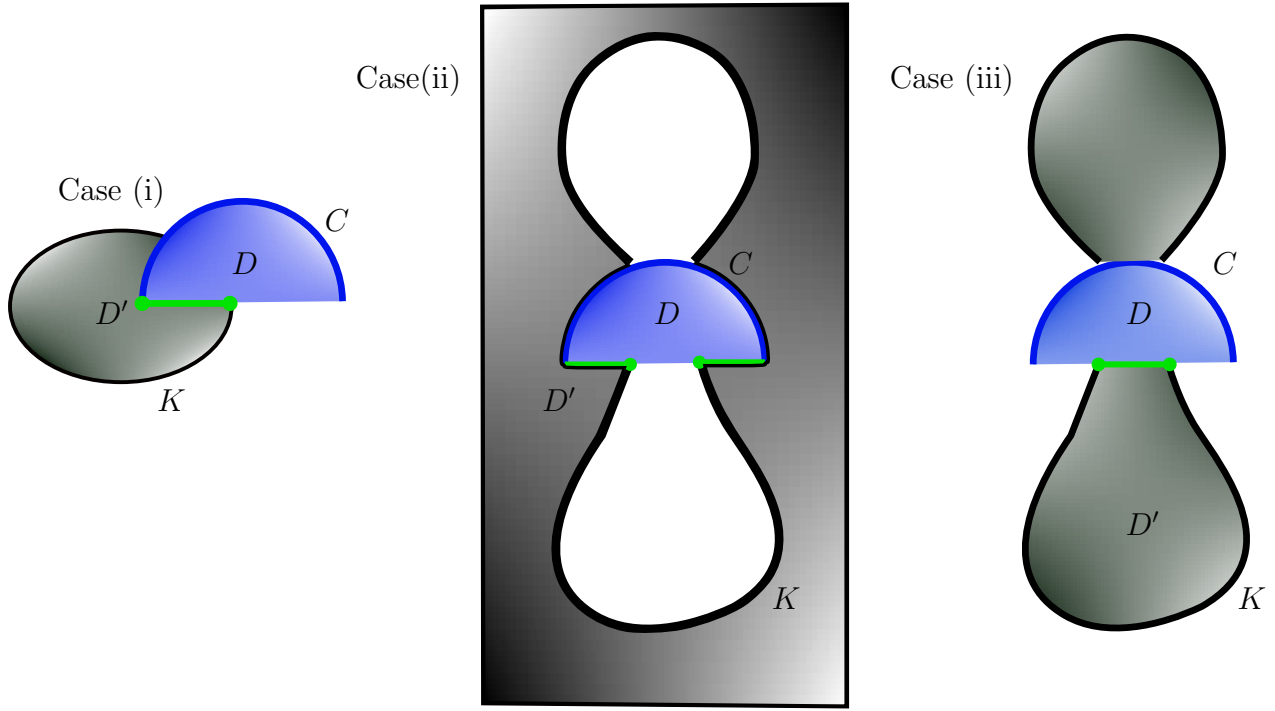
\begin{figure}[ht]
	\centering
	\begin{overpic}[width = \textwidth]{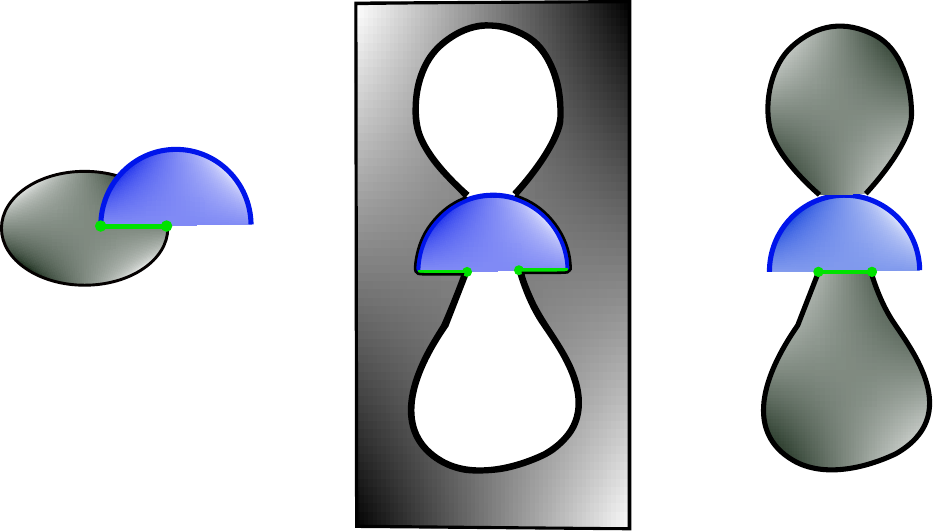}
		\put(4,41){Case (i)}
		\put(28,50){Case(ii)}
		\put(71,50){Case (iii)}
		\put(7,31){$D'$}
		\put(12,24){$K$}
		\put(18,35){$D$}
		\put(25.5,38.5){$C$}
		\put(43,24){$D'$}
		\put(61.5,8){$K$}
		\put(52,31){$D$}
		\put(58,35){$C$}
		\put(90,14){$D'$}
		\put(99,8){$K$}
		\put(90,31){$D$}
		\put(97,35){$C$}
		\end{overpic}
	\caption{The three cases for intersection patterns for $D' \cap D$ where the projection plane here is determined by $R$. In each figure, the thrice-punctured spheres $D$ and $D'$ are shaded blue and grey, respectively. Only the top half of the crossing disk $D$ is drawn in each diagram. Intersections between $D$ and $D'$ are highlighted in green.}
	\label{fig:3cases}
\end{figure}

\begin{prop}
\label{prop:CaseI}
Suppose $M = \mathbb{S}^{3} \setminus \mathcal{A}$ is a flat FAL complement with two distinct reflection-like surfaces $R$ and $R'$. Let $K$ be a component of $\mathcal{A}$ that is an $R'$-crossing circle  and an $R$-knot circle, and let $D'$ be the $R'$-crossing disk $K$ bounds.  Let $C$ be an $R$-crossing circle bounding the $R$-crossing disk $D$, with $K$ puncturing $D$. If the intersection  $D' \cap D$ is a single non-separating geodesic on each thrice-punctured sphere (case $(i)$ of Figure \ref{fig:3cases}), then $\mathcal{A}$ is either $P_n$ with $n \geq 3$, or $O_n$ with $n \geq 2$. 
\end{prop}

\begin{proof}
Let $K$, $D'$, $C$ and $D$ be as in the statement of the proposition.  By Lemma \ref{lem:factoid} the disk $D'$ is a component of $R$ and note that, since $D'\cap D$ is a single non-separating geodesic on both, $D'$ punctured once each by $K$ and $C$. Lemma \ref{lem:Reflection3ps} implies that the remaining puncture of $D'$ comes from an $R$-crossing circle, say $J$.

Since $D'$ has punctures along three distinct cusps, the $R$-crossing disk $D_J$ that $J$ bounds intersects $D'$ in a single non-separating geodesic for both disks (case $(i)$ of Figure \ref{fig:3cases}).  In the $R$-structure, then, the cusps $J,K,C$ form a chain of three components.

Now consider $J$, $K$ and $C$ in the $R'$-structure.  We know that $K$ is an $R'$-crossing circle and that $D'$ is its corresponding $R'$-crossing disk.  Thus the other two punctures, $J$ and $C$, must be $R'$-knot circles.  By Lemma \ref{lem:factoid}, then, the $R$-crossing disks $D$ and $D_J$ are components of the reflection surface $R'$.

Applying Lemma \ref{lem:Reflection3ps} to the disks $D$ and $D_J$ in the $R'$-structure implies the third puncture in each is an $R'$-crossing circle puncture.  Translating the picture to the $R$-structure, and applying Lemma \ref{lem:factoid} to each end of the chain produces a chain of four or five components.   There are four components if the punctures unaccounted for in $D$ and $D_J$ correspond to the same $R'$-crossing circle.

Since $\mathcal{A}$ has finitely many components, iterating this argument terminates in a sublink of $\mathcal{A}$ isotopic to $P_n$, and which we denote $\mathcal{P}$.  Note that the $R$-crossing disks for $\mathcal{P}$ are components of $R'$, and vice-versa.  For convenience, let $K_1,\dots,K_n$ denote the $R$-knot circles of $\mathcal{P}$, and let $S_1,\dots,S_n$  be the corresponding thrice-punctured sphere components of $R$ that they bound (so the $S_i$ are $R'$-crossing disks).  Further let $C_1,\dots,C_n$  denote the $R$-crossing circles of $\mathcal{P}$, and let $D_1,\dots,D_n$ be the respective $R$-crossing disks that they bound.

Consider the action of $\iota_{R'}$ on the components of the reflection surface $R$.  Each of the $S_i$ reflect to themselves since they are $R'$-crossing disks.  Now let $\widehat{R} = R\setminus \left(\cup_{i=1}^n S_i\right)$. Our goal is to show that the $K_i$ are the only $R$-knot circle punctures of $\widehat{R}$.  Note that cutting $\widehat{R}$ along its intersection with the $D_i$ separates $\widehat{R}$ into two subsets $\widehat{R}_0$ and $\widehat{R}_1$, which are $R'$-reflections of each other.  Let $N$ be an $R$-knot circle distinct from the $K_i$ that punctures $\widehat{R}$.  Then $N$ cannot pass through the $D_i$, which are already punctured twice by the $K_i$.  Thus $N$ is contained entirely in one of $\widehat{R}_0$ or $\widehat{R}_1$, and $\iota_{R'}(N)$ is in the other.  This contradicts the fact that $\iota_{R'}$ preserves each cusp; therefore, no such $N$ exists and the $K_i$ are the only $R$-knot circle punctures of $\widehat{R}$.  Since $R = \widehat{R} \cup \left(\cup_{i=1}^n S_i\right)$ and the only $R$-knot circle puncturing $S_i$ is $K_i$ for $i= 1, \ldots, n$, we have that the $K_i$ are the only $R$-knot circles puncturing $R$, and so, they are also the only $R$-knot circles of $\mathcal{A}$. 

At this stage, any components of $\mathcal{A}$ that are not in $\mathcal{P}$ must be $R$-crossing circles that puncture $\widehat{R}$ twice.  Our goal is to show there is at most one such $R$-crossing circle.  Let $\widetilde{C}$ be an $R$-crossing circle of $\mathcal{A}$ that is not a component of $\mathcal{P}$, and let $\widetilde{D}$ be an $R$-crossing disk bounded by $\widetilde{C}$.  Since the $K_i$ are the only $R$-knot circles of $\mathcal{A}$, the disk $\widetilde{D}$ intersects some $K_j$ and the thrice-punctured sphere $S_j\subset R$ that it bounds. The crossing disks of $\mathcal{P}$ intersect $S_j$ in the two non-separating geodesics on $S_j$ with endpoints on $K_j$, so  $\gamma_j=S_j\cap \widetilde{D}$ must be a separating geodesic on $S_j$ since $R$-crossing disks must be disjoint.  

Suppose there are at least two  $R$-crossing circles in $\mathcal{A} \setminus \mathcal{P}$, and label a second one $\widetilde{C}'$ with $R$-crossing disk $\widetilde{D}'$. We now describe how to construct an open, embedded, essential annulus $A \subset M$, with boundary (in $\mathbb{S}^3$) curves $\widetilde{C}$ and $\widetilde{C}'$. Consider the standard genus one Heegaard decomposition of $\mathbb{S}^{3} = T_1 \cup T_2$, where $T_1$ and $T_2$ are solid tori. A sufficiently small closed neighborhood of $\displaystyle{\cup_{i=1}^n\left(K_i\cup S_i \cup C_i \cup D_i\right)}$ produces a solid torus in $\mathbb{S}^{3}$ that is isotopic to one of these solid tori, say $T_1$, and disjoint from $\widetilde{C}$, $\widetilde{C}'$, as well as any additional $R$-crossing circles not in $\mathcal{P}$. In addition, by taking appropriate neighborhoods, we can assume each of $\widetilde{D}$ and $\widetilde{D}'$ intersects $T_1$ in a meridional disk of $T_1$.   Note that, $\partial T_1 \setminus (\widetilde{D} \cup \widetilde{D}')$ produces two embedded annuli in $M$, and label one of these $A'$. Then $A = A' \cup (\widetilde{D} \setminus int(T_1)) \cup (\widetilde{D}' \setminus int(T_1))$ provides the desired annulus, contradicting hyperbolicity. Thus, $\mathcal{A} \setminus \mathcal{P}$ is either empty or contains exactly one $R$-crossing circle that intersects $\widehat{R}$ twice, once in  $\widehat{R}_0$ and once in $\widehat{R}_1$, as needed. 
\end{proof}

\begin{prop}
\label{prop:CaseII}
Suppose $M = \mathbb{S}^{3} \setminus \mathcal{A}$ is a flat FAL complement with two distinct reflection-like surfaces $R$ and $R'$. Let $K$ be a component of $\mathcal{A}$ that is an  $R'$-crossing circle and $R$-knot circle, and let $D'$ be the $R'$-crossing circle disk $K$ bounds. Let $C$ be an $R$-crossing circle bounding the $R$-crossing disk $D$, with $K$ puncturing $D$. If the intersection  $D' \cap D$ is a pair of non-separating geodesics on each thrice-punctured sphere (case $ii$ in Figure \ref{fig:3cases}), then $\mathcal{A}$ is the Borromean Rings.
\end{prop}

\begin{proof}
By the work of Yoshida \cite[Lemma 3.7]{Y2018}, any hyperbolic $3$-manifold with two thrice-punctured spheres intersecting in this manner must be a (possibly empty) Dehn filling of one of three  manifolds: a certain double cover of the Whitehead link complement, the Borromean rings complement, or the minimally twisted hyperbolic 4-chain link complement. See Figure 17 in \cite{Y2018} for diagrams of these links. All three of these manifolds have the common hyperbolic volume of $2v_{8}$, where $v_{8}$ denotes the volume of a regular ideal hyperbolic octahedra. At the same time, the work of Purcell \cite[Proposition 3.6]{P2011} shows that the Borromean rings complement is the unique minimal volume flat  FAL complement. Since non-empty Dehn filling strictly decreases volume, the only flat  FAL complement with thrice-punctured spheres intersecting in this manner is the Borromean rings complement, as needed.
\end{proof}

\begin{prop}
\label{prop:CaseIII}
Suppose $M = \mathbb{S}^{3} \setminus \mathcal{A}$ is a flat FAL complement with two distinct reflection-like surfaces $R$ and $R'$. Let $K$ be a component of $\mathcal{A}$ that is an  $R'$-crossing circle  and $R$-knot circle, and let $D'$ be the $R'$-crossing circle disk $K$ bounds. Let $C$ be an $R$-crossing circle bounding the $R$-crossing disk $D$, with $K$ puncturing $D$. If the intersection  $D' \cap D$ is a separating geodesic on $D'$ and a non-separating geodesic on $D$ (case $(iii)$ in Figure \ref{fig:3cases}), then $\mathcal{A}$ is either the Borromean Rings or $O_n$ with $n \geq 2$. 
\end{prop}

\begin{proof}
Let $K$, $D'$, $C$ and $D$ be as in the statement of the proposition.  By Lemma \ref{lem:factoid} the disk $D'$ is a component of $R$, and note that it is punctured once by $K$. Since $D'$ is  a thrice-punctured sphere, we know that it must have two more punctures. We break down this proof into cases depending on whether those punctures come from $R$-crossing circles or $R$-knot circles. Note that, since $D \cap D'$ is a separating geodesic $\gamma_K$ on $D'$, the thrice-punctured sphere $D$ partitions $D'$ into two regions separated by $\gamma_K$.

\textbf{Case I:} Suppose the other two punctures of $D'$ come from $R$-crossing circles $C_1$ and $C_2$. Further, suppose $C_1 = C_2$. Then $C_1$ and $C_2$ must puncture different regions of $D'\setminus D$ since $D'$ must have a puncture on each side of the separating geodesic $\gamma_k$. In this case, $\mathcal{A}$ contains a Borromean Rings sub-link, $L_B = K \cup C \cup C_1$. Let $D_1$ designate the $R$-crossing disk corresponding to $C_1$. As an $R$-crossing disk, $D_1$ intersects $R$ in the three nonseparating geodesics on $D_1$. Two of these geodesics must be contained in $D' \subset R$ since $C_1$ punctures $D'$ in two different regions and $D_1$ can not intersect $\gamma_K$. Since $D' \cap D_1$ is a pair of non-separating geodesics on $D_1$, it follows that this intersection is also a pair of non-separating geodesics on $D'$ by the work of Yoshida \cite[Proposition 3.1]{Y2018}. Following the proof of Proposition \ref{prop:CaseII} with $D_1$ replacing $D$ implies that $A = L_B$, i.e., $A$ is the Borromean rings. 


Now suppose that $C_1$ and $C_2$ are distinct $R$-crossing circles, each of which puncture $D'$. Then $C_1$ has an $R$-crossing disk $D_1$, which is punctured by $K$ and some distinct $R$-knot circle $K_1$. This implies that the intersection $D' \cap D_1$ is a single non-separating geodesic on each of these thrice-punctured spheres. Then Proposition \ref{prop:CaseI} with $D_1$ replacing $D$ shows that $L$ is either $P_n$ or $O_n$. However, $P_n$ does not contain a crossing circle $C$ whose crossing disk $D$ separates a region of $R$, which implies that $\mathcal{A}$ must be $O_n$ here. 

\textbf{Case II:} Suppose that at least one of the other two punctures of $D'$ comes from an $R$-knot circle, $K_1$. We will show this case leads to a contradiction. By Lemma \ref{lem:Reflection3ps}, $D'$ must have an odd number of $R$-knot circle punctures, and so, the third puncture of $D'$ comes from an $R$-knot circle distinct from $K$  and $K_1$, which we label $K_2$. Furthermore, since $\gamma_K$ is a separating geodesic on $D'$, $K_1$ and $K_2$ must puncture different regions of $D'\setminus D$. Any $R$-crossing disk punctured by $K_1$, call it $D_1$, is disjoint from $D$ and so must intersect $D'$ in geodesic(s) disjoint from $\gamma_{K} =D\cap D'$. The only geodesics of $D'$ disjoint from $D\cap D'$ are the non-separating geodesics $\{ \gamma_1, \gamma_2\}$ joining $K$ to the other punctures $K_1$ and $K_2$ respectively.  Since $D_1$ is   punctured by $K_1$, it's one other $R$-knot circle puncture must come from $K$, and so, $\gamma_1 = D_1\cap D'$. Now there is at most one connected, embedded, totally geodesic surface orthogonal to $R$ and containing $\gamma_1$; therefore, at most one crossing disk intersects $D'$ along $\gamma_1$.  This implies at most one crossing circle links $K_1$, contradicting the fact that every knot circle must be linked by at least two crossing circles in an FAL. 
\end{proof}

We can now give the following classification of flat FAL complements that admit multiple reflection surfaces. A slightly less general version of this theorem was originally stated in Theorem \ref{thm:MainTheorem2} in Section \ref{sec:intro}.

\begin{thm}
\label{thm:MultipleRS}
Suppose $M = \mathbb{S}^{3} \setminus \mathcal{A}$ is a flat FAL complement with multiple distinct reflection-like surfaces. Then either 
\begin{itemize}
\item  $\mathcal{A}$ is equivalent to the Borromean rings, and $M$ contains exactly three distinct reflection-like surfaces, all of which are reflection surfaces, or
\item  $\mathcal{A}$ is equivalent to $P_n$ with $n \geq 3$, or $O_n$ with $n \geq 2$, and $M$ contains exactly two distinct reflection-like surfaces, both of which are reflection surfaces. 
\end{itemize}
\end{thm}

\begin{proof} Suppose $M$ is a flat FAL complement with at least two distinct reflection-like surfaces. Then Lemma \ref{lem:samesurface} shows that some $R'$-crossing circle $K$ of $\mathcal{A}$ must switch to become an $R$-knot circle  (or vice versa). Let $D'$ be the $R'$-crossing disk corresponding to $K$. By Lemma \ref{lem:factoid}, $D' \subset R$. At the same time, since $K$ is an $R$-knot circle, it punctures at least two $R$-crossing disks, one of which we label $D$.  Then $D$ and $D'$ are both totally geodesic thrice-punctured spheres in $M$ that intersect non-trivially since $D$ must intersect $R$ on each side of $K$ (thinking of $K$ as a simple closed curve in the projection plane), and one of these components is $D'$.  The work of Yoshida \cite[Proposition 3.1]{Y2018}, shows that there are exactly three possibilities for such intersection patterns, which are covered in Proposition \ref{prop:CaseI}, Proposition \ref{prop:CaseII}, and Proposition \ref{prop:CaseIII}. Combined, these propositions tell us that $\mathcal{A}$ is equivalent to either the Borromean rings, $P_n$ with $n \geq 3$, or $O_n$ with $n \geq 2$. Theorem \ref{thm:3rs} distinguishes the Borromean rings as the only flat FAL whose complement admits at least three distinct reflection-like surfaces. In this case, this FAL complement has exactly three reflection-like surfaces, all of which are reflection surfaces, as noted in Theorem \ref{thm:3rs}; see Figure \ref{fig:BringsMRS} for a visualization of the three reflection surfaces. If an FAL complement has exactly two reflection-like surfaces, then the corresponding link is either $P_n$ with $n \geq 3$ or $O_n$ with $n \geq 2$. The discussion at the beginning of Section \ref{sec:2RS} shows that the corresponding FAL complements for these links each admit two distinct reflection surfaces, and so, every reflection-like surface is also a reflection surface in all of these cases. 
\end{proof}

This classification theorem implies that within the family of flat FALs, those whose complements admit multiple reflection surfaces are determined by their complements. We highlighted this result in the following corollary.

\begin{cor}\label{cor:MRSdetermined}
Let $\mathcal{A}$ and $\mathcal{A'}$ be flat FALs, and suppose the complement of $\mathcal{A}$ admits multiple distinct reflection surfaces. Then $\mathbb{S}^{3} \setminus \mathcal{A}$ is homeomorphic to $\mathbb{S}^{3} \setminus \mathcal{A'}$ if and only if $\mathcal{A}$ and $\mathcal{A'}$ are equivalent links. 
\end{cor}

\begin{proof}
If $\mathcal{A}$ and $\mathcal{A}'$ are equivalent links, the orientation-preserving homeomorphism between the pairs $(\mathbb{S}^3,\mathcal{A})$ and $(\mathbb{S}^3,\mathcal{A}')$ induces one between their complements. So, suppose $\mathbb{S}^{3} \setminus \mathcal{A}$ and $\mathbb{S}^{3} \setminus \mathcal{A}'$ are homeomorphic flat FAL complements  where $\mathbb{S}^{3} \setminus \mathcal{A}$ admits multiple reflection surfaces. Then these flat FAL complements are isometric and $\mathbb{S}^{3} \setminus A'$ also contains multiple reflection-like surfaces. By Theorem \ref{thm:MultipleRS}, $A'$ is equivalent to either the Borromean rings $B$, $P_n$ with $n \geq 3$, or $O_n$ with $n \geq 2$. Thus, we just need to consider the cases where $\mathbb{S}^{3} \setminus \mathcal{A}$ is homeomorphic to the complement of one of these links.  Note that $B$ has three components, each $P_n$ has $2n$ components with $n \geq 3$, and $O_n$ has $2n+1$ components with $n \geq 2$. Thus the number of components distinguishes the links $B$, $ \{ P_n \}_{n = 3}^{\infty} $ and $\{ O_n \}_{n = 2}^{\infty}$. Since there is a one-to-one correspondence between  components of a link and cusps of the corresponding link complement, this shows that the number of cusps of one of these link complements determines the corresponding link, completing the proof.
\end{proof}

The work from this section places an important restriction on the behavior of homeomorphisms between flat FAL complements. 

\begin{cor}
\label{thm:uniquereflectionhomeo}
Let $M$ and $M'$ be flat FAL complements and suppose there exists a homeomorphism $h: M \rightarrow M'$, which induces an isometry $\rho_{h}$. Then $R$ is a reflection-like surface for $M$ if and only if $R$ is a reflection surface for $M$. In particular, $\rho_h$  provides a one-to-one correspondence between reflection surfaces. 
\end{cor}

\begin{proof}
By the comments immediately following Definition \ref{defn:Rlike}, the isometry $\rho_h$ produces a one-to-one correspondence between reflection-like surfaces.  If we show that every reflection-like surface is actually a reflection surface, we will be done.

Theorem \ref{thm:MultipleRS} covers the multiple reflection-like surface case.  On the other hand, the reflection surface in a flat FAL is one reflection-like surface. Therefore, if a flat FAL has a unique reflection-like surface it must be the reflection surface, completing the proof.
\end{proof}

Before moving on, we make two useful observation that follow from Corollary \ref{thm:uniquereflectionhomeo}. First off, we will no longer use the term reflection-like surface since a reflection-like surface is a reflection surface. In addition, if $h: M \rightarrow M'$ is a homemorphism, $M$ has a unique reflection surface $R$, and $M'$ contains a reflection surface $R'$, then  $M'$ also has a unique reflection surface and $\rho_{h}(R) = R'$. 

Our final result from this section shows that in most cases, every symmetry of a flat FAL complement with multiple reflection surfaces is induced by a symmetry of that link. In particular, the following theorem proves the first statement from Theorem \ref{thm:SymmetryThm} in the introduction.

\begin{thm}\label{thm:SymmetryMRS}
Let $\mathcal{A}$ be a flat FAL, other than $P_3$, whose complement admits multiple reflection surfaces.  Then both $\mathcal{A}$ and its complement $M = \mathbb{S}^3\setminus \mathcal{A}$ have the same symmetry group.
\end{thm}

\begin{proof}
Since every symmetry of $(\mathbb{S}^{3}, \mathcal{A})$ induces one of its complement, it's enough to show that every homeomorphism $h:M\to M$ extends to an isotopy of $\mathbb{S}^3$.

Since $\mathcal{A}$ is not  $P_3$, Theorem \ref{thm:MultipleRS} implies it is either $P_n$ with $n \geq 4$,  $O_n$ with $n \geq 2$, or the Borromean rings, and we consider the cases separately.

Let $\mathcal{A} = O_n$, with $n\ge 2$, and let $R$ be a reflection surface in $M= \mathbb{S}^3\setminus O_n$.  Note that $R$ consists of $n$ thrice-punctured spheres and one $(n+2)$-punctured sphere $S_{n+2}^2$ whose punctures are $n$ longitudes along the $R$-knot circles of $O_n$ and $2$ punctures by the same $R$-crossing circle $C_0$. Since $n\ge 2$, $S^2_{n+2}$ is the only component of $R$ with more than three punctures.

Now let $h:M\to M$ be a homeomorphism with induced isometry $\rho_h:M\to M$. Corollary \ref{thm:uniquereflectionhomeo}, applied to the case where $M'=M$, implies that $R'=\rho_h(R)$ is a reflection surface for $O_n$.  Since $R'$ is a reflection surface for $O_n$, it has a unique $(n+2)$-punctured sphere component ${S^2_{n+2}}'$.  The unique cusp punctured twice by ${S^2_{n+2}}'$ corresponds to an $R'$-crossing circle while the remaining $n$ cusps punctured by ${S^2_{n+2}}'$ correspond to the $R'$-knot circles of $O_n$. Since $\rho_h$ preserves the topology of components of $R$ we have $\rho_h\left({S^2_{n+2}}\right) = {S^2_{n+2}}'$, and $R$-knot circles must map to $R'$-knot circles.  Moreover, $C_0' = \rho_h(C_0)$ must be the $R'$-crossing circle of $O_n$ puncturing ${S^2_{n+2}}'$ twice.  The remaining components of $O_n$ are $R$-crossing circles and, since all $R'$-knot circles are accounted for, $\rho_h$ must map them to $R'$-crossing circles.  

Thus $\rho_h$ preserves the type of each component of $O_n$--crossing circles map to crossing circles, and similarly with knot circles.  Let $L$ be a component of $O_n$ with image $L'=\rho_h(L)$, and let $\{m,\ell\}$ and $\{m',\ell'\}$ denote the respective $R$- and $R'$- meridian-longitude pairs.  Lemma \ref{lem:MeridianLongitude} implies $\rho_h$ maps the set $\{m,\ell\}$ to the $\{m',\ell'\}$, and we must show it takes meridians to meridians.

If $K$ is an $R$-knot circle of $O_n$, then $T_K\cap R$ is a pair of simple closed curves, each representing an $R$-longitude,  which we denoted by $\ell$ earlier.  As above, let $\{m',\ell'\}$ denote the $R'$ meridian and longitude for $T_{K'}= \rho_h(T_K)$.  Then $\rho_h(T_K\cap R) = T_{K'}\cap R'$, and so maps longitudinal slopes of $K$ to those of $K'$.  As above, Lemma \ref{lem:MeridianLongitude} implies that $\rho_h$ preserves meridians as well.  Thus $\rho_h$ preserves peripheral structures on all $R$-knot circles.

To see that $h$ preserves peripheral structures on $R$-crossing circles, let $C$ be an $R$-crossing circle with $R'$-crossing circle image $C'=\rho_h(C)$.  Using notation similar to the above, we have the calculation $\rho_h(T_C\cap R) = T_{C'}\cap R'$ so $\rho_h$ preserves meridional slopes, and Lemma \ref{lem:MeridianLongitude} implies it preserves peripheral structures on crossing circles of $O_n$ as well.

Thus $\rho_h$, and therefore our original $h:M\to M$, preserves peripheral structures on all components and extends to an isotopy of $\mathbb{S}^3$.  This implies, of course, that a symmetry of $M$ is the restriction of a symmetry of $(\mathbb{S}^{3}, O_n)$ to its complement.

The proof for $P_n$, with $n \ge 4$, follows similarly, but is a little more direct.  In this case the reflection surface $R$ consists of $n$ thrice punctured spheres and one $n$-punctured sphere $S_n^2$.  Since $n\ge 4$, the sphere $S^2_n$ is the unique component of $R$ with more than three punctures.  The proof follows as above, with the simplification that punctures of $S^2_n$ correspond to the distinct $R$-knot circles of $P_n$.

We have proven the theorem for flat FALs with two reflection surfaces and the only case remaining is the flat FAL with three reflection surfaces -- the Borromean rings. In this case, a quick check using SnapPy confirms the symmetry group of the Borromean rings and its complement coincide, with common symmetry group $\mathbb{Z}_2 \times G$, where $G$ represents the group of symmetries of the octahedron. 
\end{proof}

Note that the proof of Theorem \ref{thm:SymmetryMRS} does not extend to $P_3$. For $P_3$ all components of the reflection surface $R$ are thrice-punctured spheres, with one punctured by three knot circles. A simple puncture count, then, does not guarantee that the isometry $\rho_h$ preserves the component of $R$ punctured by knot circles.  Section \ref{Sec:SigLinks} will show that this is more than just a shortcoming of the above proof.  In fact $P_3$ is a signature link, and it will be seen that signature link complements have more symmetries than the links themselves.


\section{Transition to the unique reflection surface case}
\label{sec:Transition}

Corollary \ref{thm:uniquereflectionhomeo} tells us that homeomorphic flat FAL complements must have the same number of reflection surfaces. This allows us to break down the proof of our main result, Theorem \ref{MAINTHEOREM}, into two cases: 

(1) There exists a homeomorphism between flat FAL complements, each with multiple reflection surfaces. This case is already covered by Corollary \ref{cor:MRSdetermined}.  

(2) There exists a homeomorphism between flat FAL complements, each with a unique reflection surface. 

We can actually put a more narrow focus on the homeomorphisms we need to analyze in case (2). Let $M$ and $M'$ be the complements of the flat FALs $\mathcal{A}$ and $\mathcal{A}'$, respectively, each containing unique reflection surfaces denoted by $R$ and $R'$.  If there exists a homeomorphism $h:M\to M'$, which induces isometry $\rho_h$, then we must have $R' = \rho_{h}(R)$. Recall that the reflection surfaces determine peripheral structures on each component of $\mathcal{A}$ and $\mathcal{A}'$, respectively. Thus if $h$ preserves both knot circles and crossing circles, then it preserves peripheral structures and extends to an isotopy of $\mathbb{S}^3$, making the links $\mathcal{A}$ and $\mathcal{A}'$ equivalent.  For this reason we will be mainly interested in homeomorphisms that change a knot circle $K$ into a crossing circle, or vice-versa.  In this case we will say that $h$ \textbf{changes the type} of a component $K$ of $\mathcal{A}$, and will call $h$ a \textbf{type-changing homeomorphism}. 

The rest of this paper is dedicated to analyzing type-changing homeomorphisms of flat FAL complements containing unique reflection surfaces. Section \ref{Sec:SigLinks} will introduce a particular class of type-changing homeomorphisms where we can easily find an isotopy in $\mathbb{S}^{3}$ between the corresponding FALs. Section \ref{subsec:SepSets} examines how separating sets (partially introduced in Section \ref{sec:FALs}) behave under type-changing homeomorphisms to help restrict the behavior of such homeomorphisms. Combining the work of these two sections, we then prove our main result in Section  \ref{sec:CDFF}.

\section{Signature links and full-swap homeomorphisms} \label{Sec:SigLinks}

In this section we define signature links, which are a special class of FALs, and show that they admit a particular type-changing homeomorphisms, which we call a full-swap, on their complements.  Furthermore, we will show that signature links whose complements are homeomorphic via a full-swap exhibit an explicit isotopy  in $\mathbb{S}^{3}$.

\begin{df}\label{defn:SigLink} 
A \textbf{signature link} is a flat FAL $\mathcal{L}$ whose components can be partitioned into four non-empty sets
\[
\mathcal{L}= \{K_f\} \cup \mathcal{K}\cup\mathcal{C}\cup\mathcal{C_K},
\]
with the following properties. The first set consists solely of a knot circle $K_f$ which cuts the projection plane into two disks, one of which contains the remaining knot circles $\mathcal{K} = \left\{K_1,\dots,K_n\right\}$, which we call the inside of $K_f$.  Each $K_i\in\mathcal{K}$ is linked with $K_f$ by a unique crossing circle $C_i$, and $\mathcal{C} = \left\{C_1,\dots,C_n\right\}$.  The set of remaining crossing circles is denoted $\mathcal{C_K}$, and components in $\mathcal{C_K}$ link two distinct knot circles in $\mathcal{K}$.  A crossing circle of $\mathcal{C_K}$ that links $K_i,K_j\in\mathcal{K}$ will be denoted by $C_{ij}$.
\end{df}

Throughout this section, we let $D_i$ designate a crossing disk for $C_i \in \mathcal{C}$ and we let $D_{ij}$ designate a crossing disk for $C_{ij} \in \mathcal{C_K}$.

\begin{figure}[h]
\[
\begin{array}{ccc}
\includegraphics{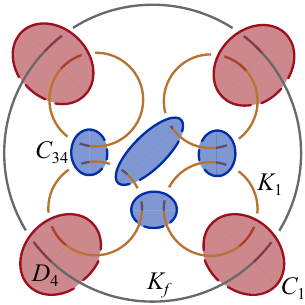} &\hspace{.65in} &\includegraphics{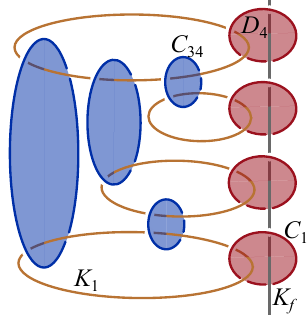}\\
(a)\textrm{ Aesthetic diagram} && (b)\textrm{ Pragmatic diagram}
\end{array}
\]
\caption{Two diagrams of a signature link $\mathcal{L}$.}
\label{fig:SigLink}
\end{figure}

Figure \ref{fig:SigLink}$(a)$ depicts a signature link.  The knot circles  in $\mathcal{K}$ are all inside $K_f$, and are numbered counterclockwise around $K_f$ (only $K_1$ is labeled in Figure \ref{fig:SigLink}$(a)$).  The diagram of Figure \ref{fig:SigLink}$(b)$ will be helpful in visualizing a full-swap, and is obtained by isotoping $K_f$ until it is vertical.  We remark that there can be more than one way to decompose $\mathcal{L}$ as a signature link.  The knot circle $K_4$ (unlabeled in Figure \ref{fig:SigLink}) could have been designated $K_f$ instead since it is linked to every other knot circle.  

According to Definition \ref{defn:SigLink} the link $P_3$ is a signature link, but a quick check verifies that this is the only link whose complement contains multiple reflection surfaces which is a signature link.  Thus all other signature link complements have a unique reflection surface. 

The fact that $\mathcal{L}$ is hyperbolic places restrictions on the set $\mathcal{C_K}$.  The set $\mathcal{C_K}$, for example, can not be empty.  More can be said, of course, about properties of $\mathcal{C_K}$ resulting from the hyperbolicity of $\mathcal{L}$, but we content ourselves with a result about longitudinal disks which requires the following technical lemma.

\begin{lem}\label{lem:NdisksDisjoint}
Two $N$-disks in an FAL complement are either identical or disjoint.
\end{lem}

\begin{proof}
To see this, we show that if two $N$-disks intersect, then they are identical.  Let $D_1,D_2$ be $N$-disks in an FAL complement $M$, and so, they each intersect the reflection surface $R$ in their non-separating geodesics.  Lemma 3.4 of \cite{Y2018} states that thrice-punctured spheres in orientable three-manifolds cannot intersect along a geodesic that is separating in both.  Thus, if $\gamma \in D_1\cap D_2$, then $\gamma$ is non-separating in at least one disk, say $D_1$.  The disk $D_1$ is an $N$-disk so $\gamma$ is contained in the reflection surface $R$ and $D_2$ intersects $D_1$ along a geodesic in $R$. By Theorem \ref{prop:BeltSumSummary} both $D_1$ and $D_2$ are orthogonal to $R$ along $\gamma$. Therefore, since $D_1$ and $D_2$ are both embedded totally geodesic surfaces that intersect in a common geodesic and are orthogonal to $R$, we can conclude that they must be equal, completing the proof.
\end{proof}

\begin{lem}\label{lem:SigLong}
Let $\mathcal{L}$ be a signature link.  Then every crossing circle $C_{ij}\in\mathcal{C_K}$ bounds a totally geodesic, longitudinal disk with crossing circles $C_i$ and $C_j$.
\end{lem}

\begin{proof}
Let $\mathcal{L}$ be a signature link with complement $M = \mathbb{S}^3\setminus \mathcal{L}$. Given a crossing circle $C_{ij}\in\mathcal{C_K}$, we will construct a longitudinal disk with punctures $C_i,C_j,C_{ij}$  by gluing two disks, which are essentially topological descriptions of the geodesic disks described in \cite{mrstz}. Afterwards, we will show this longitudinal disk is totally geodesic. 

A knot circle $K_i\in \mathcal{K}$ of a signature link $\mathcal{L}$ bounds two disks in the projection plane, and we define the \emph{inside} of $K_i$ to be the disk $\mathcal{P}_i$ that does not contain $K_f$.  Thus $\mathcal{P}_i$ is punctured once by $C_i$ and once by each crossing circle of $\mathcal{C_K}$ that links $K_i$.  

Now consider how crossing disks intersect $\mathcal{P}_i$.  Since all crossing circles in $\mathcal{L}$ link distinct knot circles, and since $K_i$ is the only knot circle puncture of $\mathcal{P}_i$, no crossing disk intersects $\mathcal{P}_i$ in a geodesic arc with both endpoints on $K_i$.  Crossing disks that intersect $\mathcal{P}_i$, then, do so in a geodesic joining a crossing circle puncture to the boundary curve $K_i$.  Further, since crossing disks are disjoint they intersect $\mathcal{P}_i$ in disjoint arcs (see Figure \ref{fig:In}$(a)$). The complement of crossing disks in $\mathcal{P}_i$ is then connected and there is an embedded arc, disjoint from crossing disks, between any two crossing circle punctures of $\mathcal{P}_i$.   

\begin{figure}[h]
\[
\begin{array}{ccc}
\includegraphics{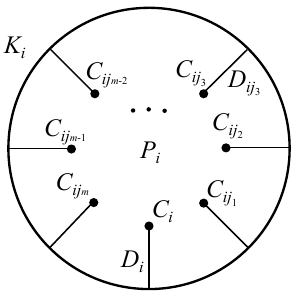} &\hspace{0.25in}& \includegraphics{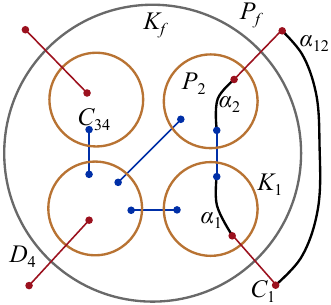}\\
(a) \textrm{ Crossing disks in }\mathcal{P}_i & &(b)\  \alpha\textrm{ curves for }D_{\ell}
\end{array}
\]
\caption{Crossing disks and a longitudinal disk intersecting the reflection surface.}
\label{fig:In}
\end{figure}

Similarly, the knot circle $K_f$ divides the projection plane into two topological disks, one punctured by the knot circles of $\mathcal{K}$ and the other by the crossing circles of $\mathcal{C}$. The \emph{outside} of $K_f$, denoted $\mathcal{P}_f$, refers to the disk punctured by the crossing circles of $\mathcal{C}$.  The argument of the previous paragraph shows that there is an embedded arc in $\mathcal{P}_f$ joining any two crossing circle punctures that is disjoint from crossing disks of $\mathcal{L}$.

Given a crossing circle $C_{ij}\in\mathcal{C_K}$ we construct a longitudinal disk $D_{\ell}$, with punctures $C_i, C_j, C_{ij}$, by gluing a disk $D_+$ above the reflection surface to its reflection $D_-$.  We begin by describing the boundary of $D_+$.  Let $\alpha_i\subset \mathcal{P}_i$ be an embedded arc disjoint from crossing disks that joins $C_i$ and $C_{ij}$ punctures, and define $\alpha_j$ similarly. Also let $\alpha_{ij}$ denote an embedded arc in $\mathcal{P}_f$ which is disjoint from crossing disks and joins the $C_i$ and $C_j$ punctures. Figure \ref{fig:In}$(b)$, for example, illustrates the $\alpha$ arcs corresponding to the crossing circle $C_{12}$ of the signature link in Figure \ref{fig:SigLink}.

The arcs $\alpha_i,\alpha_j,\alpha_{ij}$, together with the top halves of the crossing circles $C_i, C_j, C_{ij}$, form a simple closed curve $\gamma$ in $\mathbb{S}^3$.  Let $M^+$ be the region of $M$ above the reflection surface and we wish to show that $\gamma$ bounds a disk in $M^+$. First, $M^+$ is a handlebody, since it is a three-ball with  arcs removed for each crossing circle.  Further, the top half of each crossing disk is a meridional disk for each handle.  Thus removing an open neighborhood of each crossing disk from $M^+$ results in a three ball with $\gamma$ in its boundary.  The curve $\gamma$, then, bounds a disk $D_+$ in $M^+$.  The disk $D_+$ is disjoint from crossing disks and intersects the reflection surface along the $\alpha$ arcs in its boundary.  Let $D_-$ be the reflection of $D_+$, and let $D_{\ell} = D_+\cup D_-$.  Then the crossing circles $C_i, C_j, C_{ij}$ form the boundary of $D_{\ell}$ in $\mathbb{S}^3$, and the interior of $D_{\ell}$ is an embedded thrice-punctured sphere in $M$ with longitudinal punctures along the crossing circles $C_i, C_j, C_{ij}$, as desired.

To see that $D_{\ell}$ is totally geodesic it is enough to show that it is incompressible and boundary incompressible, by Theorem 3.1 of \cite{ad1}.  The proof given here is essentially that of \cite[Lemma 2.1]{P2011}, but slightly simpler because the punctures of $D_{\ell}$ are distinct.  Suppose $\alpha$ is a curve in $D_{\ell}$ that bounds a compressing disk $D\subset M\setminus D_{\ell}$.  Then $\alpha$ separates $D_{\ell}$ into two pieces, one of which contains a single puncture $C$ of $D_{\ell}$.  Thus $\alpha \cup C$ bound an annulus in $D_{\ell}$ whose union with $D$ is a boundary compressing disk for the crossing circle $C$, contradicting the fact that $M$ is hyperbolic. 

When discussing $\partial$-incompressibility it will be convenient to think of $M$ as the interior of a orientable, closed three-manifold $\overline{M}$ with torus boundary components.  In this case, $D_{\ell}$ is properly embedded in $\overline{M}$ with longitudinal boundary curves along $T_i,T_j,T_{ij}$, the boundary tori of $\overline{M}$ corresponding to cusps $C_i, C_j, C_{ij}$ of $M$. For convenience, and by an abuse of notation, we let $C_i, C_j, C_{ij}$ denote the boundary curves of $D_{\ell}$ as well.  

Now suppose that $D$ is a $\partial$-compressing disk for $D_{\ell}$. Then $\partial D = \alpha\cup \beta$ where $\alpha = D\cap D_{\ell}$ is an arc in $D_{\ell}$ with both endpoints on the same boundary curve, say $C_i$, and $\beta$ is an arc in $T_i$.  The arc $\alpha$ is isotopic to a separating geodesic of $D_{\ell}$, which decomposes $D_{\ell}$ into two annuli, and we let $A$ denote the one containing $C_j$.  The other boundary component of the annulus $A$ consists of the arc $\alpha$ together with a subarc of $C_i$, call it $\delta$. Gluing the $\partial$-compressing disk $D$ to $A$ along their intersection $\alpha$ yields another annulus $A\cup D$ with $C_j$ as one boundary component.  The other boundary component of $A\cup D$ is the simple closed curve $\beta\cup\delta$ on the boundary torus $T_i$.

If $\beta\cup\delta$ bounds a disk on $T_i$, a copy of it in $M$ caps off one boundary component of $A\cup D$, creating a $\partial$-compressing disk for $C_j$, which is impossible.  On the other hand, suppose $\beta\cup\delta$ is non-trivial on $T_i$.  Then $A\cup D$ is an incompressible annulus which is not boundary parallel since its boundary curves are on separate boundary components of $\overline{M}$.  Again, this contradicts the fact that $M$ is hyperbolic. 

Thus $D_{\ell}$ is incompressible and $\partial$-incompressible and it has a totally geodesic representative by \cite[Theorem 3.1]{ad1}.
\end{proof}

The remainder of this section is devoted to constructing a ``full-swap", which is a type-changing homeomorphism of the complement of a signature link $\mathcal{L}$.  Full-swaps, despite changing the types of some components of $\mathcal{L}$, will be shown to produce a link equivalent to $\mathcal{L}$. 

To begin, we define an $ml$-swap homeomorphism on a  a flat FAL complement $M = \mathbb{S}^{3} \setminus \mathcal{A}$ in terms of Dehn twists on a Hopf sublink $\mathcal{H} \subset \mathcal{A}$. In particular, the Hopf sublink will consist of a knot and crossing circle pair that are linked.  Zevenbergen, in \cite{Z2021}, first constructed a product that exchanged meridional and longitudinal slopes on each component of $\mathcal{H}$ and was able to analyze the effect on the other components of $\mathcal{A}$.  We review his construction here, then provide an alternative cut-and-paste construction which highlights how $ml$-swaps effect reflection surfaces and crossing disks.

Figure \ref{fig:DehnTwists} illustrates an $ml$-swap on the Hopf sublink determined by the knot- and crossing-circle pair labeled $\{K,C\}$.  Before describing the Dehn twists we observe some features of $\mathcal{A}$ relative to $\mathcal{H}$.  Since $\mathcal{H}= K\cup C$ is a Hopf link, the crossing circle $C$ must link distinct knot circles and we label the other one $J$.  Let $D$ be a crossing disk for $C$ and let $N$ be an open regular neighborhood of the cell complex $K\cup C \cup D$ in $\mathbb{S}^3$.  Then $N$ is an unknotted open solid torus, so $W = \mathbb{S}^3\setminus N$ is a closed unknotted solid torus in $\mathbb{S}^3$.  The neighborhood $N$ can be chosen so that $N\cap \mathcal{A}$ contains the components $K$ and $C$, and an arc of $J$ that intersects $D$ (see Figure \ref{fig:DehnTwists}$(a)$). Then $W$ contains all components of $\mathcal{A}\setminus (K\cup C\cup J)$ together with one arc of $J$.

\begin{figure}[h]
\begin{center}
\includegraphics{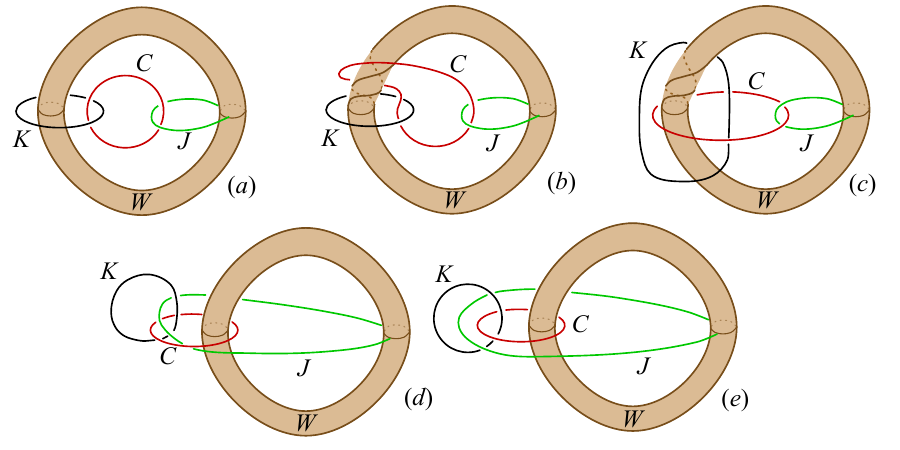}
\end{center}
\caption{An $ml$-swap as a product of Dehn twists}
\label{fig:DehnTwists}
\end{figure}

We now describe Zevenbergen's Dehn twists that make up an $ml$-swap and, abusing notation, we refer to components by their original labels throughout the Dehn twist process.  First perform a Dehn twist along $K$, which adds a full twist to $W$ and links $C$ around $W$ as in Figure \ref{fig:DehnTwists}$(b)$.  For simplicity, isotope $C$ so that it is flat, twisting $J$ and $K$ vertical in the process, to get Figure \ref{fig:DehnTwists}$(c)$.  Now perform a Dehn twist along $C$ that untwists $W$ so that it is returned to its original form.  This twist unlinks $K$ and $W$ while linking the arc of $J$ with both $W$ and $K$, as in Figure \ref{fig:DehnTwists}$(d)$.  Finally, perform a Dehn twist on $K$ that unlinks $C$ and $J$, obtaining the link depicted in Figure \ref{fig:DehnTwists}$(e)$.  This composition of Dehn twists is an $ml$-swap.  The result will not be another flat FAL in general, as observed in \cite{Z2021}, but we will see that performing multiple $ml$-swaps on signature links can produce a flat FAL.
 
Note that this product of Dehn twists is ultimately a local operation in the sense that changes to the link occur within a 3-ball containing the Hopf sublink. Moreover, if there are multiple Hopf sublinks that are contained in disjoint three balls, then the result of performing Dehn twists on each Hopf sublink is independent of the order in which they are done. With this background in place we make the following definition.

\begin{df}\label{defn:MLSwap}
An \textbf{$ml$-swap} on a Hopf sublink  $\mathcal{H}$ of a flat FAL $\mathcal{A}$ is the homeomorphism resulting from performing the Dehn twists just described above on the components of  $\mathcal{H}$.    A    \textbf{full-swap} on a signature link $\mathcal{L}$ is the composition of all $ml$-swaps on Hopf sublinks $K_i\cup C_i$, for $i=1, \ldots, n$.
\end{df}

Dehn twists provide a convenient description of an $ml$-swap, and we now consider an alternative description that highlights the effect of an $ml$-swap on the reflection surface and crossing disks involved. Figure \ref{fig:mlSwap}$(a)$ highlights a crossing circle $C$ and knot circle $K$ whose meridians and longitudes will be swapped.  The other knot circle linked by $C$ is included to emphasize how the reflection surface moves, but the rest of the link is not pictured.  The homeomorphism maps $C$ to the knot circle $C'$ and $K$ to the crossing circle $K'$ depicted in Figure \ref{fig:mlSwap}$(g)$.  Note that an $ml$-swap preserves the reflection surface, while moving the location of the component $R_0$ to that of $R_0'$. 
 
\begin{figure}[h]
\begin{center}
\includegraphics{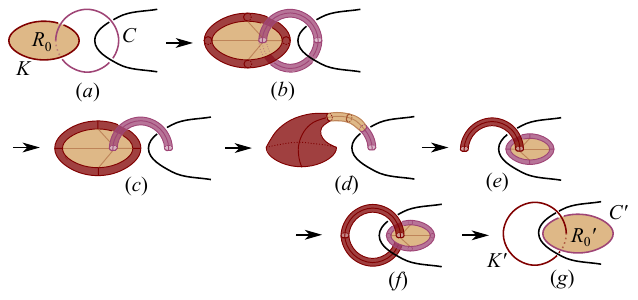}
\end{center}
\caption{An $ml$-swap}
\label{fig:mlSwap}
\end{figure}

Let us walk through the homeomorphism one step at a time.  In Figure \ref{fig:mlSwap}$(b)$, torus neighborhoods of $C$ and $K$ are pictured to emphasize what happens to meridians and longitudes.  Now slice the manifold along the reflection surface and consider the top half pictured in Figure \ref{fig:mlSwap}$(c)$, which is a handlebody $H_+$ (the homeomorphism on the bottom half is the reflection of that pictured).  In Figure \ref{fig:mlSwap}$(c)$, the half-tori around $K$ and $C$, as well as the copy of $R_0$, form three annuli. The middle row of Figure \ref{fig:mlSwap} depicts an isotopy sliding the three annuli along a handle of $H_+$.  Note that in the process meridians and longitudes of $K$ and $C$ are swapped.  The final row depicts regluing the isotoped $H_{\pm}$ along $R$, and finally removing the torus neighborhoods of $K'$ and $C'$.  In terms of peripheral structures on cusps of $M$, an $ml$-swap swaps meridians and longitudes on cusps corresponding to $C$ and $K$, and leaves the remaining structures the same.  

It is instructive to consider the image under an $ml$-swap of crossing disks punctured by the knot circle involved. The image of a crossing disk $D$ bounded by $C$ is the natural first choice to consider.  Figure \ref{fig:CrossDiskRot} illustrates that $D$ is sliced in half, each half rotated by ``a third", then re-glued along its non-separating geodesics. A similar rotation is done on the bottom half, so swapping the types of $C$ and $K$ has the effect of ``rotating" $D$.

\begin{figure}[h]
\begin{center}
\includegraphics{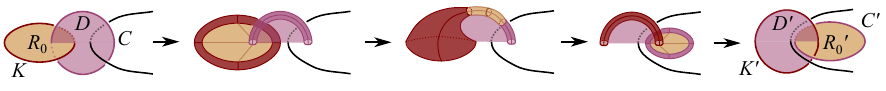}
\end{center}
\caption{Rotating a crossing disk in an $ml$-swap}
\label{fig:CrossDiskRot}
\end{figure}

Let $C^{\ast}$ be a crossing circle other than $C$ linking $K$. Let $D^{\ast}$ a crossing disk bounded by $C^{\ast}$, and consider the image of $D^{\ast}$ under the $ml$-swap. Since the $K$-puncture of $D^{\ast}$ becomes a crossing circle puncture, its image ${D^{\ast}}'$ has two crossing circle punctures. This is illustrated in Figure \ref{fig:CrossDiskSlide}. Since an $N$-disk in an FAL complement has either one or three crossing circle punctures, we know ${D^{\ast}}'$ is not a crossing disk, assuming the image of $M = \mathbb{S}^{3} \setminus \mathcal{A}$ under this $ml$-swap is a flat FAL complement.

\begin{figure}[h]
\begin{center}
\includegraphics{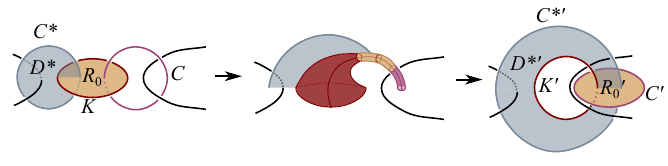}
\end{center}
\caption{Sliding a crossing disk in an $ml$-swap}
\label{fig:CrossDiskSlide}
\end{figure}

In fact, an $ml$-swap on a flat FAL can result in a link that is not an FAL (see \cite{Z2021}).  Performing $ml$-swaps on all possible $(C_i, K_i)$ pairs in a signature link (i.e. a full-swap), however, does produce another flat FAL.  Consider, for example, the simplest signature link: the chain $P_3$ in Figure \ref{fig:TwoMLswaps}$(a)$.  A full-swap homeomorphism $h_{f}$ is realized by performing successive $ml$-swaps on the Hopf sublinks $C_1\cup K_1$ and $C_2\cup K_2$, which yields the sequence of Figure \ref{fig:TwoMLswaps}.  

\begin{figure}[h]
\[
\begin{array}{ccc}
\includegraphics{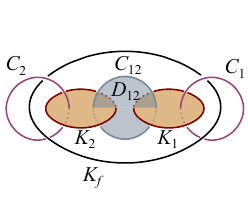}&\includegraphics{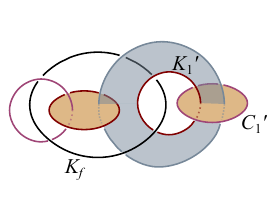}&\includegraphics{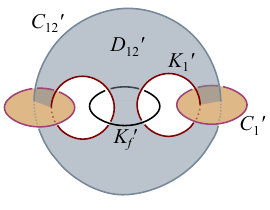}\\
(a)\textrm{ Original FAL} & (b)\textrm{ Non-FAL after one $ml$-swap} & (c)\textrm{ FAL isotopic to original}\\
&&\textrm{after two $ml$-swaps}
\end{array}
\]
\caption{Two $ml$-swaps}
\label{fig:TwoMLswaps}
\end{figure}

Note that the image of $P_3$ is isotopic to $P_3$. Proposition \ref{prop:MLswaps} will show that this holds more generally--that the image of a signature link under a full-swap is always isotopic to the original link.  While full-swaps produce equivalent links, they do interchange some crossing- and longitudinal-disks. For example, in Figure \ref{fig:TwoMLswaps}$(c)$, the image of the crossing disk $D_{12}$ is the longitudinal disk $D_{12}'$. In addition, the longitudinal disk with punctures $\{C_1,C_2,C_{12}\}$ of Figure \ref{fig:TwoMLswaps}$(a)$ becomes the crossing disk that $C_{12}'$ bounds (neither are pictured, but reading Figure \ref{fig:TwoMLswaps} backwards illustrates the change from longitudinal to crossing disk).  The proof of Proposition \ref{prop:MLswaps} will show that a full-swap on an arbitrary signature link interchanges crossing- and longitudinal-disks for every crossing circle in $\mathcal{C}_{\mathcal{K}}$.

The pragmatic diagram of Figure \ref{fig:SigLink}$(b)$ will be more convenient for the proof of Proposition \ref{prop:MLswaps}, so we assume $K_f$ is vertical and crossing circles of $\mathcal{C}$ are ordered from bottom to top.  Figure \ref{fig:TypeChangeHomeo} illustrates the effect of a full swap on such a diagram of a signature link.  Do note that $h_f$  \emph{does not} map the crossing disk for $C_{24}$ to that of $C'_{24}$ but to the longitudinal disk with punctures $K_2'$, $K_4'$, and $C'_{24}$. This happens on a more general scale and will be justified in the following proof. 


\begin{figure}[h]
\begin{center}
\includegraphics{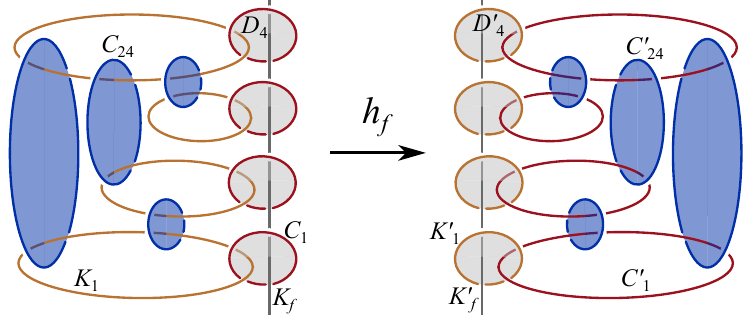}
\end{center}
\caption{The type-changing homeomorphism $h_{f}$ applied to a signature link}
\label{fig:TypeChangeHomeo}
\end{figure}

\begin{prop}\label{prop:MLswaps}
Let $\mathcal{L}$ be a signature link and let the homeomorphism $h_{f}$ designate the full swap on all Hopf sublinks $C_i\cup K_i$. Then $\mathcal{L}$ and $h_{f}\left(\mathcal{L}\right)$ are isotopic links. In particular, $h_{f}\left(\mathcal{L}\right)$ is a signature link.
\end{prop}

\begin{proof}
Let $\mathcal{L} = \{K_f\} \cup \mathcal{K}\cup\mathcal{C}\cup\mathcal{C_K}$ be a signature link with  $\mathcal{K}$ and $\mathcal{C}$ each containing $n$ components. Since each ml-swap maps a link in $\mathbb{S}^{3}$ to a link in $\mathbb{S}^{3}$, we have that $h_{f}\left(\mathcal{L}\right)$ is a link in $\mathbb{S}^{3}$ by construction. We will  describe how to build $h_{f}(\mathcal{L})$ by performing a full-swap on the pragmatic diagram of $\mathcal{L}$, where $K_{f}$ corresponds with  the z-axis  and so that each $C_j$ is contained in a vertical translate of the $xy$-plane at height $z=j$, as depicted on the left side of Figure \ref{fig:TypeChangeHomeo}.   Here, the $yz$-plane corresponds with the projection plane for an FAL diagram of $\mathcal{L}$ with each $C_{ij} \in \mathcal{C_K}$ linking only $K_i$ and $K_j$ and meeting the $yz$-plane orthogonally. 

Now, consider the sublink $\mathcal{L}_s  = \{K_f\} \cup \mathcal{K}\cup\mathcal{C}$ and its image $h_{f}(\mathcal{L}_s)$. Recall that $h_{f}$ is the composition of $n$ ml-swaps, one performed on each Hopf sublink $(C_i, K_i) \in \mathcal{C} \times \mathcal{K}$ for $i =1, \ldots, n$; see Figure \ref{fig:mlSwap} for the local picture of a single ml-swap. Then $h_{f}(\mathcal{L}_s)$ is constructed from  $\mathcal{L}_s$ by keeping $K_f$ fixed as the $z$-axis, while each Hopf sublink $h_{f}(K_j \cup C_j)$ links $K_{f}' = h_{f}(K_f)$ via $K_{j}' = h_{f}(K_j)$. In addition, each $K_{j}'$ is contained in the vertical translate of the $xy$-plane at height $z=j$, as depicted on the right in Figure \ref{fig:TypeChangeHomeo}.

We will now show that the images of the crossing circles in $\mathcal{C}_{\mathcal{K}}$ under $h_f$ are unlinked unknots, which will help us determine how $h_{f}(\mathcal{C}_{\mathcal{K}})$ behaves. Every component of $\mathcal{L}$ is unknotted, and links those components used in the Dehn twists of an $ml$-swap at most once.  In this situation, the Dehn twists never knot an unknotted component so the image of each component in $\mathcal{L}$ is an unknot as well. Further, an $ml$-swap does not link two crossing circles of $\mathcal{C_K}$ because the linking introduced by the first Dehn twist along $K$ in Figure \ref{fig:DehnTwists} is undone by the following twist along $C$. In addition, since the ml-swaps that comprise a full-swap occur in disjoint $3$-balls, the same applies to a full-swap.

Now unknots in $\mathbb{S}^3$ have a canonical peripheral structure in which a meridian links the component once and a longitude bounds an embedded disk in its complement.  An $ml$-swap preserves this $\mathbb{S}^3$-peripheral structure on all components of $\mathcal{L}$ except $K$ and $C$, for which it swaps meridians and longitudes.  This observation allows us to discuss the \emph{topology} of the images of thrice-punctured spheres under a full swap by analyzing images of their punctures.

A crossing disk $D_{ij}$ has a longitudinal puncture along $C_{ij}$ and meridional punctures along $K_i$ and $K_j$.  Since meridians of $K_i, K_j$ map to $\mathbb{S}^3$-longitudes of $K_i', K_j'$, the image $D_{ij}'$ of $D_{ij}$ under a full swap is a thrice-punctured sphere with longitudinal punctures along each of $C_{ij}'$, $K_i'$ and $K_j'$.  Similarly, the longitudinal disk $D_{ij}^{\ell}$ (guaranteed by Lemma \ref{lem:SigLong}) has longitudinal punctures along $C_{ij}$, $C_i$ and $C_j$.  A full swap maps longitudes of $C_i$ and $C_j$ to meridians in the $\mathbb{S}^3$-peripheral structure of $C_i'$ and $C_j'$, so ${D_{ij}^{\ell}}'$ has a $\mathbb{S}^3$-longitudinal puncture along $C_{ij}'$, and $\mathbb{S}^3$-meridional punctures along $C_i'$ and $C_j'$.  The component $C_{ij}'$, then, bounds an embedded disk in $\mathbb{S}^3$ punctured once by each of $C_i'$ and $C_j'$.  This implies $C_{ij}'$ links only $C_i'$ and $C_j'$.


Using these representatives for $\mathcal{L}$ and $h_{f}(\mathcal{L})$, we see that a rotation along the $z$-axis by $180^{\circ}$ provides the necessary isotopy between $\mathcal{L}$ and $h_{f}(\mathcal{L})$ in $\mathbb{R}^{3} \cup \{ \infty \} \cong \mathbb{S}^{3}$.   Furthermore, $h_{f}(\mathcal{L})$ is a signature link $\{K_{f}'\} \cup \mathcal{K}' \cup \mathcal{C}' \cup \mathcal{C}_{\mathcal{K}}'$, where $K_{f}' = h_{f}(K_{f})$, $\mathcal{K}' = h_{f}(\mathcal{C})$, $\mathcal{C}' = h_{f}(\mathcal{K})$, and $\mathcal{C}_{\mathcal{K}}' = h_{f}(\mathcal{C}_{\mathcal{K}})$. \end{proof}

\section{Separating Sets}
\label{subsec:SepSets}

In this section, we will utilize two important subsets of thrice-punctured spheres whose removal separates an FAL complement: separating pairs and separating quadruples. The behavior of type-changing homeomorphisms between flat FAL complements is significantly restricted by the existence of separating sets.

Separating pairs were introduced in Section \ref{sec:FALs} as a pair of disjoint thrice-punctured spheres whose union separates an FAL complement. Theorem \ref{thm:SepPair} from \cite{mrstz} was also introduced in that section, which states that a pair of thrice-punctured spheres $\{ S_1, S_2 \}$  is a separating pair if and only if each is either a crossing disk or a singly-separated disk and their longitudinal slopes coincide.




We now introduce a particular type of separating set consisting of four thrice-punctured spheres.  Let $D$ be a longitudinal disk with longitudinal punctures along the crossing circles $C_1,C_2,C_3$, and let $D_i$ be a crossing disk with crossing circle puncture $C_i$.  Then the set $Q=\{D,D_1,D_2,D_3\}$ is a separating set of four thrice-punctured spheres. The sets we are concerned with have one additional property.

\begin{df}\label{defn:SepQuad}
Let $D$ be a longitudinal disk in a flat FAL complement $M$ with crossing circle punctures $C_1, C_2$, and $C_3$ that bound crossing disks $D_1,D_2$ and $D_3$.  Then $Q = \{D,D_1,D_2,D_3\}$ is a \textbf{separating quadruple} if each crossing disk $D_i$ is punctured by distinct knot circles. 
\end{df}

The signature links of Section \ref{Sec:SigLinks} contain separating quadruples. To see this, note that in a signature link every crossing circle links distinct knot circles. Moreover, Lemma \ref{lem:SigLong} guarantees the existence of a longitudinal disk $D^{\ell}_{ij}$ for every $C_{ij} \in \mathcal{C}_{\mathcal{K}}$. Thus, each triple $(C_i, C_j, C_{ij})$ in a signature link generates a separating quadruple $Q_{ij}$.  The remark below will show that $Q_{ij}$ is unique up to a choice of crossing disks.

We introduce some terminology. A general separating quadruple $Q$ is illustrated in Figure \ref{fig:SeparatingQuadruples}, which depicts only those components which puncture $Q$. In a flat FAL knot circles cannot cross each other, so knot circles puncture adjacent crossing disks. Let $K_i$ denote the knot circle puncturing $Q$ that is \textbf{opposite} the crossing circle $C_i$ in the sense that they are not linked. By an abuse of terminology, a \textbf{component of (or in) $Q$} will refer to components that puncture disks in $Q$. 

\textbf{Remark:} We also point out that two separating quadruples with the same crossing circle punctures have the same knot circle punctures and longitudinal disk.   Suppose $Q$ and $Q'$ are separating quadruples with the same crossing circle punctures.  Lemma 4.2 of \cite{mrstz} shows that there is at most one longitudinal disk containing any two given crossing circle punctures, let alone three, so $Q$ and $Q'$ have the same longitudinal disk.  Now recall that crossing circles in a separating quadruple link distinct knot circles, and only those components of $\mathcal{A}$.  This implies that every disk they bound is punctured by the same two knot circles.  Hence two separating quadruples sharing the same crossing circle punctures can differ only in their crossing disks.


\begin{figure}[h]
\begin{center}
\includegraphics{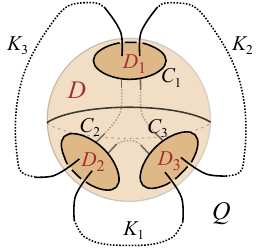}
\end{center}
\caption{A separating quadruple}
\label{fig:SeparatingQuadruples}
\end{figure}

We begin with the following lemma which immediately leads to a special case, Corollary \ref{cor:TwoBridge}, of our main result. This lemma will also be essential for establishing some technical results on how separating quadruples can behave under a type-changing homeomorphism. 

\begin{lem}\label{lem:SameKCPSameType}
Let $M=\mathbb{S}^3\setminus \mathcal{A}$ flat FAL complement with a unique reflection surface, and let $C$ a crossing circle in $\mathcal{A}$ that links the same knot circle $K$ twice.  If $h:M\to M'$ is a homeomorphism of flat FAL complements, then $h$ preserves the types of both $C$ and $K$.
\end{lem}

\begin{proof}
Let $C$ be a crossing circle in a flat FAL complement $M$ containing a unique reflection surface $R$. Further, let $C$ bound a crossing disk $D$ which is punctured twice by the same knot circle $K$, and suppose $h:M\to M'$ is a homeomorphism of flat FAL complements.  The assumption that the reflection surface $R\subset M$ is unique implies that $M'$ has a unique reflection surface $R'$, and that $R' = h(R)$.

We first prove that $K' = h(K)$ must be a knot circle in $M'$.  Suppose, on the contrary, that $K'$ is a crossing circle in $M'$. Then meridians of $K$ map to longitudes of $K'$ because both are perpendicular to reflection surfaces and $R' = h(R)$.  The fact that $R' = h(R)$ further implies that $D' = h(D)$ is a non-reflection thrice-punctured sphere in $M'$, since $D$ is in $M$.  The two meridional $K$-punctures of $D$ map to two longitudinal $K'$-punctures of $D'$. However, by Theorem \ref{prop:BeltSumSummary}, non-reflection thrice-punctured spheres in a flat FAL complement do not have two longitudinal punctures along the same crossing circle. So, $K'$ must be a knot circle in $M'$. 

Since $K'$ is a knot circle puncturing the disk $D'$ twice in $M'$, the remaining puncture of $D'$ must be an $M'$ crossing circle by the characterization of Theorem \ref{prop:BeltSumSummary}. The remaining puncture of $D'$ is the image $C' = h(C)$ of $C$, and $h$ preserves the types of both $C$ and $K$.
\end{proof}

If a flat FAL (whose complement has  a unique reflection surface) has a single knot circle, as is the case for FALs of two-bridge links with an even number of twist regions (other than the Borromean rings, which has three reflection surfaces), then every crossing circle links the same knot circle twice. Thus any homeomorphism preserves the type of all components and their peripheral structures and can be realized by an isotopy of $\mathbb{S}^{3}$. This observation leads to the following immediate corollary of Lemma \ref{lem:SameKCPSameType}:

\begin{cor} \label{cor:TwoBridge}
A flat FAL with a single knot circle is determined by its complement among all flat FALs. In particular, flat FALs corresponding to $2$-bridge links with an even number of twists are determined by their complements. 
\end{cor}

\begin{proof}
The only flat FAL with one knot circle and multiple reflection surfaces is the Borromean rings, which we've already seen to be determined by its complement.  The unique reflection surface case follows from Lemma \ref{lem:SameKCPSameType}.
\end{proof}

\begin{rmk}
We would like to emphasize the necessity of the unique reflection surface hypothesis in Lemma \ref{lem:SameKCPSameType}. For instance, as noted in Theorem \ref{thm:3rs}, the Borromean rings complement admits three distinct reflection surfaces and the flat FAL diagram for this link has two crossing circles and a single knot circle that links each crossing circle twice. However, there exists homeomorphisms of the Borromean rings complement  where both a crossing circle and the knot circle switch types. 

\end{rmk}

We now prove a technical lemma considering homeomorphisms that change a crossing disk to a longitudinal disk, or vice-versa. It turns out that such a homeomorphism $h$ implies the existence of a separating quadruple $Q$, and the action of $h$ on $Q$ can be made quite precise. 

\begin{lem}\label{lem:SepQuad}
Let $M$ and $M'$ be homeomorphic flat FAL complements with unique reflection surfaces, and $h:M\to M'$ a homeomorphism that changes the type of an $N$-disk $D\subset M$. Then  
\begin{enumerate}[label=\roman*.]
\item The disk $D$ is part of a separating quadruple $Q$ in $M$ whose image $Q'=h(Q)$ in $M'$ is also a separating quadruple,
\item The homeomorphism $h$ fixes the types of exactly one opposite knot- and crossing-circle pair $K_f,C_f$ in $Q$.
\item The longitudinal disk and exactly one crossing disk in $Q$ change type under $h$.  These disks share the crossing circle puncture $C_f$.
\end{enumerate}
\end{lem}

\begin{proof}
We are given that $M$ and $M'$ are homeomorphic flat FAL complements with unique reflection surfaces, and that $h:M\to M'$ is a homeomorphism changing the type of an $N$-disk $D$.  Since $M$ and $M'$ each have unique reflection surfaces $R$ and $R'$, we have $R' = h(R)$. Then the image of a non-reflection thrice-punctured sphere in $M$ is non-reflection in $M'$.  Further, $N$-disks and singly-separated disks are distinguished by the topological property of separating, which implies that $h$ maps $N$-disks to $N$-disks. 

Consider first the case where $D$ is a longitudinal disk whose image $D' = h(D)$ is a crossing disk.  Let $C_1, C_2, C_3$ denote the crossing circle punctures of $D$, and let $D_i$ be a choice of crossing disk bounded by $C_i$. To show $D$ is part of a separating quadruple we must show that the $D_i$ are punctured by distinct knot circle components.

Two of the crossing circle punctures of $D$, say $C_2,C_3$, change type because $D'$ is a crossing disk.  Thus Lemma \ref{lem:SameKCPSameType} implies $D_2, D_3$ are punctured by distinct knot circle punctures.  Now consider $D_1$, whose image $D_1'$ must be a crossing disk or longitudinal disk since $N$-disks map to $N$-disks. If $D_1'$ stays a crossing disk, then Theorem \ref{thm:SepPair} implies that $\{D_1', D'\}$ is a separating pair since both are crossing disks sharing the crossing circle puncture $C_1'$.  This cannot happen because the pair $\{D_1,D\}$ does not separate in $M$ (again by Theorem \ref{thm:SepPair}). Therefore, $D_1$ changes type and $D_1'$ must be a longitudinal disk.  Then, by Lemma \ref{lem:SameKCPSameType}, $D_1$ has distinct knot circle punctures and $Q=\{D,D_1,D_2,D_3\}$ is a separating quadruple in $M$.

For convenience, label the knot circle punctures $K_1,K_2,K_3$, where $K_i$ is the knot circle opposite $C_i$ in that it does not puncture $D_i$.

\begin{figure}[h]
\begin{center}
\includegraphics{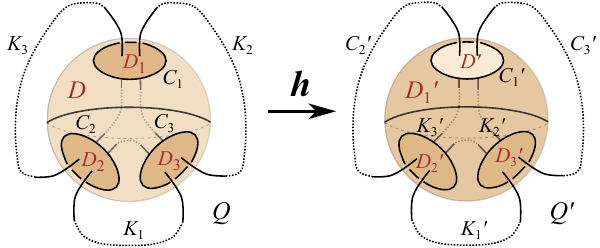}
\end{center}
\caption{Homeomorphic image of $Q$ when longitudinal disk changes type}
\label{fig:ImageOfQ}
\end{figure}

To see that $Q'=\{D',D_1',D_2',D_3'\}$ is also a separating quadruple, we must show that it consists of three crossing disks, with distinct knot circle punctures, and one longitudinal disk.  The disk $D'$ is assumed to be a crossing disk and, in this case, the disk $D_1'$ was shown to be longitudinal.  Moreover, since $C_2',C_3'$ are knot circles, the disks $D_2',D_3'$ must be crossing disks in $M'$ because these are the only non-reflection disks with knot circle punctures (Theorem \ref{prop:BeltSumSummary}).  Further, each of the disks in $Q$ are punctured by three distinct components, so their images are as well, and each crossing disk in $Q'$ is punctured by distinct knot circle components.  Thus $Q'$ is a separating quadruple in $M'$.

Note that this analysis proves the third conclusion as well, since $D, D_1$ change type while the crossing disks $D_2,D_3$ do not.  

To see statement $(ii)$ note that $C_1', K_2',K_3'$ are the crossing circle punctures of $Q'$ since they are the punctures of the longitudinal disk $D_1'$. The remaining punctures of $Q'$ must be knot circles in $M'$, so $K_1'$ is a knot circle.  Thus $h$ preserves the type of the opposite knot- and crossing-circle pair $C_1, K_1$, while changing the type of all other components.  Observe that $C_1$ is the crossing circle puncture shared by the disks that change type, namely $D$ and $D_1$.




Now suppose $D$ is a crossing disk in $M$ that changes type to a longitudinal disk $D'$ in $M'$.  Apply the previous argument to $h^{-1}$ and $D'$, then note that if $h^{-1}$ and $D'$ satisfy the conclusions of the lemma, then so does $h$ and $D$.
\end{proof}

Lemma \ref{lem:SepQuad} can be applied to the full-swap homeomorphisms discussed in Section \ref{Sec:SigLinks}, revealing some of the geometric structure inherent in such homeomorphisms.  Before proceeding we remark that any homeomorphism between flat FAL complements with unique reflection surfaces that does not preserve peripheral structures must change the type of a knot circle.  Indeed, if a crossing circle changes type, then one of the knot circles it links changes type as well since there are no (non-reflection) thrice-punctured spheres with three knot circle punctures.  Thus if $h$ changes the type of a component, there must be a knot circle that changes type. 


\begin{lem}\label{lem:UniqueCTC}
Let $h:M\to M'$ be a homeomorphism between flat FAL complements with unique reflection surfaces, and suppose $h$ changes the type of the knot circle $K$ in $M$.  Then $h$ changes the type of exactly one crossing circle $C_1$ linked by $K$ and, of all crossing disks punctured by $K$, $h$ fixes the type of exactly those bounded by $C_1$. 
\end{lem}

\begin{proof}
First we show $h$ changes the type of at most one crossing circle linking $K$.  Suppose, on the contrary, that $C_1, C_2$ are distinct crossing circles linking $K$ whose images $C_1',C_2'$ are both knot circles in $M'$. Let $D_1, D_2$ be a choice of crossing disks they bound and note that, since $C_1,C_2$ are distinct, the disks $D_1$ and $D_2$  do not form a separating pair by Theorem \ref{thm:SepPair}.  To determine if the image of an $N$-disk is a crossing disk, it is enough to show that it has a knot circle puncture because homeomorphisms between flat FAL complements with unique reflection surfaces map $N$-disks to $N$-disks. Since $D_i' = h(D_i)$ is punctured by the knot circle $C_i'$, the disk $D_i'$ must be a crossing disk in $M'$.   The disks $D_1',D_2'$ also share the crossing circle puncture $K'$ and so form a separating pair in $M'$, again by Theorem \ref{thm:SepPair}.  The homeomorphic image of a non-separating set, however, cannot be separating, and $h$ changes the type of at most one crossing circle linking $K$.

Now we argue that the image of at least one crossing circle linking $K$ is a knot circle in $M'$.  Since $K$ is linked by at least two crossing circles, at most one of which can change type, there is a crossing circle $C$ linking $K$ whose image $C'$ is a crossing circle in $M'$.  Let $D$ be a crossing disk bounded by $C$ and $J$ be the knot circle other than $K$ which punctures $D$. By Lemma \ref{lem:SameKCPSameType}, $J \neq K$.  Since $D'=h(D)$ has two crossing circle punctures in $K'$ and $C'$, we find $J'$ must also be a crossing circle. Hence $D'$ is a longitudinal disk in $M'$ (see Figure \ref{fig:Cross2Long}).  

\begin{figure}[h]
\begin{center}
\includegraphics{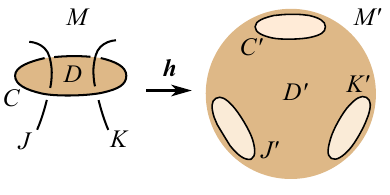}
\end{center}
\caption{Crossing disk $D$ maps to longitudinal disk $D'$}
\label{fig:Cross2Long}
\end{figure}

Thus $D$ is a crossing disk in $M$ that changes type, and Lemma \ref{lem:SepQuad} shows that $D$ is part of a separating quadruple $Q = \{D,D_1,D_2,D_3\}$.  Assume we've labeled disks so that $D_3$ is the longitudinal disk, and $J$ punctures $D_2$ while $K$ punctures $D_1$. Finally, let $C_1$ be crossing circle puncture of $D_1$ and $K_f$ the final knot circle in $Q$.  Figure \ref{fig:SchematicQ} depicts a ``schematic" diagram of this labeling in the sense that components of the FAL which are not in $Q$ are not pictured.

Now $h$ changes the crossing disk $D$ to a longitudinal disk, so Lemma \ref{lem:SepQuad}$(iii)$ implies the images of $D_1,D_2$ are again crossing disks. Since $D_1'$ is a crossing disk with crossing circle puncture $K'$, the crossing circle $C_1$ changes type under $h$. Thus at least one crossing circle linking $K$ maps to a knot circle.  Combining that with the first part of the proof shows that $h$ changes the type of exactly one crossing circle linking $K$.

\begin{figure}[h]
\begin{center}
\includegraphics{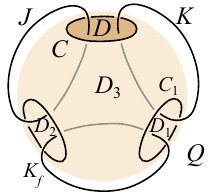}
\end{center}
\caption{Schematic of $Q$}
\label{fig:SchematicQ}
\end{figure}

To see the last statement, let $C_1$ be the unique crossing circle linking $K$ that changes type and suppose $D_1$ is any crossing disk punctured by $K$ which is bounded by $C_1$.  Since $C_1$ changes type, the image $D_1' = h(D_1)$ has a knot circle puncture in $M'$ and must be a crossing disk.

Conversely, suppose $D$ is a crossing disk punctured by $K$ and bounded by the crossing circle $C\ne C_1$.  Then two punctures of $D'$, both $C'$ and $K'$, are crossing circles in $M'$. Theorem \ref{prop:BeltSumSummary} implies the third puncture of $D'$  is also a crossing circle, and $D'$ is a longitudinal disk.  
\end{proof}

Lemmas \ref{lem:SepQuad} and \ref{lem:UniqueCTC} allow us to show, in the following lemma, that separating quadruples exist in the presence of type-changing homeomorphisms.

\begin{lem}\label{lem:CQC1}
Let $h:M\to M'$ be a homeomorphism between flat FAL complements with unique reflection surfaces, and suppose $K$ is a knot circle that changes type under $h$. Further, let $C_1$ be the unique crossing circle linking $K$ that changes type under $h$. Then each crossing circle $C\ne C_1$ that links $K$ is part of a separating quadruple that includes the punctures $K$, $C_1$ and $C$. 
\end{lem}

\begin{proof}
Since $C\ne C_1$, Lemma \ref{lem:UniqueCTC} implies that $h$ fixes the type of $C$ so $C'=h(C)$ is a crossing circle in $M'$.  The image $D'$ of a crossing disk $D$ bounded by $C$, then, is punctured by the crossing circles $K'$ and $C'$; therefore, $D'$ must be a longitudinal disk.  Thus $D$ is an $N$-disk that changes type, and Lemma \ref{lem:SepQuad} implies it is part of a separating quadruple $Q$. 

It remains to show that $C_1$ must be a crossing circle puncture in $Q$. First, since $K$ punctures $D$ it is a knot circle component of $Q$ and must puncture one other crossing disk, say $D_1$, in $Q$. Moreover, since exactly one crossing disk of $Q$ changes type, by Lemma \ref{lem:SepQuad}$(iii)$, $h$ fixes the type of $D_1$. Lemma \ref{lem:UniqueCTC} then implies that $D_1$ is bounded by $C_1$, finishing the proof.
\end{proof}

Lemmas \ref{lem:UniqueCTC} and \ref{lem:CQC1} demonstrate that flat FALs which admit a type-changing homeomorphism contain features similar to those of signature links. This motivates considering a certain sublink, the signature sublink, which we define in the next section.


\section{Complements determine flat FALs}\label{sec:CDFF}

In this section we show that flat FALs are determined by their complements.  We start by assuming that a flat FAL complement $M = \mathbb{S}^{3} \setminus \mathcal{A}$ has a unique reflection surface and admits a type-changing homeomorphism to another flat FAL complement, since all other cases are either trivial or covered by the work in Section \ref{sec:PropsRefSur}. With these assumptions, we first show in Subsection \ref{subsec:sigsub} that any such $\mathcal{A}$ contains a sublink, $\mathcal{L}_h$, that features many of the properties of a signature link. Then in Subsection \ref{subsec:StandardBall} we prove some technical tools involving separating quadruples, which are used to introduce an embedded two-sphere, $S^2_{\alpha}$, in $\mathbb{S}^{3}$ that intersects any such $\mathcal{A}$ in only two points on $\mathcal{K}_f$ and separates the other components of $\mathcal{A}$ in a useful manner. From here, our next goal is to show that $\mathcal{L}^{c} = \mathcal{A} \setminus \mathcal{L}_h$ must be empty and in fact, $\mathcal{A}$ must be a signature link. This is all done in Subsection \ref{subsec:hyp}. Essentially, if $\mathcal{L}^c$ is nonempty or  $\mathcal{A}$ fails to have any of the features necessary to be a signature link, then we can use our two-sphere $S^2_{\alpha}$ to show that $\mathcal{A}$ is either a connect-sum or a split link, contradicting hyperbolicity. Once we have proven that $\mathcal{A}$ must be a signature link, then we can quickly show that this type-changing homeomorphism is a full-swap, possibly pre- or post-composed with homeomorphisms that extend to isotopies of $\mathbb{S}^3$. At this point, we can use Proposition \ref{prop:MLswaps} to obtain the desired result.

\subsection{Signature Sublinks} \label{subsec:sigsub}

The forthcoming lemma highlight the necessary properties for us to define a signature sublink.

\begin{lem}\label{lem:SigSublink}
Let $M = \mathbb{S}^{3} \setminus \mathcal{A}$ be a flat FAL  complement with a unique reflection surface, and suppose there is a type-changing homeomorphism $h:M\to M'$ to another flat FAL complement $M'$.  Then $\mathcal{A}$ contains a sublink
\[
\mathcal{L}_h= \{K_f \} \cup \mathcal{K}\cup\mathcal{C}\cup\mathcal{C_K},
\]
whose components satisfy the following properties:
\begin{enumerate}[label=\roman*.]
\item The sets $\mathcal{K}$ and $\mathcal{C}$ contain $n\ge 2$ knot- and crossing-circles, respectively.
\item Each crossing circle of $\mathcal{C}=\{C_1,\dots,C_n\}$ links the corresponding knot circle of $\mathcal{K}=\{K_1,\dots,K_n\}$ and the knot circle $K_f$.  Further, the homeomorphism $h$ changes the types of the components in $\mathcal{C\cup K}$ and fixes the type of $K_f$. 
\item Let $\mathcal{C_K}$ denote the set of all crossing circles of $\mathcal{A}$ that link distinct knot circles in $\mathcal{K}$. Then $h$ fixes the type of each component in $\mathcal{C_K}$.  Moreover, each knot circle in  $\mathcal{K}$ is linked with at least one crossing circle in $\mathcal{C_K}$, and at most one crossing circle in $\mathcal{C_K}$ links two given knot circles in $\mathcal{K}$.
\end{enumerate}
\end{lem}

\begin{proof}
We begin by showing that $h$ changes the type of at least one knot circle in $\mathcal{A}$.  Since $M$ and $M'$ are homeomorphic and $M$ has a unique reflection surface $R$, Corollary \ref{thm:uniquereflectionhomeo} say that $R' = h(R)$, where $R'$ is the unique reflection surface for $M'$.  Suppose $h$ only changes the type of crossing circles. Let $C$ be a crossing circle that changes type and $D$ a crossing disk bounded by $C$.  Then $D' = h(D)$ would be a thrice-punctured sphere with three knot circle punctures, and Theorem \ref{prop:BeltSumSummary} implies that $D'$ must be part of the reflection surface $R'$.  This cannot occur since $D$ is a non-reflection thrice-punctured sphere in $M$ and $R' = h(R)$; therefore, $h$ changes the type of a knot circle, call it $K_1$.

Since $h$ changes $K_1$ to a crossing circle, Lemma \ref{lem:UniqueCTC} implies there is a unique crossing circle $C_1$ that links $K_1$ and changes type. Lemma \ref{lem:SameKCPSameType} implies that $C_1$ links distinct knot circles, $K_1$ and another which we denote $K_f$. The type of $K_f$ is fixed by $h$; otherwise, a crossing disk bounded by $C_1$ would map to a non-reflection thrice punctured sphere with exactly one knot circle puncture, contradicting Theorem \ref{prop:BeltSumSummary}.

The type-changing homeomorphism $h$, then, guarantees the existence of a knot circle $K_f$ whose type is fixed, together with a Hopf sublink $\{K_1,C_1\}$ whose types change and for which $K_f$ is linked by $C_1$.

Define $\mathcal{C}$ to be all crossing circles of $\mathcal{A}$ which link $K_f$ \emph{and} change type.  Note that $\mathcal{C}$ contains at least two crossing circles. Indeed, since $K_1$ and $C_1$ satisfy the hypotheses of Lemma \ref{lem:CQC1} we conclude that every crossing circle $C\ne C_1$ that links $K_1$ is part of a separating quadruple $Q_C$ that includes the components $K_1,C_1$, and $C$.  Moreover, since $C_1$ links $K_f$, the knot circle $K_f$ is a puncture of $Q_C$ as well.  Lemma \ref{lem:SepQuad}$(ii)$ shows that both crossing circles of $Q_C$ that link $K_f$ change type, so that both are in $\mathcal{C}$. Thus $\mathcal{C}$ contains at least two crossing circles of $\mathcal{A}$.

Now each $C_i\in\mathcal{C}$ changes type so links distinct knot circles (Lemma \ref{lem:SameKCPSameType}), $K_f$ and a second knot circle $K_i\in\mathcal{A}$.  Moreover, since $C_i$ changes type and the type of $K_f$ is fixed, the argument above for the existence of $K_1$ shows that $K_i$ must change type.  Finally, since each knot circle that changes type is linked by a unique crossing circle that changes type (Lemma \ref{lem:UniqueCTC}), the $K_i$ are distinct.  Let $\mathcal{K}=\{K_1,\dots,K_n\}$ denote the knot circles (other than $K_f$) linked by crossing circles in $\mathcal{C}$, and note that $h$ changes the type of each knot circle in $\mathcal{K}$.  

At this stage we have proven parts $(i)$ and $(ii)$ of the lemma. 

Now define $\mathcal{C_K}$ to be all crossing circles of $\mathcal{A}$ that link two knot circles of $\mathcal{K}$, and let $C_{ij}$ denote a crossing circle linking $K_i,K_j\in\mathcal{K}$.  To see that $h$ fixes the type of $C_{ij}$, note that a crossing disk $D_{ij}$ bounded by $C_{ij}$ is punctured by both $K_i$ and $K_j$.  Since $h$ changes $K_i$ and $K_j$ to crossing circles, the disk $h(D_{ij})$ must be a thrice-punctured sphere with at least two crossing circle punctures. Theorem \ref{prop:BeltSumSummary} implies $h(D_{ij})$ must be a longitudinal disk, so $h$ fixes the type of $C_{ij}$.

We now address the existence statements of part $(iii)$ of the lemma.  First fix a knot circle $K_i\in\mathcal{K}$, we must show there is at least one crossing circle in $\mathcal{C_K}$ linking $K_i$.  As in the above proof that $\mathcal{C}$ contains at least two crossing circles, Lemma \ref{lem:CQC1} applies to the Hopf sublink $\{K_i,C_i\}$.  Thus each crossing circle $C\ne C_i$ that links $K_i$ is part of a separating quadruple $Q_C$ that includes $K_f$ as a puncture.  Again as above, Lemma \ref{lem:SepQuad}$(ii)$ implies that $C$ links two knot circles of $\mathcal{K}$, so that for each $K_i$ there is at least one $C_{ij}\in\mathcal{C_K}$.  In this case let $Q_{ij}$ denote the separating quadruple $Q_C$, and note that the crossing circles of $Q_{ij}$ are $C_i,C_j,$ and $C_{ij}$. 

Now fix a pair of knot circles $K_i,K_j\in\mathcal{K}$, and suppose $C_{ij}\in\mathcal{C_K}$ exists.  We begin by showing $C_{ij}$ is a puncture in a separating quadruple $Q$ whose crossing circle punctures are $C_i, C_j,$ and $C_{ij}$.  Since $C_{ij}\ne C_i$ links $K_i$, Lemma \ref{lem:CQC1} implies it is part of a separating quadruple $Q$ containing the punctures $K_i,C_i,$ and $C_{ij}$.  The knot circles $K_i,K_j,$ and $K_f$ are the knot circle punctures of $Q$, since they are linked by $C_i$ and $C_{ij}$. The final crossing circle of $Q$ must link $K_f$ and $K_j$, and must change type since $C_{ij}$ is the only crossing circle of $Q$ whose type is fixed by $h$.  Therefore the final crossing circle is $C_j$, and $Q$ has crossing circle punctures $C_i, C_j,$ and $C_{ij}$.

Hence each crossing circle linking $K_i$ and $K_j$ forms a longitudinal disk with $C_i$ and $C_j$. Now Lemma 4.2 of \cite{mrstz} shows there is at most one longitudinal disk with two given crossing circle punctures, so there is at most one $C_{ij}$. \end{proof}

Note that the sublink $\mathcal{L}_h$ of Lemma \ref{lem:SigSublink} satisfies many of the properties of a signature link. In particular, $\mathcal{L}_h$ consists of knot and crossing circles partitioned into non-empty sets that satisfy all linking requirements of Definition \ref{defn:SigLink}. The only properties of a signature link not yet verified are that all knot circles in $\mathcal{K}$ lie on the same side of $K_f$ and that $\mathcal{L}_h$ is a flat FAL itself.  Proofs of these properties appear in Theorems \ref{thm:SameSide} and \ref{thm:LisLs}, respectively, but several preliminary results are required.  We remark that Lemma \ref{lem:SigSublink} highlights a further similarity: the homeomorphism $h$ changes types on components of $\mathcal{L}_h$ in precisely the same way that a full-swap does on a signature link. Finally, note that the sublink depends on the choice of a type-changing homeomorphism $h$ together with a knot circle $K_1$ that changes type. The subscript of $\mathcal{L}_h$ is intended to emphasize this dependency.

Lemma \ref{lem:SigSublink} motivates the following definition.

\begin{df}\label{defn:SigSublink} 
Let $M = \mathbb{S}^3\setminus\mathcal{A}$ be a flat FAL complement with a unique reflection surface, and suppose $M$ admits a type-changing homeomorphism $h$ that changes the type of the knot circle $K_1$ of $\mathcal{A}$. Let $K_f$, $\mathcal{C}$, $\mathcal{K}$, and $\mathcal{C_K}$ be as in Lemma \ref{lem:SigSublink}. 

The \textbf{signature sublink} $\mathcal{L}_h$ of $\mathcal{A}$ is the union of these components, so that
\[
\mathcal{L}_h= \{ K_f \} \cup \mathcal{K}\cup\mathcal{C}\cup\mathcal{C_K},
\]
endowed with a fixed choice of crossing disk $D_i$ for each $C_i\in \mathcal{C}$. Let $\mathcal{D}$ denote the set of chosen crossing disks $\{D_1,\dots,D_n\}$. 
\end{df}

Given a signature sublink there may be several choices for crossing disks bounded by crossing circles in $\mathcal{C}$. For example, the crossing circle $C_4$ of Figure \ref{fig:SigLink} bounds the crossing disk $D_4$ pictured as well as a crossing disk passing between $K_2$ and $K_3$. Fixing an orientation on $K_f$, we make the convention that the ordering on $\mathcal{C}$ and $\mathcal{D}$ is determined by traversing $K_f$ from $D_1$ in the chosen direction. Different choices for $\mathcal{D}=\{D_1,\dots,D_n\}$ can lead to different orderings, so for convenience we assume the choice is fixed throughout.

The next lemma shows that every crossing circle of $\mathcal{A\setminus C}$ that bounds a crossing disk punctured by a knot circle of $\mathcal{K}$ is an element of $\mathcal{C_K}$.  The proof will show both that the crossing circle links distinct knot circles (rather than the same knot circle twice), and that both knot circles are from $\mathcal{K}$.  The lemma also associates a unique separating quadruple $Q_{ij}$ to each $C_{ij}\in\mathcal{C}$, and characterizes possible intersections of these separating quadruples. The separating quadruples guaranteed by Lemma \ref{lem:C_K} will be the main tool used in what follows, so we let $\mathcal{Q} = \cup_{\mathcal{C_K}}Q_{ij}$ denote their union.

\begin{lem}\label{lem:C_K}
Let $M = \mathbb{S}^{3} \setminus \mathcal{A}$ be a flat FAL  complement with a unique reflection surface,  suppose there is a type-changing homeomorphism $h:M\to M'$ to another flat FAL complement $M'$,  and let $\mathcal{L}_h$ be the corresponding signature sublink of $\mathcal{A}$. If $C$ is a crossing circle of $\mathcal{A} \setminus \mathcal{C}$, with a crossing disk punctured by some $K_i\in\mathcal{K}$, then 
\begin{enumerate}[label=\roman*.]
\item There is an index $j$ with $C=C_{ij}\in\mathcal{C_K}$,
\item The crossing circle $C_{ij}$ bounds a unique crossing disk $D_{ij}$ in $M$, 
\item There is a separating quadruple $Q_{ij}$ uniquely determined by $C_i, C_j, C_{ij}$ and the chosen crossing disks in $\mathcal{D}$, and
\item Distinct separating quadruples $Q_{ij},Q_{kl}\in\mathcal{Q}$, are either disjoint or share a single crossing disk in $\mathcal{D}$.
\end{enumerate}
\end{lem}

\begin{proof}
Let $C$ be a crossing circle in $\mathcal{A}\setminus \mathcal{C}$ that bounds a crossing disk $D$ punctured by $K_i\in \mathcal{K}$. Since $K_i$ changes type, $C$ links distinct knot circles by Lemma \ref{lem:SameKCPSameType}, and we let $K^{\ast}$ denote the other knot circle puncturing $D$.  Further, again since $K_i$ changes type, $C_i$ is the unique crossing circle linking it that changes type by Lemma \ref{lem:UniqueCTC}.  Now $C$ is not equal to $C_i$, so there is a separating quadruple $Q$ with components $K_i,C_i$, and $C$ by Lemma \ref{lem:CQC1}.  The crossing circles $C_i$ and $C$ link all three knot circle components of $Q$, so $K^{\ast}$ and $K_f$ are the other knot circles of $Q$ and must be linked by the final crossing circle, say $C^{\ast}$, of $Q$.  Note that $C$ and $K_f$ are the only components of $Q$ whose type is fixed by the homeomorphism $h$ (by Lemma \ref{lem:SepQuad}$(ii)$), so $h$ changes the types of $C^{\ast}$ and $K^{\ast}$.  Thus $C^{\ast}$ is a crossing circle linking $K_f$ which changes type under $h$, implying there is an index $j$ with $C^{\ast}=C_j\in\mathcal{C}$, and $K^{\ast} = K_j$ as well. The original crossing circle $C$ is then $C_{ij}\in\mathcal{C_K}$, verifying part $(i)$ of the lemma.

Now consider statement $(ii)$ of the lemma. Suppose $C_{ij}$ bounded two distinct crossing disks, the original disk $D$ and another $D_1$. Then $D$ and $D_1$ form a separating pair by Theorem \ref{thm:SepPair}, and every knot circle of $\mathcal{A}$ intersects $D\cup D_1$ an even number of times (possibly zero). This implies both are punctured by $K_i,K_j\in\mathcal{K}$, both of which change to crossing circles under $h$.  Since $C_{ij}$ doesn't change type, the homeomorphism $h$ maps $D$ and $D_1$ to distinct longitudinal disks with the same punctures.  This is impossible, however, because Lemma 4.2 of \cite{mrstz} shows that there is at most one longitudinal disk with two given punctures, let alone three. Hence $C_{ij}$ bounds a unique crossing disk.

The existence portion of statement $(iii)$ is guaranteed by Lemma \ref{lem:CQC1}, as noted above.  To see uniqueness, note that two separating quadruples with the same crossing circle components have the same longitudinal disks by Lemma 4.2 of \cite{mrstz}. Thus two separating quadruples with the same crossing circle components can differ only in their crossing disks.  Since $C_{ij}$ bounds a unique crossing disk and, by convention, there is a fixed choice of disks $\mathcal{D}$ for crossing circles of $\mathcal{C}$, the $Q_{ij}$ are unique.

Finally, we prove statement $(iv)$ of the lemma. If $Q_{ij}$ and $Q_{kl}$ are disjoint we are done, so suppose their intersection is non-empty.  All $N$-disks are identical or disjoint (Lemma \ref{lem:NdisksDisjoint}), so if $Q_{ij}$ intersects $Q_{kl}$ nontrivially, they share some subset of disks.  We will show that if the intersection is other than a single crossing disk in $\mathcal{D}$, the separating quadruples are equal.

An initial observation is that if $\{D_i,D_j\}\subset Q_{ij}\cap Q_{kl}$ then $Q_{ij} = Q_{kl}$.  In this case both $Q_{ij}$ and $Q_{kl}$ contain crossing circles $C_i$ and $C_j$.  Lemma \ref{lem:SigSublink}$(iii)$ shows they also contain the unique crossing circle $C_{ij}\in\mathcal{C_K}$, and the proof of statement $(iii)$ above shows that the separating quadruples are equal. 

Now if $Q_{ij}\cap Q_{kl}$ contains a longitudinal disk, they are equal by the above argument since they would share the punctures $C_i$ and $C_j$.  Similarly, $Q_{ij}$ and $Q_{kl}$ are equal if they share $D_{ij}$ since this implies they share $C_{ij}$, and the definition of $Q_{ij}$ implies they share $C_i$ and $C_j$ as well.  

Thus the intersection $Q_{ij}\cap Q_{kl}$ of distinct separating quadruples from $\mathcal{Q}$ is either empty or a single disk from $\mathcal{D}$.
\end{proof}

One way to phrase Lemma \ref{lem:C_K}$(i)$ is to say that the signature sublink $\mathcal{L}_h$ contains all crossing circles of the FAL $\mathcal{A}$ that link a knot circle of $\mathcal{K}$. Do note that since each knot circle in $\mathcal{K}$ changes type, Lemma \ref{lem:SameKCPSameType} implies that if $K_i$ punctures a crossing disk bounded by $C\in\mathcal{A}$ then it does so once, and $C$ links distinct knot circles.  For convenience let $\mathcal{L}^c = \mathcal{A}\setminus\mathcal{L}_h$ denote the components of $\mathcal{A}$ not in $\mathcal{L}_h$. If $C$ is a crossing circle of $\mathcal{L}^c$, then, it  either links the fixed component $K_f$ or only knot circles of $\mathcal{L}^c$.

\subsection{The Standard Ball}\label{subsec:StandardBall}

Our next objective  is to associate a two-sphere $S^+_{\alpha}$ with every arc $\alpha$ of $K_f\setminus\mathcal{D}$, with the property that $S^+_{\alpha}\cap\mathcal{A}$ consists of two points on $K_f$. This is formally stated and proved at the end of this subsection in Lemma \ref{prop:2ps}.  To build this two-sphere, we first need to introduce a number of properties and terminology related to separating quadruples associated with a signature sublink $\mathcal{L}_h$ of $\mathcal{A}$. As noted in the introduction to Section \ref{sec:CDFF}, this two-sphere will play an essential role in classifying the flat FALs under consideration in this section.

Choose an orientation on $K_f$, say counterclockwise, and note that the crossing disks of $\mathcal{D}$ partition $K_f$ into $n$ oriented arcs.  Let $\alpha$ be an \textit{open arc} of $K_f\setminus \mathcal{D}$ and use the orientation on $K_f$ to order the disks of $\mathcal{D}$ (and associated components of $\mathcal{L}_h$) so that $\alpha$ goes from $D_n$ to $D_1$.  Index the components of $\mathcal{K}$ and $\mathcal{C}$ so that $C_i\in\mathcal{C}$ bounds $D_i$ and links $K_i\in\mathcal{K}$, and use this ordering to index crossing circles $\mathcal{C_K}$ and separating quadruples $\mathcal{Q}$.  Also number the arcs of $K_f\setminus \mathcal{D}$ so that $\alpha_i$ goes from $D_{i-1}$ to $D_i$, for $2 \le i \le n$ ($\alpha$ could be considered $\alpha_1$ in this ordering, but we continue to refer to it as $\alpha$ because of the special role it plays).  Thus $\alpha$ induces an ordering on components and disks of the signature sublink $\mathcal{L}_h$, other than $K_f$, which we call the $\alpha$\textbf{-ordering} of $\mathcal{L}_h$.

Let $S^2_{ij}$ denote the separating quadruple $Q_{ij}\in\mathcal{Q}$ thought of as a two-sphere embedded in $\mathbb{S}^3$.  Note that the crossing circles $\{C_i, C_j,C_{ij}\}$ and disks $\{D_i,D_j,D_{ij}, D_{ij}^{\ell}\}$ of $Q_{ij}$ are subsets of $S^2_{ij}$; whereas, the knot circle components $\{K_i, K_j, K_f\}$ each puncture $S^2_{ij}$ twice. Here, $D_{ij}^{\ell}$ is a longitudinal disk, with longitudinal punctures along the crossing circles $C_i$, $C_j$, and $C_{ij}$, as discussed at the beginning of Section \ref{subsec:SepSets}. Moreover, $\mathbb{S}^3 \setminus S^2_{ij}$ is two open three balls, one of which contains $\alpha$. Define the \textbf{inside} $\mathbb{I}_{ij}$ of $Q_{ij}$ (or of $S^2_{ij}$) to be the open three-ball component of $\mathbb{S}^3\setminus S^2_{ij}$ that does not contain $\alpha$ (thinking of $\alpha$ as \emph{outside} each $Q_{ij}$).  Observe that $S^2_{ij}$ is the boundary of $\mathbb{I}_{ij}$, so the closure is $\overline{\mathbb{I}_{ij}} = \mathbb{I}_{ij}\cup S^2_{ij}$.

We highlight the consequence of Lemma \ref{lem:C_K}$(iv)$ that insides are either disjoint or nested.

\begin{lem}\label{lem:InsidesQij}
Let $\mathcal{A}$ be a flat FAL whose complement admits a unique reflection surface and a type-changing homeomorphism $h$ to another flat FAL complement.  Let $\mathcal{L}_h$ be a signature sublink of $\mathcal{A}$ and $\mathcal{Q}$ be the set of all separating quadruples determined by $\mathcal{C_K}$.  Finally, let $\alpha$ be an open arc of $K_f\setminus\mathcal{D}$ inducing insides for each separating quadruple of $\mathcal{Q}$.

If $Q_{ij},Q_{kl}\in\mathcal{Q}$ are distinct separating quadruples, then their insides $\mathbb{I}_{ij}$ and $\mathbb{I}_{kl}$ are either disjoint or properly nested (so $\mathbb{I}_{ij}$ is a proper subset of $\mathbb{I}_{kl}$, or vice-versa).
\end{lem}

\begin{proof}
We let $\alpha_{ij}$ denote the arc of $K_f$ that is inside $Q_{ij}$, so that $\alpha_{ij}$ runs from $D_i$ to $D_j$ and is disjoint from $\alpha$. Note that if $\alpha_{ij}$ and $\alpha_{kl}$ overlap, but are not nested, then the disks $D_i,D_j$ alternate with $D_k,D_l$ around $K_f$.  This implies the separating quadruples $Q_{ij}$ and $Q_{kl}$ intersect but not along a crossing disk in $\mathcal{D}$, contradicting Lemma \ref{lem:C_K}$(iv)$. Thus if the arcs $\alpha_{ij}$ and $\alpha_{kl}$ overlap then they are nested, and we conclude the insides $\mathbb{I}_{ij}$ and $\mathbb{I}_{kl}$ are either nested or disjoint.  
\end{proof}

We now use the insides of separating quadruples and set inclusion to define a partial ordering on the set $\mathcal{Q}$.

\begin{df}\label{defn:PO}
Let $\alpha$ be an open arc of $K_f\setminus\mathcal{D}$ inducing insides on elements of $\mathcal{Q}$. The separating quadruple $Q_{ij}$ is \textbf{inside} $Q_{kl}$, denoted $Q_{ij}\prec_{\alpha} Q_{ij}$, if $\mathbb{I}_{ij}$ is a proper subset of $\mathbb{I}_{kl}$.
\end{df}

Observe that the inside relation, being defined using set inclusion, is indeed a strict partial ordering on $\mathcal{Q}$.  First, the subset relation is transitive so the inside relation is as well. Second, the \emph{proper} subset restriction implies the inside relation is neither reflexive nor symmetric, so it is a strict partial ordering.  

The inside relation is most easily seen by choosing the midpoint of $\alpha$ to be infinity in $\mathbb{S}^3 = \mathbb{R}^3\cup \{\infty\}$ and viewing the link from infinity as in Figure \ref{fig:StandardBall}.  The separating quadruple $Q_{12}$ is not related to any other separating quadruple of $\mathcal{Q}$, while both relations $Q_{35}\prec_{\alpha} Q_{25}$ and $Q_{56}\prec_{\alpha} Q_{57}$ hold.

We now introduce more terminology that will be helpful in the ensuing discussion.

The term \textbf{inside} will frequently be used in the context of link components or thrice-punctured spheres to imply containment within a separating quadruple.  For example, the disk $D_3$ of Figure \ref{fig:StandardBall} is inside the separating quadruple $Q_{25}$ since $D_3\subset\mathbb{I}_{25}$, as are the open arcs $\alpha_3, \alpha_4, \alpha_5$ of $K_f$.  More generally, recall that two separating quadruples $Q_{ij}$ and $Q_{kl}$ intersect in at most one disk of $\mathcal{D}$ (together with its punctures). This implies that if $Q_{ij}\prec_{\alpha} Q_{kl}$, then the components $C_{ij},D_{ij}$, and $D_{ij}^{\ell}$ of $Q_{ij}$ are inside $Q_{kl}$ as well.

A component, disk, or arc is \textbf{outside} the separating quadruple $Q_{ij}$ if it is in the open three-ball of $\mathbb{S}^3\setminus Q_{ij}$ containing $\alpha$.  Thus the disks $D_1$ and $D_{12}$ of Figure \ref{fig:StandardBall} are outside $Q_{25}$.  More generally, if disjoint separating quadruples are not comparable in the inside partial ordering, then they are outside each other.

We will also have occasion to describe separating quadruples as being on the same or opposite sides of a common disk in $\mathcal{D}$.  Suppose two separating quadruples $Q,Q'\in\mathcal{Q}$ share a common disk $D\in\mathcal{D}$.   Then the closed 3-balls $\overline{\mathbb{I}}$ and $\overline{\mathbb{I}'}$ share a common boundary disk $D$ and have interiors that are either nested or disjoint. Define $Q$ and $Q'$ to be on the \textbf{same side} of the crossing disk $D \in \mathcal{D}$ if their insides are nested, and on \textbf{opposite sides} if they are disjoint.  For example, the separating quadruples $Q_{56}$ and $Q_{57}$ in Figure \ref{fig:StandardBall} are on the same side of $D_5$ while $Q_{25}$ and $Q_{57}$ are on opposite sides.

Recall that an element $Q\in\mathcal{Q}$ is \textbf{maximal} if $Q_{ij}\prec_{\alpha}Q$ whenever $Q_{ij}$ and $Q$ are comparable. Thus the separating quadruples $Q_{12}$, $Q_{25}$, and $Q_{57}$ of Figure \ref{fig:StandardBall} are maximal elements.  The term \textbf{outermost quadruple} will be used for maximal elements in the inside partial ordering on $\mathcal{Q}$, as it is more intuitive.  

\begin{figure}[h]
\begin{center}
\includegraphics{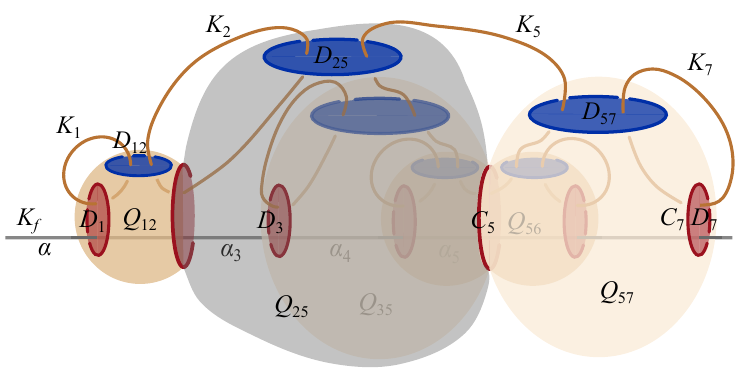}
\end{center}
\caption{A signature sublink with three outermost quadruples.}
\label{fig:StandardBall}
\end{figure}

We now highlight some important properties of this partial ordering.

\begin{lem}\label{lem:POprops}
Let $\mathcal{A}$ be a flat FAL whose complement admits a unique reflection surface and a type-changing homeomorphism $h$ to another flat FAL complement.  Let $\mathcal{L}_h$ be a signature sublink, $\alpha$ an open arc of $K_f\setminus\mathcal{D}$, and let $\prec_{\alpha}$ denote the induced inside relation on $\mathcal{Q}$.  The inside relation is a strict partial ordering on $\mathcal{Q}$ with the following properties:
\begin{enumerate}[label=\roman*.]
\item\label{lem:DistinctComp} Distinct separating quadruples of $\mathcal{Q}$ are not comparable if and only if their insides are disjoint.
\item\label{lem:SameSideD} Two separating quadruples that share a crossing disk in $D\in\mathcal{D}$ are comparable if and only if they are on the same side of $D$.
\item\label{lem:UniqueMax} Each $Q_{ij}\in \mathcal{Q}$ is either outermost or contained in a unique outermost element of $\mathcal{Q}$.
\item\label{D1Dn} The disk $D_1\in \mathcal{D}$ is part of a unique outermost quadruple, and is not inside any element of $\mathcal{Q}$. The same result is true of the disk $D_n\in\mathcal{D}$.
\end{enumerate}
\end{lem}

\begin{proof}
Let $Q$ and $Q'$ be distinct separating quadruples in $\mathcal{Q}$.  When their insides are properly nested, $Q$ and $Q'$ are comparable; however, if $Q$ and $Q'$ have disjoint insides they are not comparable.  Lemma \ref{lem:InsidesQij} shows these are the only two cases, proving statement $(i)$. To see that statement (ii) holds, suppose $Q, Q'\in \mathcal{Q}$ are distinct separating quadruples that share the crossing disk $D\in \mathcal{D}$.  By statement $(i)$ they are not comparable if and only if their insides are disjoint which, by definition, is equivalent to saying they are on opposite sides of $D$.

Now, we consider statement (iii). Suppose $Q\in\mathcal{Q}$ is not outermost, so that there is a $Q'\in\mathcal{Q}$ with $Q\prec_{\alpha} Q'$. If $Q''$ is another separating quadruple with $Q\prec_{\alpha} Q''$, then the insides of $Q'$ and $Q''$ intersect nontrivially, and statement $(i)$ implies they are comparable.  Thus the set of all separating quadruples larger than $Q$ is a finite linearly ordered subset, and so contains a unique maximal element. 

The argument for statement $(iv)$ follows from the facts that every disk in $\mathcal{D}$ is either inside or on an outermost element of $\mathcal{Q}$, and that disks adjacent to $\alpha$ are not inside any element of $\mathcal{Q}$.  To see the first fact, note that the knot circle $K_i$ punctures $D_i$, and Lemma \ref{lem:SigSublink}$(iii)$ proves that $K_i$ is linked by some crossing circle $C_{il}\in \mathcal{C_K}$.  The crossing circle $C_{il}$ generates a separating quadruple $Q_{il}\in \mathcal{Q}$ which contains the disk $D_i$, by Lemma \ref{lem:C_K}.  We just verified that $Q_{il}$ is either outermost or inside an outermost $Q\in\mathcal{Q}$. If $Q_{il}$ is outermost then $D_i$ is part of an outermost quadruple; otherwise, $D_i$ is either on or inside $Q$. Thus every $D_i\in \mathcal{D}$ is either inside or on an outermost element of $\mathcal{Q}$.  Now suppose $D_i\subset \mathbb{I}_{jk}$ for some disk $D_i\in \mathcal{D}$ and where $\mathbb{I}_{jk}$ is the inside of an outermost quadruple $Q_{jk}\in \mathcal{Q}$.  Then both arcs of $K_f$ adjacent to $D_i$ are inside $Q_{jk}$ as well.  
By definition of inside, however, the arc $\alpha\subset K_f$ is outside every element of $\mathcal{Q}$ so disks adjacent to $\alpha$ are not inside any element of $\mathcal{Q}$.  Since $\alpha$ is adjacent to $D_1$ and $D_n$, they cannot be inside any separating quadruple and must be part of an outermost separating quadruple. To see uniqueness, note that all separating quadruples containing $D_1$ must be on the side of $D_1$ opposite $\alpha$.  Hence every pair of separating quadruples containing $D_1$ are comparable, making all such separating quadruples a linearly ordered subset, which must have a unique maximal element.  The same observations hold for the disk $D_n$.
\end{proof}

Consider the set of all outermost separating quadruples, which we denote $\{Q_1,Q_2,\dots,Q_l\}$.  No two outermost separating quadruples are comparable, so their insides $\{\mathbb{I}_i\}$ are disjoint by Lemma \ref{lem:POprops}(\ref{lem:DistinctComp}). The open arcs $\beta_i = \mathbb{I}_i\cap K_f$, therefore, are disjoint as well. By definition of inside, the arc $\alpha$ is disjoint from all $\{\mathbb{I}_i\}$, so the arcs $\{\beta_i\}$ can be ordered as they are encountered starting at $\alpha$ and traversing $K_f$ according to its orientation.  We assume the sequence $\{Q_1,Q_2,\dots,Q_l\}$ is listed using this order on the $\{\beta_i\}$.  

As an example, note that for a given $\alpha$, there is a unique outermost quadruple $\{Q_1\}$ if and only if $Q_1 = Q_{1n}$.  An alternative characterization is that there is a crossing circle $C_{1n}\in\mathcal{C_K}$ linking the first and last knot circle in the $\alpha$-ordering of $\mathcal{K}$.

If $Q_i, Q_j$ are not consecutive in the ordered sequence $\{Q_1,Q_2,\dots,Q_l\}$, there is some $Q_k$ between them along $K_f$ and they cannot share a disk of $\mathcal{D}$.  Hence, non-consecutive separating quadruples in the sequence $\{Q_1,Q_2,\dots,Q_l\}$ are disjoint.
Consecutive outermost quadruples $Q_i, Q_{i+1}$, on the other hand, can either share a disk or be disjoint.  They are disjoint if there is an arc of $K_f$ between $\beta_i$ and $\beta_{i+1}$. We say that consecutive quadruples $Q_i, Q_{i+1}$ are \textbf{adjacent} if $Q_i\cap Q_{i+1}=D_{j_i}$ for some $D_{j_i}\in\mathcal{D}$.  Since the insides of $Q_i$ and $Q_{i+1}$ are disjoint, if they are adjacent they are on opposite sides of $D_{j_i}$ (Lemma \ref{lem:POprops}(\ref{lem:SameSideD})).  The subsequence $\{Q_i,\dots,Q_j\}$ of $\{Q_1,Q_2,\dots,Q_l\}$, is a \textbf{maximally adjacent subsequence} if consecutive quadruples in the subsequence are adjacent while the pairs $\{Q_{i-1},Q_i\}$ and $\{Q_{j},Q_{j+1}\}$ are disjoint. 
The ordered sequence of all outermost quadruples $\{Q_1,Q_2,\dots,Q_l\}$ partitions into maximally adjacent subsequences. For example, the  maximally adjacent subsequence of Figure \ref{fig:StandardBall} is $\{ Q_{12},Q_{25},Q_{57}\}$.

Let $\{Q_1,\dots,Q_m\}$ be the initial maximally adjacent subsequence, so that consecutive quadruples of the subsequence are adjacent, while $Q_m$ and $Q_{m+1}$ are not.  We will use the subsequence $\{Q_1,\dots,Q_m\}$ to define the standard ball associated with $\alpha$.

Some elementary observations are in order before we define the standard ball. Let $Q_j$ be a outermost separating quadruple with inside $\mathbb{I}_j$, and note that its closure $\overline{\mathbb{I}_j}$ is a closed three-ball with boundary sphere $Q_j$ (we abuse notation and use $Q_j$ to refer to the two-sphere embedded in $\mathbb{S}^{3}$ corresponding to this separating quadruple).  Now suppose $Q_{j-1}$ and $Q_{j}$ are adjacent, outermost separating quadruples which share the crossing disk $D_{l_j}\in\mathcal{D}$.  Since $Q_{j-1}$ and $Q_{j}$ are outermost, Lemma \ref{lem:POprops}(\ref{lem:DistinctComp}) implies their insides are disjoint. The union $\overline{\mathbb{I}_{j-1}} \cup \overline{\mathbb{I}_{j}}$ is a closed 3-ball, since it is two closed 3-balls (with disjoint interiors) glued along a common disk in their boundary spheres.  The open disk $D_{l_j}^{\circ}$ is interior to $\overline{\mathbb{I}_{j-1}} \cup \overline{\mathbb{I}_{j}}$, so the boundary is $\partial\left(\overline{\mathbb{I}_{j-1}} \cup \overline{\mathbb{I}_{j}}\right)=\left(Q_{j-1}\cup Q_{j}\right)\setminus D_{l_j}^{\circ}$.
These observations extend to subsequences $\{Q_i,\dots,Q_k\}$ of adjacent, outermost separating quadruples. For each outermost quadruple $Q_j$ in the subsequence, form the closed three-ball $\overline{\mathbb{I}_j} = \mathbb{I}_j\cup Q_j$.  Then $\cup_{j=i}^k\overline{\mathbb{I}_j}$ is a closed three-ball, because it is a sequence of closed three-balls with disjoint interiors in which only consecutive balls are glued together disk on their boundary spheres.  
Moreover, the boundary sphere of $\cup_{j=i}^k\overline{\mathbb{I}_j}$ is given explicitly by $\left(\cup_{j=i}^k{Q_j}\right)\setminus \left(\cup_{j=i+1}^{k}{D_{l_j}^{\circ}}\right)$, where $D_{l_j} = Q_{j-1}\cap Q_{j}$. Applying this to the initial maximally adjacent subsequence $\{Q_1,\dots,Q_m\}$ yields the standard ball associated with $\alpha$. 

\begin{df}\label{defn:StandardBall}
Let $\alpha$ be an arc of $K_f\setminus \mathcal{D}$ with initial maximally adjacent subsequence $\{Q_1,\dots,Q_m\}$ of outermost separating quadruples.  Let $\mathbb{I}_j$ be the inside  of $Q_j$, and $\overline{\mathbb{I}_j}$ its closure.  The \textbf{standard ball} of $\alpha$, denoted $\mathbb{B}^3_{\alpha}$, is the union 
\[
\mathbb{B}^3_{\alpha} = \cup_{j=1}^m\overline{\mathbb{I}_j},
\]
and let $S^2_{\alpha}=\left(\cup_{j=1}^m{Q_j}\right)\setminus \left(\cup_{j=2}^{m}D_{l_j}^{\circ}\right)$ denote the boundary sphere of $\mathbb{B}^3_{\alpha}$, where $D_{l_j}=Q_{j-1}\cap Q_{j}$.  
\end{df} 

In what follows, components of $\mathcal{A}$ that intersect $S^2_{\alpha}$ will be important so, for convenience, we introduce some notation. Each $Q_j\in\{Q_1,\dots,Q_m\}$ is some $Q_{l_jl_{j+1}}\in \mathcal{Q}$, and it is with this notation we defined $D_{l_j}=Q_{j-1}\cap Q_{j}$, for $2 \le j \le m$.  Rather than using double-subscripts, we adopt lower-case letters to represent components of the $Q_j$. For example, denote the knot circles $K_{l_j}, K_{l_{j+1}}$ of $Q_j = Q_{l_jl_{j+1}}$ by $k_j$ and $k_{j+1}$ respectively, let $c_{j,j+1}= C_{l_jl_{j+1}}$, while $d_{j,j+1}^{\ell}$ denotes the longitudinal disk $D_{l_jl_{j+1}}^{\ell}$.  We continue to use $K_f$ for the fixed knot circle puncturing $Q_j$.  Please refer to Figure \ref{fig:StandardBall2} for an illustration of this notation.

Now consider the boundary $S^2_{\alpha}$ of the standard ball with this notation. The crossing disk shared by the outermost quadruples $Q_{j-1}$ and $Q_j$ is $d_j=D_{l_j}$, and $k_j$ is the knot circle puncturing them. Further, by Lemma \ref{lem:POprops}(\textit{\ref{D1Dn}}), the disk $D_1=d_1$ lies on the separating quadruple $Q_1$, so that $k_1 = K_1$. This implies $S^2_{\alpha}=\left(\cup_{j=1}^m{Q_j}\right)\setminus \left(\cup_{j=2}^{m}{d_{j}^{\circ}}\right)$, and $S^2_{\alpha}$ is punctured twice by each of the knot circles $K_f,k_1,k_2,\dots,k_m,k_{m+1}$.

The boundary $S^2_{\alpha}$ of the standard ball, then, is a two-sphere embedded in $\mathbb{S}^3$ that is intersected by the link $\mathcal{A}$ in many components of $\mathcal{L}_h$.  Our immediate goal is to extend $S^2_{\alpha}$ to an embedded two-sphere $S^+_{\alpha}$ that intersects $\mathcal{A}$ in exactly two points of $K_f$.  Since $\mathcal{A}$ is hyperbolic, the desired $S^+_{\alpha}$ cannot define a connect-sum decomposition, and one component of $\mathbb{S}^3\setminus S^+_{\alpha}$ must be a standard ball-arc pair. This significantly restricts the link $\mathcal{A}$ and allows us to prove that $\mathcal{A} = \mathcal{L}_h$ (Theorem \ref{thm:LisLs}), which ultimately leads to our main result. 

Some preliminary definitions, and a technical lemma, are necessary before constructing the two-sphere $S^+_{\alpha}$.  Lemma \ref{lem:C_K} shows that the structure of a signature link outlined in Lemma \ref{lem:SigLong} (a longitudinal disk and separating quadruple) persists in signature sublinks, even in the (potential) presence of additional components.  In the proof of Lemma \ref{lem:SigLong} the inside $\mathcal{P}_i$ of the knot circle $K_i$ can be described as the component of the reflection surface bounded by $K_i$ and not containing $K_f$.  

Now let $k_i$ be a knot circle in $\mathcal{K}$ puncturing the standard ball. Analogously define the inside $\mathcal{P}_i$ of $k_i$ to be that component of the reflection surface not containing $K_f$, including the boundary curve $k_i$.  Each $k_i$ punctures the boundary $S^2_{\alpha}$ twice, so half of the closed disk $\mathcal{P}_{i}$ is inside and half outside of $\mathbb{B}^3_{\alpha}$.  We let $\mathcal{P}_{i}^{\ast}$ denote the portion of $\mathcal{P}_{i}$ outside of $\mathbb{B}^3_{\alpha}$.  Precisely we have $\mathcal{P}_{i}^{\ast} = \mathcal{P}_{i} \setminus\mathbb{B}^3_{\alpha}$.  Note $\mathcal{P}_{i}^{\ast}$ is a half-open disk with the arc of $k_i$ outside $\mathbb{B}^3_{\alpha}$ part of its boundary.  For $2 \le i \le m$, the interior of $\mathcal{P}_{i}^{\ast}$, then, is an open disk whose boundary consists of an arc of $k_i$, geodesics on each of $Q_{i-1}$ and $Q_i$, and single point of the crossing circle $c_i$ on $Q_{i-1} \cap Q_i$ (see Figure \ref{fig:StandardBall2}). Note that, for $i=1, m+1$,  $\mathcal{P}_{i}^{\ast}$ is also an open disk, but whose boundary only consists of an arc of $k_i$ and a geodesic on $Q_i$. 

\begin{figure}[h]
\begin{center}
\includegraphics{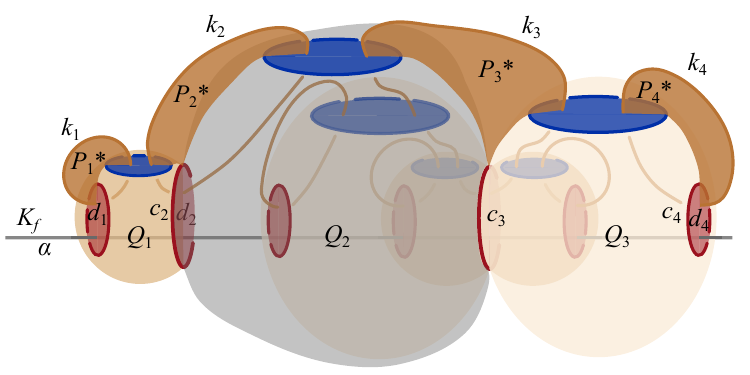}
\end{center}
\caption{The standard ball with disks $\mathcal{P}_i^{\ast}$.}
\label{fig:StandardBall2}
\end{figure}

Before constructing $S^{+}_{\alpha}$, we prove a technical lemma showing that no components of $\mathcal{A}$ intersect $\mathcal{P}_i^{\ast}$ (other than the boundary arc of $k_i$).  Components of $\mathcal{L}_h$ and its complement $\mathcal{L}^c = \mathcal{A}\setminus\mathcal{L}_h$ will be considered separately.

\begin{lem}\label{lem:Pstar}
Let $M = \mathbb{S}^{3} \setminus \mathcal{A}$ be a flat FAL complement with a unique reflection surface that admits a type-changing homeomorphism $h$ to another flat FAL complement. Let $\mathcal{L}_h$ denote the signature sublink of $\mathcal{A}$.  Let $\alpha$ be an arc of $K_f\setminus \mathcal{D}$ with standard ball $\mathbb{B}_{\alpha}^3$ generated by the initial maximally adjacent subsequence $\{Q_1,\dots,Q_m\}$ of outermost quadruples.  

If $\mathcal{P}_i^{\ast}$ is the portion of the reflection surface inside of $k_i\in\mathcal{K}$ but outside $\mathbb{B}_{\alpha}^3$, then the components of $\mathcal{A}$ are disjoint from the interior of $\mathcal{P}_i^{\ast}$. 
\end{lem}

\begin{proof}
Initially focus on components of $\mathcal{L}_h$ and consider $\mathcal{P}_i\cap \mathcal{L}_h$. First observe that $k_i$ is the only knot circle of $\mathcal{L}_h$ that is a boundary curve of $\mathcal{P}_i$.  To see this, suppose that $K$ is a knot circle of $\mathcal{A}$ interior to $\mathcal{P}_i$.  Then $K$ and $K_f$ are on opposite sides of $k_i$, and cannot be linked by a crossing circle.  Since each $K_j\in\mathcal{K}$ is linked to $K_f$ by $C_j$, we see $K\notin\mathcal{K}$ so $K$ is not in $\mathcal{L}_h$. Thus $\mathcal{P}_i^{\ast}\cap\mathcal{L}_h$ consists only of punctures by crossing circles of $\mathcal{L}_h$.

We turn our attention to crossing circles of $\mathcal{L}_h$ which intersect $\mathcal{P}_i^{\ast}$.  The crossing circle $C_i$ is the only one of $\mathcal{C}$ that intersects $\mathcal{P}_i$, and it does so in one point on the boundary of, not interior to, $\mathcal{P}_i^{\ast}$.  Now suppose $C\in\mathcal{C_K}$ is a crossing circle that links $k_i$, and let $Q$ be the separating quadruple it generates (guaranteed by Lemma \ref{lem:C_K}). In this case, $Q$ contains the crossing disk $D_i$ and is comparable to exactly those separating quadruples of $\mathcal{Q}$ on the same side of $D_i$ as $Q$ (Lemma \ref{lem:POprops}(\emph{\ref{lem:SameSideD}})). Consider the cases $i = 1$, $2\le i \le m$, and $i = m+1$ separately.

In the case $i=1$, the disk $d_1$ is $D_1\in\mathcal{D}$ by \ref{lem:POprops}(\emph{\ref{D1Dn}}), and every separating quadruple containing $D_1$ is opposite to $\alpha$.  Thus $Q$ is comparable to $Q_1$ and, by maximality of $Q_1$, we have $Q\prec_{\alpha}Q_1$.  This implies the crossing circle $C$ is inside $Q_1$ and disjoint from $\mathcal{P}_i^{\ast}$.

In the case $2\le i \le m$, the separating quadruple $Q$ shares the disk $d_i$ with both $Q_{i-1}$ and $Q_i$.  The outermost quadruples $Q_{i-1}$ and $Q_i$ are not comparable so must be on opposite sides of $d_i$.  Thus $Q$ is comparable to one of either $Q_{i-1}$ or $Q_i$, making $C$ interior to that outermost quadruple and disjoint from $\mathcal{P}_i^{\ast}$.

Finally consider $\mathcal{P}_{m+1}^{\ast}$ which is bounded by an arc of the knot circle $k_{{m+1}}$. In this case $Q$ shares the disk $d_{m+1}$ with $Q_m$ (e.g. disk $\mathcal{P}_4^{\ast}$ of Figure \ref{fig:StandardBall2}).  If $Q$ and $Q_m$ are on the opposite sides of $d_{m+1}$, then $Q$ is contained in a unique outermost quadruple $Q_{m+1}$ on the opposite side of $d_{m+1}$ from $Q_m$ (Lemma \ref{lem:POprops}). But then the sequence $\{Q_{1}, Q_{2}, \dots,Q_{m}\}$ can be extended by $Q_{m+1}$ to a longer sequence of adjacent, outermost separating quadruples.  This contradicts the definition of $\{Q_{1}, Q_{2}, \dots,Q_{m}\}$, so $Q$ and $Q_m$ are on the same side of $d_{m+1}$.  Maximality of $Q_m$ implies that $C$ is inside $Q_m$ and disjoint from $\mathcal{P}_i^{\ast}$.

In all cases, then, the interior of $\mathcal{P}_i^{\ast}$ is disjoint from crossing circles of $\mathcal{L}_h$.  The preceding argument showed that the same is true of knot circles in $\mathcal{L}_h$, so the interior of $\mathcal{P}_i^{\ast}$ is disjoint from the signature sublink $\mathcal{L}_h$. 

It remains to show that components of $\mathcal{L}^c = \mathcal{A}\setminus \mathcal{L}_h$ do not intersect the interior of $\mathcal{P}_{i}^{\ast}$. The proof amounts to showing that if the set of components of $\mathcal{L}^c$ intersecting $\mathcal{P}_{i}^{\ast}$ is non-trivial, then $\mathcal{A}$ is a split link, hence $\mathcal{L}^c$  must be disjoint from $\mathcal{P}_{i}^{\ast}$.

First recall that any crossing disk punctured by $k_i$ is bounded by a crossing circle in $\mathcal{L}_h$ by Lemma \ref{lem:C_K}, and consider a crossing circle $C\in \mathcal{L}^c$ that punctures the interior of $\mathcal{P}_{i}^{\ast}$.  Then, if $D$ a crossing disk bounded by $C$ it is disjoint from $k_i$ as well as from separating quadruples in $\mathcal{Q}$--in other words, $D$ is disjoint from the boundary of $\mathcal{P}_{i}^{\ast}$.  Now $C$ punctures $\mathcal{P}_{i}^{\ast}$, so $D$ intersects the reflection surface entirely within $\mathcal{P}_{i}^{\ast}$.  In particular, $C$ punctures the interior of $\mathcal{P}_{i}^{\ast}$ twice and any knot circle(s) linked by $C$ are interior to $\mathcal{P}_{i}^{\ast}$. 

Now let $K$ be a knot circle of $\mathcal{L}^c$ and recall that every crossing circle in $\mathcal{L}_h$ links only knot circles in $\mathcal{L}_h$. Thus all crossing circles of $\mathcal{A}$ that link $K$ are contained in $\mathcal{L}^c$.  Now suppose $K$ intersects $\mathcal{P}_{i}^{\ast}$. Knot circles of $\mathcal{A}$ are disjoint in the reflection surface and $K$, being in $\mathcal{L}^c$, is disjoint from $\mathcal{Q}$, so $K$ is contained in the interior of $\mathcal{P}_{i}^{\ast}$.  Then any crossing circle $C$ that links $K$ punctures $\mathcal{P}_{i}^{\ast}$, and the above argument shows that $C$ links only knot circles interior to $\mathcal{P}_{i}^{\ast}$.

Now suppose that $\mathcal{L}^c$ intersects $\mathcal{P}_{i}^{\ast}$ non-trivially, and let $\mathcal{L}_i^c$ denote the sublink of components of $\mathcal{A}$ that intersect $\mathcal{P}_{i}^{\ast}$. The above argument shows that crossing circles of $\mathcal{L}_i^c$ link only knot circles of $\mathcal{L}_i^c$, and vice versa.  Thus the components of flat FAL $\mathcal{A}$ partition into two non-empty subsets, $\mathcal{L}_i^c$ and its complement, in which crossing circles only link knot circles within their respective subset.  In an FAL this results in a split link, contradicting the fact that $\mathcal{A}$ is hyperbolic.  

Thus $\mathcal{P}_{i}^{\ast}$ is disjoint from $\mathcal{L}^c$, completing the proof that the components of $\mathcal{A}$ are disjoint from the interior of $\mathcal{P}_i^{\ast}$. 
\end{proof}

\begin{prop}\label{prop:2ps}
Let $\mathcal{A}$ be a flat FAL whose complement admits a type-changing homeomorphism $h$, with associated signature sublink $\mathcal{L}_h$. Let $K_f$ be the knot circle component of $\mathcal{L}_h$ whose type is fixed by $h$, and $\alpha$ be an arc of $K_f\setminus \mathcal{D}$. The standard ball $\mathbb{B}^3_{\alpha}$ has a neighborhood $N(\mathbb{B}^3_{\alpha})$ in $\mathbb{S}^3$ whose boundary is a two-sphere $S^+_{\alpha}$ such that $S^+_{\alpha}\cap \mathcal{A}$ is precisely two distinct points of $K_f$.  Moreover, $S^+_{\alpha}$ can be chosen so that  
\begin{enumerate}[label=\roman*.]
\item Every component, other than $K_f$, of $\mathcal{A}$ that intersects $\mathbb{B}^3_{\alpha}$ is inside $S^+_{\alpha}$, and
\item Every component of $\mathcal{A}$ that is outside $\mathbb{B}^3_{\alpha}$ is also outside $S^+_{\alpha}$.
\end{enumerate}
\end{prop}

\begin{proof}
Let $K_f$ be oriented with $\alpha$ an arc of $K_f$ between two consecutive disks of $\mathcal{D}$ and endow components of $\mathcal{L}_h$ with the $\alpha$-ordering. Further, let $\{Q_1,\dots,Q_m\}$ be the initial maximally adjacent subsequence and let $\mathbb{B}^3_{\alpha}$ be the standard ball of $\alpha$ with boundary sphere $S^2_{\alpha}$.  The proof consists of constructing a cell complex $X$ consisting of $\mathbb{B}^3_{\alpha}$ together with portions of the projection plane.  The desired sphere $S^+_{\alpha}$ will be the boundary of an appropriately chosen regular neighborhood $N(X)$ of this cell complex.  

Let $\mathcal{R}$ denotes the unique reflection surface of $M$. As in Lemma \ref{lem:Pstar}, let $\mathcal{P}_{i}$ denote the disk of $\mathcal{R}\setminus k_{i}$ that does not contain $K_f$, together with its boundary curve $k_{i}$.  Now define $X$ to be the cell complex
\[
X=\mathbb{B}^3_{\alpha} \cup \left(\cup_{1\le i \le m+1} \mathcal{P}_{i}\right).
\] 
Again following Lemma \ref{lem:Pstar}, the standard ball intersects each disk $\mathcal{P}_{i}$ and we let $\mathcal{P}_{i}^{\ast}$ denote the portion of $\mathcal{P}_{i}$ outside $\mathbb{B}^3_{\alpha}$ (i.e. $\mathcal{P}_{i}^{\ast} = \mathcal{P}_{i} \setminus\mathbb{B}^3_{\alpha}$--see Figure \ref{fig:StandardBall2}). 

The first statement of the proposition involves components of $\mathcal{A}$ that intersect $\mathbb{B}^3_{\alpha}$, and we show that each of these is contained in $X$.  Since the components of the quadruples $\{Q_1,\dots,Q_m\}$ are the only components of $\mathcal{A}$ that intersect the boundary $S^2_{\alpha}$ of the standard ball, we address those first.  All crossing circles in these quadruples lie on the boundary of the standard ball and are, therefore, subsets of $X$.  Recall that the knot circles of $\mathcal{A}$ that puncture $S^2_{\alpha}$ have been labeled $K_f,k_1,k_2,\dots,k_m,k_{m+1}$, which are not subsets of $\mathbb{B}^3_{\alpha}$.  Since $X$ includes the closed disks $\mathcal{P}_i$, however, each $k_i$ is a subset of $X$.  Thus all components of $\mathcal{A}$ that intersect  $S^2_{\alpha}$, other than $K_f$, are subsets of $X$. The remaining components of $\mathcal{A}$ that intersect $\mathbb{B}^3_{\alpha}$ are interior to $\mathbb{B}^3_{\alpha}$. Thus every component of $\mathcal{A}$ that intersects $\mathbb{B}^3_{\alpha}$ is contained in $X$ and, therefore, interior to every neighborhood of $X$ in $\mathbb{S}^3$.

The goal is to show that an appropriate open neighborhood $N(X)$ of $X$ is an open three ball whose boundary sphere satisfies the requirements for $S^+_{\alpha}$.   First observe that each $\mathcal{P}_{i}^{\ast}$ can be retracted into $\mathbb{B}^3_{\alpha}$ along the disk $\mathcal{P}_{i}$.  Since small enough neighborhoods of $\mathbb{B}^3_{\alpha}$ are three-balls, the same is true for $X$.  Moreover, since the knot circle $K_f$ punctures $S^2_{\alpha}$ twice, the same will be true of small enough neighborhoods of $X$.  From now on we assume $N(X)$ is an open three-ball neighborhood of $X$ with boundary sphere $S^2_X$ that is punctured twice by $K_f$.  The previous paragraph demonstrates that $S^2_X$ contains all other components of $\mathcal{A}$ that intersect $\mathbb{B}^{3}_{\alpha}$, so $S^2_X$ satisfies the first statement of the proposition as well.

We have shown that components of $\mathcal{A}$ that intersect $\mathbb{B}^3_{\alpha}$ behave as desired relative to $S^2_X$. Furthermore, the neighborhood $N(X)$ can be chosen small enough so that its boundary sphere doesn't intersect any components of $\mathcal{A}$ that are disjoint from $X$.  Thus any component of $\mathcal{A}$ that is disjoint from (or outside) $X$ is also outside $S^2_X$. 

To finish the proof of statement $(ii.)$, then, we must show that every component of $\mathcal{A}$ that is outside $\mathbb{B}^3_{\alpha}$ is also outside $X$. Or, contrapositively, show that every component of $\mathcal{A}$ that intersects $X\setminus\mathbb{B}^3_{\alpha} = \left(\cup_{1\le i \le m+1} \mathcal{P}_{i}^{\ast} \right)$ also intersects $\mathbb{B}^3_{\alpha}$.  Lemma \ref{lem:Pstar}, however, demonstrates that the only components of $\mathcal{A}$ intersecting the $\mathcal{P}_{i}^{\ast}$ are components of the outermost quadruples $\{Q_1,\dots,Q_m\}$, which also intersect $\mathbb{B}^3_{\alpha}$. Thus every component of $\mathcal{A}$ that is outside $\mathbb{B}^3_{\alpha}$ is also outside $X$, and $S^2_X$ satisfies the second statement of our proposition, making it our desired $S^2_+$.
\end{proof}

As an example, the sphere $S^+_{\alpha}$ of Figure \ref{fig:StandardBall2} contains the entire signature sublink $\mathcal{L}_h$ except the arc $\alpha$.  In fact, the twice-punctured sphere $S_{\alpha}^+$ of Lemma \ref{prop:2ps} will be used to provide a connect-sum decomposition of $\mathcal{A}$ if $\mathcal{A} \ne \mathcal{L}_h$, contradicting hyperbolicity.  Hence it is the key to finishing the proof that $\mathcal{A} = \mathcal{L}_h$.

\subsection{Hyperbolicity Conditions} \label{subsec:hyp}

We can now use the properties of $S^{+}_{\alpha}$ from Proposition \ref{prop:2ps} along with hyperbolicity conditions, specifically a hyperbolic link can not be a split link and can not be a connect-sum, to prove our main results. Before continuing, let us compare the \emph{signature sublinks} of Definition \ref{defn:SigSublink} to the \emph{signature links} of Definition \ref{defn:SigLink}.  One glaring difference is that a signature link is assumed to be a flat FAL while a signature sublink, being a sublink of a flat FAL, is not necessarily a flat FAL by itself.  The other difference is that all knot components of $\mathcal{K}$ in a signature link are assumed to be on the same side of $K_f$, and this has yet to be proven for signature sublinks.  By \textbf{same side of} $K_f$ we mean the same component of its complement in the projection plane.  Theorem \ref{thm:SameSide} shows that all knot circles in $\mathcal{K}$ for a signature sublink $\mathcal{L}_h$ are indeed on the same side of $K_f$.  Theorem \ref{thm:LisLs} proves that $\mathcal{L}_h$ is a flat FAL by proving (the stronger result) that it equals the flat FAL $\mathcal{A}$.  With these results in hand it is not hard to finish, in Theorem \ref{thm:LisL'}, the proof that flat FALs are determined by their complements.

\begin{thm}\label{thm:SameSide}
Let $M = \mathbb{S}^{3} \setminus \mathcal{A}$ and $M' = \mathbb{S}^{3} \setminus \mathcal{A}'$ be flat FAL complements, each with unique reflection surfaces. Suppose $h:M\to M'$ is a type-changing homeomorphism, and let $\mathcal{L}_h$ be the signature sublink of $\mathcal{A}$.  Then all knot circles in $\mathcal{K}$ are on the same side of $K_f$. 
\end{thm}

\begin{proof}
The result will follow once we show that if knot circles of $\mathcal{K}$ are on opposite sides of $K_f$, then there is an arc $\alpha$ in $K_f\setminus \mathcal{D}$ such that the two-sphere $S^+_{\alpha}$ of Proposition \ref{prop:2ps} provides a connect-sum decomposition of $\mathcal{A}$.  

A crossing circle in $\mathcal{C_K}$ links two knot circles of $\mathcal{K}$ that are on the same side of $K_f$, since a crossing circle of $\mathcal{A}$ that punctures opposite sides of $K_f$ necessarily links $K_f$.  Therefore, if $C_{ij},C_{kl}\in \mathcal{C_K}$ link knot circles on opposite sides of $K_f$, then the sets of crossing circles $\left\{C_i,C_j,C_{ij}\right\}$ and $\left\{C_k,C_l,C_{kl}\right\}$ are disjoint.  Since separating quadruples can intersect only along common crossing disks (Lemma \ref{lem:C_K}$(iv)$), this implies $Q_{ij}$ and $Q_{kl}$ are disjoint.  In particular, outermost separating quadruples that link knot circles of $\mathcal{K}$ on opposite sides of $K_f$ are not adjacent.

Now suppose there are knot circles in $\mathcal{K}$ that are on opposite sides of $K_f$, and label knot circles of $\mathcal{K}$ so that $K_n$ and $K_1$ are on opposite sides of $K_f$.  Let $\alpha$ denote the arc of $K_f$ between crossing disks $D_n$ and $D_1$, and let $\mathbb{B}^3_{\alpha}$ be the standard ball of $\alpha$.  Since $K_f$ intersects $D_1$ and $D_n$ at the endpoints of $\alpha$, they are contained in outermost separating quadruples $Q_{1i}$ and $Q_{jn}$ by Lemma \ref{lem:POprops}(\emph{\ref{D1Dn}}). The maximally adjacent subsequence $\{Q_1,\dots,Q_m\}$ contains only outermost quadruples on the same side of $K_f$ as $Q_1$, and create the standard ball $\mathbb{B}^3_{\alpha}$. Since the outermost quadruple $Q_{jn}$ is on the opposite side from those in $\{Q_1,\dots,Q_m\}$, its components are outside $\mathbb{B}^3_{\alpha}$.  There are components of $\mathcal{A}$, then, on both sides of the two-sphere $S^+_{\alpha}$ constructed in Proposition \ref{prop:2ps}.  But then $S^+_{\alpha}$ provides a non-trivial connect-sum decomposition of $\mathcal{A}$, contradicting the hyperbolicity of $\mathcal{A}$.
\end{proof}

\begin{thm}\label{thm:LisLs}
Let $M = \mathbb{S}^{3} \setminus \mathcal{A}$ and $M' = \mathbb{S}^{3} \setminus \mathcal{A}'$ be flat FAL complements each with unique reflection surfaces. Suppose $h:M\to M'$ is a type-changing homeomorphism.  Then $\mathcal{A}$ is a signature link. 
\end{thm}

\begin{proof}
To prove this result we show that $\mathcal{A}$ equals its signature sublink $\mathcal{L}_h$.  Once we've shown $\mathcal{L}_h=\mathcal{A}$, the signature sublink is a flat FAL and satisfies the final property of Definition \ref{defn:SigLink}.  We assume that $\mathcal{L}^c = \mathcal{A} \setminus \mathcal{L}_h$ is non-empty, and will arrive at a contradiction.

We begin by showing that if $C$ is a crossing circle of $\mathcal{L}^c$, then any crossing disk $D$ bounded by $C$ is punctured only by knot circles in $\mathcal{L}^c$.  First, since $C\in\mathcal{A}\setminus\mathcal{L}_h$, Lemma \ref{lem:C_K} implies that $D$ is not punctured by any $K_i\in \mathcal{K}$.

Now suppose that $D$ is punctured by $K_f$.  The crossing disks $\mathcal{D}$ cut $K_f$ into $n$ arcs. Let $\alpha$ be an arc of $K_f$ that punctures the disk $D$, then construct the standard ball $\mathbb{B}^3_{\alpha}$ of $\alpha$ and the associated twice-punctured sphere $S^+_{\alpha}$ of Proposition \ref{prop:2ps}. Now the crossing disk $D$ must be outside $\mathbb{B}^3_{\alpha}$ since it is punctured by $\alpha$ and is disjoint from $S^2_{\alpha} = \partial \mathbb{B}^3_{\alpha}$.  Then $C=\partial D$ is also outside $\mathbb{B}^3_{\alpha}$, and $C$ must be outside $S^+_{\alpha}$ by Proposition \ref{prop:2ps}. There are also components of $\mathcal{A}$ inside $S^+_{\alpha}$, by Lemma \ref{prop:2ps}, since $S^+_{\alpha}$ contains the components of the maximally adjacent subsequence $\{Q_1,\dots,Q_m\}$, which form a non-trivial subset of $\mathcal{A}$. Thus $S^+_{\alpha}$ provides a non-trivial connect-sum decomposition of $\mathcal{A}$. Since $\mathcal{A}$ is hyperbolic this is a contradiction, and every crossing circle $C\in\mathcal{L}^c$ links knot circles in $\mathcal{L}^c$.

Further recall that crossing circles of $\mathcal{L}_h$ link only knot circles in $\mathcal{L}_h$.  At this stage the components of $\mathcal{A}$ have been partitioned into two non-empty subsets $\mathcal{L}_h$ and $\mathcal{L}^c$ with the property that crossing circles from one subset bound crossing disks punctured only by knot circle(s) from the same set.  This is a contradiction, since such an FAL admits a disconnected diagram, indicating it came from a splittable link. Thus the link that generated $\mathcal{A}$ was split, contradicting the hyperbolicity of $\mathcal{A}$.

We conclude that $\mathcal{L}^c$ is empty, and $\mathcal{L}_h=\mathcal{A}$.
\end{proof}

The following theorem highlights our main result: flat FALs are determined by their complements. 

\begin{thm}\label{thm:LisL'}
Let $\mathcal{A}, \mathcal{A}'$ be flat FALs with homeomorphic complements. Then $\mathcal{A}$ is isotopic to $\mathcal{A}'$. 
\end{thm}

\begin{proof}
First, if $h:M\to M'$ is not type-changing, then it preserves peripheral systems. Under this assumption, $h$ extends to an isotopy of $\mathbb{S}^{3}$, making $\mathcal{A}$ and $\mathcal{A}'$ isotopic links. Moving forward, we will assume that $h$ is type-changing. Our proof breaks down into just two cases since two flat FAL complements each with a different number of reflection surfaces can not be homeomorphic by Corollary \ref{cor:MRSdetermined}.

\underline{Case I}: If $M$ contains multiple distinct reflection surfaces, then Corollary \ref{cor:MRSdetermined} tells us that $\mathcal{A}$ and $\mathcal{A'}$ are isotopic links. 

\underline{Case II}: Suppose $M$ and $M'$ each contain a unique reflection surface. Then by Theorem \ref{thm:LisLs} we know that $\mathcal{A}$ has the structure of a signature link $\mathcal{L}$.  Thus, there is a unique knot circle $K_f$ whose type is fixed by $h$, and $h$ changes the type of every crossing circle $C_i$ linking $K_f$.  Each $C_i$ links a knot circle $K_i$ that changes type and $h$ preserves the type of all other crossing circles, each of which link two of the $K_i$. The fact that $\mathcal{A}$ has all these properties is justified at the beginning of this section in Lemma \ref{lem:SigSublink}.

Since $\mathcal{A}$ is a signature link, we can consider the full-swap homeomorphism $h_{f}: M \rightarrow M''$, discussed in Section \ref{Sec:SigLinks}, where $M'' = \mathbb{S}^{3} \setminus \mathcal{A}''$ and $\mathcal{A}''$ is also a signature link. This homeomorphism has the same effect on  peripheral structures as $h$ and Proposition \ref{prop:MLswaps} tells us that $\mathcal{A}$ and $\mathcal{A}''$ are isotopic links. Now $h_{f}$ and $h$ act identically on peripheral systems, so $h_{f} \circ h^{-1}:M'\to M''$ preserves peripheral systems and extends to an isotopy of $\mathbb{S}^3$.  This implies that the links $\mathcal{A}'$ and $\mathcal{A}''$ are isotopic. Thus, $\mathcal{A}$ and $\mathcal{A}'$ are isotopic, finishing this case.
\end{proof}

By combining Theorem \ref{thm:SymmetryMRS} with the work from this section, we can now provide a complete proof of Theorem \ref{thm:SymmetryThm}, which we first restate.

\begin{named}{Theorem \ref{thm:SymmetryThm}}
Let $\mathcal{A}$ be a flat FAL.  Then either
\begin{itemize}
\item $\mathcal{A}$ is not a signature link and both $\mathcal{A}$ and its complement $M = \mathbb{S}^3\setminus \mathcal{A}$ have the same symmetry group, or
\item $\mathcal{A}$ is a signature link and full-swaps on $\mathcal{A}$ generate symmetries of $M = \mathbb{S}^3\setminus \mathcal{A}$ which are not restrictions of symmetries of $\mathcal{A}$ to $M$.
\end{itemize}
\end{named}

\begin{proof}
Let $\mathcal{A}$ be a flat FAL with $M = \mathbb{S}^{3} \setminus \mathcal{A}$. Note that, $P_3$ is a signature link (see Figure \ref{fig:TwoMLswaps}) whose complement contains multiple reflection surfaces.

First, suppose $A$ is not a signature link. If $M$ contains multiple reflection surfaces, then  Theorem \ref{thm:SymmetryMRS} tells us that $Sym(\mathbb{S}^{3}, \mathcal{A}) = Sym(\mathbb{S}^{3} \setminus \mathcal{A})$. If $M$ contains a unique reflection surface and $A$ is not a signature link, then Theorem \ref{thm:LisLs} implies every self-homeomorphism of $M$ is not type-changing. Thus, in this case, every self-homeomorphism of $M$ preserves peripheral structures, and so,  extends to an isotopy of $\mathbb{S}^{3}$, as needed. 

Now, suppose $A$ is a signature link. Then $M$ admits a full-swap homeomorphism $h_f$. As discussed in Section \ref{Sec:SigLinks}, full-swaps are compositions of ml-swaps, where these ml-swaps exchange meridional and longitudinal slopes on (distinct) Hopf sublinks of $\mathcal{A}$. Such homeomorphisms do not extend to isotopies of $\mathbb{S}^{3}$, and so, $h_f$ is a representative for an element of $Sym(M)$ that does not restrict to an element of $Sym(\mathbb{S}^{3}, \mathcal{A})$. 
\end{proof}




\begin{thebibliography}{1}


\bibitem{ad2}
C.C. Adams
\newblock \emph{Augmented alternating link complements are hyperbolic},
\newblock Low-dimensional topology and Kleinian groups (Coventry/Durham, 1984), London Math. Soc. Lecture Note Ser. 112, Cambridge University Press, Cambridge (1986), 115-130. 


\bibitem{ad1} 
C. C. Adams 
\newblock \emph{Thrice-punctured spheres in hyperbolic 3-manifolds}, 
\newblock Trans. A. M. S. 287 (1985), no. 2, 645-656. MR 768730 (86k:57008).







\bibitem{BePe1992}
R. Benedetti, C. Petronio
\newblock \emph{Lectures on hyperbolic geometry},
\newblock Universitext, Springer-Verlag, Berlin (1992).

\bibitem{Berge}
J. Berge
\newblock \emph{Embedding the exteriors of one-tunnel knots and links in the 3-sphere}
\newblock unpublished manuscript


\bibitem{BFT2015}
R. Blair, D. Futer, M. Tomova
\newblock \emph{Essential surfaces in highly twisted link complements},
\newblock Algebr. Geom. Topol. 15 (2015), no. 3,  1501--1523. 



\bibitem{CDW2012}
E. Chesebro, J. DeBlois, H. Wilton,
\newblock \emph{Some virtually special hyperbolic 3-manifold groups}, 	
\newblock Comment. Math. Helv. 87 (2012), no. 3, 727--787.

\bibitem{F2017}
R. Flint,
\newblock \emph{Intercusp Geodesics and Cusp Shapes of Fully Augmented Links}, 	
\newblock ProQuest LLC, Ann Arbor, MI, 2017, Thesis (Ph.D.) - City University of New York.

\bibitem{FP2007}
D. Futer and J. Purcell,
\newblock \emph{Links with no exceptional surgeries}, 	
\newblock Comment. Math. Helv. 82 (2007), no. 3, 629--664.

\bibitem{Ga2018}
F. Gainullin, 
\newblock \emph{Heegaard Floer homology and knots determined by their complements},
\newblock Algebr. Geom. Topol. 18 (2018), no. 1 69--109.

\bibitem{Go2002}
C. McA. Gordon,
\newblock \emph{Links and their complements},
\newblock Amer. Math. Soc., 314 (2002), 71--82. 

\bibitem{GL1989}
C. McA. Gordon and J. Luecke,
\newblock \emph{Knots are determined by their complements}, 	
\newblock J. Amer. Math. Soc. 2 (1989), no. 2, 371-415.


\bibitem{HeWe1992}
S.R. Henry and J.R. Weeks, 
\newblock \emph{Symmetry groups of hyperbolic knots and links},
\newblock Journal of Knot Theory and its Ramifications, 1 (1992), no. 2 185--201. 


\bibitem{HMW2020}
N. Hoffman, C. Millichap, W. Worden
\newblock \emph{Symmetries and Hidden Symmetries of $(\epsilon, d_L)$-Twisted Knot Complements}
\newblock Algebr. Geom. Topol. 22 (2022), no. 2 601--656.

\bibitem{IS2021}
K. Ichihara and T. Saito,
\newblock \emph{Knots in homology spaces determined by their complements},
\newblock  Bull. Korean Math. Soc. 59 (2022), no. 4, 869--877.

\bibitem{KT2020}
S. Knavel, R.Trapp
\newblock \emph{Embedded totally geodesic surfaces in fully augmented links}, 
\newblock Accepted for publication in Comm. in  Anal. and Geom. 

\bibitem{L2004}
M. Lackenby,
\newblock \emph{The volume of hyperbolic alternating link complements}, 	
\newblock Proc. London Math. Soc. (3) 88 (2004), no. 1, 204--224,
\newblock With an appendix by Ian Agol and Dylan Thurston.


\bibitem{MaSt2001} 
B. Mangum, T. Stanford,
\newblock {\em Brunnian links are determined by their complements},
\newblock Algebraic $\&$ Geometric Topology, 1 (2001), 143--152. 

\bibitem{Ma2010}
D. Matignon, 
\newblock {\em On the knot complement problem for non-hyperbolic knots},
\newblock Topology Appl., 157 (2010), no. 12, 1900-1925.


\bibitem{MMR2020}
J. Meyer, C., Millichap, R. Trapp,
\newblock \emph{Arithmeticity and hidden symmetries of fully augmented pretzel link complements},
\newblock New York Journal of Math., 26 (2020), 149--183. 

\bibitem{mrstz} 
P.~Morgan, L.~Mork, B.~Ransom, D.~Spyropoulos, R.~Trapp, C.~Ziegler, Belted-sum decompositions of fully augmented links. Preprint: ArXiv: 2311.13540.



\bibitem{P2011}
J. Purcell,
\newblock \emph{An introduction to fully augmented links}, 	
\newblock Interactions between hyperbolic geometry, quantum topology and number theory, 205--220, 
\newblock Contemp. Math., 541, Amer. Math. Soc., Providence, RI, 2011. 

\bibitem{P2008}
J. Purcell,
\newblock \emph{Cusp shapes under cone deformation},
\newblock Journal of Differential Geometry, { \bf 80} (2008), 453--500. 
 

\bibitem{P2007}
J. Purcell, 
\newblock \emph{Volumes of highly twisted knots and links},
\newblock Algebraic \& Geometric Topology { \bf 7} (2007), 93--108.


\bibitem{Ro1993}
Y.W. Rong,
\newblock \emph{Some knots not determined by their complements},
\newblock Quantum Topology, 339--353, Ser. Knots Everything, 3, World Sc. Publ., River Edge, NJ. (1993). 


\bibitem{Y2018}
K. Yoshida,
\newblock\emph{Unions of 3-punctured spheres in hyperbolic 3-manifolds},
\newblock Comm. in Anal. and Geo., vol. 29 (2021) 1643--1689. 

\bibitem{Wh1937}
J.H.C. Whitehead,
\newblock \emph{On doubled knots},
\newblock J. London. Math Soc., vol. 12 (1937) 63--71.

\bibitem{Z2021}
M. Zevenbergen,
\newblock\emph{Crushtaceans and complements of fully augmented and nested links},
\newblock Honors Thesis at the University of Rochester, (2021) https://www.sas.rochester.edu/mth/undergraduate/honors-papers-plans.html.

\end{thebibliography}
\end{document}